\documentclass[11pt,reqno]{amsart}
\usepackage[margin=1in]{geometry}
\usepackage{amsmath,amsthm,amssymb}
\usepackage{bbm}
\usepackage{mathrsfs} 
\usepackage{enumitem}
\usepackage{pgf,tikz,pgfplots}
\usetikzlibrary{arrows}
\usepackage[unicode]{hyperref}
\hypersetup{colorlinks=true, linkcolor=blue, citecolor=blue, filecolor=blue, urlcolor=blue}
\usepackage{comment}
\usepackage{chngcntr}
\usepackage{mathabx}
\usepackage{bm}
\usepackage{etoolbox}
\usepackage{enumitem}
\usepackage{subcaption}
\usepackage{floatrow}
\usepackage{float}
\usepackage{cleveref}
\newfloatcommand{capbtabbox}{table}[][\FBwidth] 

\usepackage{blindtext}
\usepackage[comma,sort&compress, numbers]{natbib}

\numberwithin{equation}{section}

\newcommand{\bA}{\bm{A}}
\newcommand{\bmB}{\bm{B}}
\newcommand{\bC}{\bm{C}}
\newcommand{\bD}{\bm{D}}
\newcommand{\bF}{\bm{F}}
\newcommand{\bI}{\bm{I}}
\newcommand{\bM}{\bm{M}}
\newcommand{\bS}{\bm{S}}
\newcommand{\bU}{\bm{U}}
\newcommand{\bV}{\bm{V}}
\newcommand{\bW}{\bm{W}}
\newcommand{\bZ}{\bm{Z}}

\newcommand{\bs}{\bm{s}}
\newcommand{\bu}{\bm{u}}
\newcommand{\bv}{\bm{v}}

\newcommand{\bPhi}{\bm{\Phi}}
\newcommand{\bLambda}{\bm{\Lambda}}
\newcommand{\blambda}{\bm{\lambda}}
\newcommand{\bone}{\bm{1}}
\newcommand{\bcA}{\bm{\mathcal{A}}}
\newcommand{\bcB}{\bm{\mathcal{B}}}
\newcommand{\bcC}{\bm{\mathcal{C}}}
\newcommand{\cA}{\mathcal{A}}
\newcommand{\cB}{\mathcal{B}}
\newcommand{\cF}{\mathcal{F}}
\newcommand{\cH}{\mathcal{H}}
\newcommand{\cN}{\mathcal{N}}
\newcommand{\E}{\mathbb{E}}
\newcommand{\N}{\mathbb{N}}
\renewcommand{\P}{\mathbb{P}}
\newcommand{\bR}{\mathbb{R}}
\newcommand{\sfH}{\mathsf{H}}
\newcommand{\sfK}{\mathsf{K}}
\newcommand{\sfKmat}{\bm{\mathsf{K}}}
\newcommand{\pto}{\overset{P}{\rightarrow}}
\newcommand{\dto}{\overset{D}{\rightarrow}}
\newcommand{\Var}{\mathrm{Var}}
\newcommand{\diag}{\mathrm{diag}}
\newcommand{\rd}{\mathrm{d}}
\newcommand{\perm}{\mathrm{perm}}
\newcommand{\one}{\mathbbm{1}}
\newcommand{\vep}{\varepsilon}
\newcommand{\ra}{\rightarrow}
\newcommand{\red}{\textcolor{black}}
\newcommand{\blue}{\textcolor{black}}
\newcommand{\cblue}{\color{black}}
\newcommand{\cblack}{\color{black}}
\newcommand{\cred}{\color{black}}

\newtheorem{theorem}{Theorem}[section]

\newtheorem{corollary}{Corollary}[section]

\newtheorem{lemma}{Lemma}[section] 
 
\newtheorem{prop}{Proposition}[section]

\newtheorem{assumption}{Assumption}[section]

\theoremstyle{definition} 
\newtheorem{definition}{Definition}[section]
\newtheorem{remark}{Remark}[section]

\allowdisplaybreaks

\begin{document}

\title[]{Fluctuation of the largest eigenvalue of a Kernel Matrix with application in Graphon-based random graphs}
\author[Chatterjee and Huang]{Anirban Chatterjee and Jiaoyang Huang} 
\address{Department of Statistics and Data Science\\ University of Pennsylvania\\ Philadelphia\\ PA 19104\\ United States}
\email{anirbanc@wharton.upenn.edu}
\address{Department of Statistics and Data Science\\ University of Pennsylvania\\ Philadelphia\\ PA 19104\\ United States}
\email{huangjy@wharton.upenn.edu}

\begin{abstract}
In this article, we explore the spectral properties of general random kernel matrices $\left[\sfK(U_i,U_j)\right]_{1\leq i\neq j\leq n}$ from a Lipschitz kernel $\sfK$ with $n$ independent random variables $U_1,U_2,\ldots, U_n$ distributed uniformly over $[0,1]$. In particular, we identify a dichotomy in the extreme eigenvalue of the kernel matrix, where, if the kernel $\sfK$ is degenerate, the largest eigenvalue of the kernel matrix (after proper normalization) converges weakly to a weighted sum of independent chi-squared random variables. In contrast, for non-degenerate kernels, it converges to a normal distribution extending and reinforcing earlier results from Koltchinskii and Gin\'e (2000). Further, we apply this result to show a dichotomy in the asymptotic behavior of extreme eigenvalues of Graphon-based random graphs, which are pivotal in modeling complex networks and analyzing large-scale graph behavior. These graphs are generated using a kernel $W$, termed as graphon, by connecting vertices $i$ and $j$ with probability $W(U_i, U_j)$. Our results show that for a Lipschitz graphon $W$, if the degree function is constant, the fluctuation of the largest eigenvalue (after proper normalization) converges to the weighted sum of independent chi-squared random variables and an independent Gaussian variable. Otherwise, it converges to a normal distribution.



\end{abstract}

\maketitle

{\hypersetup{linkcolor=black}
\setcounter{tocdepth}{1}
\tableofcontents
}

\section{Introduction}

In recent years, the study of graph theory has gained significant momentum, owing to its applicability in diverse fields ranging from biology and physics to social sciences and computer networks  \cite{bondy1976graph, mason2007graph, stam2007graph, borgatti2009network}. 
Many interesting properties of graphs are revealed by the extreme eigenvalues and eigenvectors
of their adjacency matrices. To mention some, we refer the readers to the books \cite{brouwer2011spectra,chung1997spectral} for a general discussion on spectral graph theory, the survey article \cite{hoory2006expander} for the connection between eigenvalues and expansion properties
of graphs, and the articles \cite{mohar1997some, mohar1993eigenvalues, pothen1990partitioning, shi2000normalized, spielman2012spectral, spielman1996spectral} on the applications of eigenvalues and eigenvectors in
various algorithms, i.e., combinatorial optimization, spectral partitioning and clustering. The Erd\H{o}s--R{\'e}nyi graphs and random $d$-regular graphs serve as the two prototypical models for random graphs,  and their extreme eigenvalues have been extensively studied \cite{10.1214/19-AOP1378,vu2005spectral,latala2018dimension, alt2021extremal, tikhomirov2021outliers, erdHos2013spectral, he2021fluctuations, huang2022edge, lee2021higher, huang2021spectrum, sarid2023spectral, 10.1214/18-AOP1263, he2022spectral, bauerschmidt2020edge, 10.1214/17-AOP1180, friedman2003proof, friedman2008proof, bordenave2015new}.

 In this paper, we study the extreme eigenvalues of random graphs generated from graphons, which are  generalizations of Erd\H{o}s--R{\'e}nyi graphs. Recall that a graphon, denoted as $W$, is a symmetric, measurable function that offers a powerful framework for understanding the limiting behavior of large graph sequences, see \cite{lovasz2006limits,lovasz2012large}. A graphon $W$ gives rise to a way of generating random graphs. This construction leads to the $W$-random graphs, which serve as a fundamental tool for modeling and analyzing the behavior of large-scale networks. These graphs are generated using a graphon $W$ by first creating an $n\times n$ random kernel matrix $\left[W(U_i, U_j)\right]_{1\leq i\neq j\leq n}$ (we do not allow self-loops) using $n$ independent numbers $U_1, U_2, \cdots, U_n$ uniformly distributed over $[0, 1]$. This random kernel matrix then gives rise to a random simple graph: connecting nodes $i$ and $j$ with probability $W(U_i, U_j)$. The focus of this paper is the spectral analysis of the random kernel matrices, and the $W$-random graphs.

Our first main result concerns about the extreme eigenvalues of random kernel matrices formed from a general integral kernel $\sfK$ (graphons are special examples). The spectral properties of such random kernel matrices have been studied in the pioneer work \cite{koltchinskii2000random}. It is proved that the $L_2$ distance between the ordered spectrum
of the random kernel matrices $\{\sfK(U_i, U_j)\}_{1\leq i\neq j\leq n}$ and the ordered spectrum of $\sfK$ tends to zero. Under certain technical conditions, distributional limit theorems for the eigenvalues of the random kernel matrices
are also obtained. However, the conditions in \cite{koltchinskii2000random} are not easy to check unless $\sfK$ is of finite rank, see \Cref{r:cond}. Moreover the distributional limit theorem in \cite{koltchinskii2000random} is trivial (the limit is a normal with variance $0$) when the kernel $\sfK$ is degenerate, namely the eigenfunction corresponding to the largest eigenvalue is a constant function.  Notice that this notion of degeneracy is related to the notion of degenerate kernels appearing in the study of $U$-statistics (see \cite{van2000asymptotic}).

We revisit the spectral problem of random kernel matrices $\{\sfK(U_i, U_j)\}_{1\leq i\neq j\leq n}$ and extend the distributional limit theorems in \cite{koltchinskii2000random} in two ways. First we identify a simple condition that as long as the kernel is Lipschitz (probably can be further relaxed to piecewise lipschitz), the largest eigenvalue converges to a normal random variable. Secondly, in the degenerate case, if we further rescale by a factor $\sqrt n$, the largest eigenvalue converges to a generalized chi-squared distribution. We obtain an explicit characterization of it in terms of the spectrum of $\sfK$. This leads to \Cref{t:main1} showing a dichotomy in the extreme eigenvalues of random kernel matrices coming from a Lipschitz kernel. Specifically, if the kernel is degenerate, the largest eigenvalue converges weakly to a weighted (possibly infinite) sum of independent chi-squared random variables. In contrast, for non-degenerate kernels, it converges to a normal distribution.

To study the spectra of random kernel matrices, we first derive a master equation \eqref{eq:1AB1=12}, which characterizes their largest eigenvalues. Such master equation has been used intensively in random matrix theory to study random perturbation of low rank matrices, see \cite{bai2012sample,bai2008central,huang2018mesoscopic,benaych2011fluctuations,benaych2011eigenvalues,tao2013outliers}. However, our case is not of low rank, instead we need to invert a full rank matrix. To address this challenge, we implement a finite rank approximation, which effectively transforms our problem into one of finite rank perturbation. This can be analyzed using the Woodbury formula. A crucial aspect of our approach is to establish that the error introduced by the finite rank approximation is minor and does not impact the distribution limit theorems we aim to prove. This is particularly pertinent in the degenerate case, where the fluctuation of the largest eigenvalue is of order $O(1)$, contrasting with the $O(\sqrt{n})$ order typically expected. This is done through detailed resolvent expansion analyses, and the error can be made arbitrarily small by selecting a sufficiently high rank for our approximation.

Our second main result concerns about the extreme eigenvalues of $W$-random graphs from a graphon $W$. As an intermediate step, we study the adjacency matrix $\bA_n$ conditioning on the connectivity probability matrix $\bW_n=\{W(U_i, U_j)\}_{1\leq i\neq j\leq n}$. This can be viewed as an inhomogeneous Erd{\H o}s-R{\'e}nyi model, where edges are added independently among the $n$ vertices with varying probabilities $p_{ij} = W(U_i, U_j)$. Many popular random graph models arise as special cases of inhomogeneous Erd{\H o}s-R{\'e}nyi model such as random graphs with given expected degrees \cite{chung2003spectra} and stochastic block models \cite{holland1983stochastic}. 

The adjacency matrix $\bA_n$ decomposes as the sum of the centered adjacency matrix and the connectivity probability matrix:
\begin{align}\label{e:decomp1}
\bA_n=\left(\bA_n-\E\left[\bA_n\middle|U_1,\ldots,U_n\right]\right)+\E\left[\bA_n\middle|U_1,\ldots,U_n\right]=(\bA_n-\bW_n) + \bW_n.
\end{align}
The empirical eigenvalue distributions and the behavior of extreme eigenvalues of centered adjacency matrices in inhomogeneous Erd{\H o}s--R{\'e}nyi graphs have been the subject of extensive study, as detailed in \cite{zhu2020graphon,benaych2020spectral,benaych2019largest}. Some of these findings also cover sparse graph regimes. In the context of the uncentered adjacency matrix $\bA_n$, it has been established \cite{chakrabarty2021spectra} that in sparse settings the empirical eigenvalue distributions converge towards a deterministic measure. 
 The fluctuations of the extreme eigenvalues of adjacency matrix $\bA_n$ have been studied in a recent work \cite{chakrabarty2020eigenvalues}. It has been proven that if the connectivity probability matrix is of finite rank $k$, then the joint distribution of the $k$ largest eigenvalues of $\bA_n$ converge jointly to a multivariate Gaussian law. When the connectivity probability matrix is constant, these results coincide with the established fluctuations of the maximum eigenvalue in homogeneous Erd{\H o}s--R{\'e}nyi graphs, \cite{erdHos2013spectral}. 
 Our result, \Cref{p:main3}  extends these results to the general infinite rank connectivity probability matrix constructed from Lipschitz graphons. This together with \Cref{t:main1} leads to our second main result, \Cref{t:main2} regarding $W$-random graphs from a Lipschitz graphon. If the graphon's degree function is constant, the fluctuation of the largest eigenvalue converges to the generalized chi-squared distribution. Otherwise, it converges to a normal distribution.

When the connectivity matrix $\bW_n$ in \eqref{e:decomp1} is of finite rank, the extreme eigenvalues of \eqref{e:decomp1} can be studied as a spiked Wigner matrix model, which has been intensively studied in the past decades \cite{feral2007largest, baik2005phase,chakrabarty2020eigenvalues,furedi1981eigenvalues,baik2006eigenvalues,capitaine2009largest,knowles2013isotropic,knowles2014outliers}.
Full rank deformation of the Gaussian unitary matrix and Wigner matrices have also been studied in \cite{capitaine2016fluctuations, lee2016extremal}. In our case, it turns out the connectivity probability matrix $\bW_n$ is dominant (it is of full rank). This prompts us to consider the adjacency matrix $\bA_n$ as a small perturbation of $\bW_n$. Similarly to the study of the random kernel matrix, we again derive a master equation \eqref{eq:Anresolvent} which characterizes the largest eigenvalue of $\bA_n$. We then analyze the master equation by a perturbation argument, and express the largest eigenvalue in terms of the kernel matrix $\bW_n$. Using the estimates on the eigenvalue and eigenvectors of the kernel matrix from the first part, we finally show that the difference between the largest eigenvalue of $\bA_n$ and $\bW_n$ has a Gaussian fluctuation, independent of the contribution of $\bW_n$ conditional on the node information $U_1,\ldots, U_n$. The decomposition in \eqref{e:decomp1} along with a standard application of Weyl's inequality shows that in \blue{the} non-degenerate case the difference $\bA_n - \bW_n$ has \blue{a} negligible contribution. On the other hand, in the degenerate case, the fluctuation of $\lambda_1(\bA_n)$ follows from the contribution of $\lambda_1(\bW_n)$ and the independent Gaussian contribution of $\lambda_1(\bA_n) - \lambda_1(\bW_n)$.

The remaining part of the paper is organized as following: the main results of the paper \Cref{t:main1} and \Cref{t:main2} are stated in \Cref{s:main}. We validate the main results through numerical experiments in \Cref{sec:simulations}. and outline the proof of our main results in \Cref{s:outline}. We collect some preliminary results on kernel matrices in \Cref{s:kernel} and their proofs are deferred to Appendix \ref{s:kernelproof} and Appendix \ref{s:kernelproof2}. Proof details for \Cref{t:main1} are presented in \Cref{s:Kerneleig} and Appendix \ref{s:Kerneleigproof}. Proof details for \Cref{t:main2} are given in \Cref{s:Matrixeig} and Appendix \ref{s:Matrixeigproof}.
We collect some useful facts on the spectrum of self-adjoint compact operators in Appendix \ref{s:operator}. 

\subsection{Notations}\label{sec:notations}
In this section we collect common notations that are used throughout the article.
\begin{itemize}
    \cblue
    \item We use $\pto$ and $\dto$ to denote convergences in probability and distribution respectively as $n\ra\infty$.\cblack
    \item A random variable $X_n = O_p(a_n)$ implies that for all $\vep>0$ there exists $M_\vep>0$ such that, $\P\left[\left|X_n/a_n\right|>M_\vep\right]\leq \vep\text{ for all large enough }n$.
    \item The notations $a\lesssim_{\theta}b$ and $a=O_{\theta}(b)$ are used to say $a\leq C_{\theta}b$ for some constant $C_\theta>0$ depending on a parameter $\theta$. A similar definition applies to $a\gtrsim_\theta b$.
    \item A random variable $X_n = o_p(a_n)$ implies the ratio $X_n/a_n$ converges in probability to $0$ as $n\ra\infty$ and in the deterministic case $a_n = o(b_n)$ implies the ratio $a_n/b_n\ra 0$ as $n\ra\infty$. \cred
    \item For a symmetric and bounded function $f:[0,1]^2\ra \bR$ we use the notation $\sigma(f)$ to denote the spectrum of $f$, that is the set of eigenvalues of the integral operator $T_f:L^2[0,1]\ra L^2[0,1]$ defined as $T_f(g)(\cdot) = \int f(\cdot,y)g(y)\mathrm dy$.
    \item Enumerate the eigenvalues of the integral operator $T_f$, as $\lambda_1(f)\geq \lambda_2(f)\geq \cdots\geq 0$ and $\lambda_1^{\prime}(f)\leq \lambda_2^\prime(f)\leq \cdots\leq 0$. Furthermore for all $j\geq 1$, let $\phi_{j, f}$ and $\phi_{j, f}^\prime$ be the orthonormal eigenfunctions corresponding to the eigenvalues $\lambda_j(f)$ and $\lambda_j^\prime(f)$ respectively.\cblack
    \item To denote the Gaussian distribution with mean $\mu$ and variance $\sigma^2$ the notation $\cN(\mu,\sigma^2)$ is used and to denote Uniform distribution on $[0,1]$ the notation $\text{Unif}\ [0,1]$ is used.
    \item We use the notation $C$ to denote a universal positive constant.
\end{itemize}

\section{Main Results}\label{s:main}
\blue{
We begin by first formally defining the \textit{kernel} function.
\begin{definition}[Kernel]
    A kernel is a measurable function $\sfK:[0,1]^2\mapsto \bR$ which is symmetric that is $\sfK(x,y) = \sfK(y,x)$ for all $x,y\in [0,1]$.
\end{definition}}
We make the following assumptions on the kernel \blue{function}. The first one requires that the kernel is Lipschitz continuous, and the second one requires that there is a spectral gap.
\begin{assumption}\label{assmp:assumptionkernel}
    We assume the kernel $\sfK: [0,1]^2\mapsto \bR$, has the following properties:
    \begin{enumerate}
        \item [1.] \blue{$\|\sfK\|_{\infty}\leq 1$, $\sfK$ is symmetric, that is $\sfK(x,y) = \sfK(y,x)$ for all $x,y\in [0,1]$} and $\sfK$ is Lipschitz continuous with Lipschitz constant $L_{\sfK}>0$.
        \item [2.] Recalling notations from Section \ref{sec:notations} enumerate the eigenvalues of $T_\sfK$ as $\lambda_{1}(\sfK)\geq \lambda_{2}(\sfK)\geq\cdots\geq 0$ and $\lambda_{1}^{\prime}(\sfK)\leq \lambda_{2}^{\prime}(\sfK)\leq\cdots\leq 0$, \blue{$\phi_{j, \sfK}$ and $\phi_{j, \sfK}^{\prime}$} being the orthonormal eigenfunctions corresponding to the eigenvalues $\lambda_{j}(\sfK)$ and $\lambda_{j}^{\prime}(\sfK)$ respectively. Then,
        \begin{align*}
            \left|\lambda_{1}(\sfK)-\lambda_{2}(\sfK)\right|>0.
        \end{align*}
    \end{enumerate}
\end{assumption}

In the following we formally introduce the notion of degeneracy of a kernel $\sfK$. The behavior of the largest eigenvalue of the random kernel matrix depends on the degeneracy of the kernel. 
\begin{definition}\label{def:degen}
    Let $\sfK$ be a kernel with \blue{$\phi_{1, \sfK}$} the eigenfunction corresponding to the largest eigenvalue. Then $\sfK$ is called degenerate if \blue{$\phi_{1, \sfK}$} is almost surely constant.
\end{definition}
Notice that by Definition \ref{def:degen} a graphon $W$ is degenerate if and only if the degree function of $W$ is almost surely constant. In other words, a graphon is degenerate if and only if it is degree regular. A similar notion of degeneracy has been studied with respect to small subgraph counts in \cite{hladky2021limit} and \cite{bhattacharya2023fluctuations}.

Our first main result is on the extreme eigenvalues of the $n\times n$ kernel matrix $\sfKmat_{n}$ constructed from $\sfK$,
\begin{align}\label{e:defKn}
(\sfKmat_{n})_{i,j}=\sfK(U_i, U_j)\delta_{i\neq j},\quad 1\leq i,j\leq n,
\end{align}
where  $\bm{U}_{n}:= (U_{1},\ldots, U_{n})\overset{i.i.d.}{\sim}\text{Unif}[0,1]$. We discover a dichotomous behavior of the extreme eigenvalues of the kernel matrix $\sfKmat_{n}$.

\begin{theorem}\label{t:main1}
Adopt \Cref{assmp:assumptionkernel}, and construct the kernel matrix $\sfKmat_{n}$ as in \eqref{e:defKn}. We denote the largest eigenvalue of $\sfKmat_{n}$ as $\lambda_1(\sfKmat_{n})$, then \blue{as $n\ra\infty$ we have the following results:}
\begin{enumerate}
\item If  $\sfK$ is not degenerate, namely \blue{$\phi_{1, \sfK}$} is not a constant function, then 
\begin{align}\label{e:normal}
    \sqrt{n}\bigg(\frac{\lambda_{1}(\sfKmat_{n})}{n} - \lambda_{1}(\sfK)\bigg)\dto \cN\left(0,\lambda_{1}(\sfK)^2\Var\left(\blue{\phi_{1,\sfK}^2(U)}\right)\right),
\end{align}
where $U{\sim}\emph{Unif}[0,1]$.
\item If $\sfK$ is degenerate, namely \blue{$\phi_{1,\sfK}$} is a constant function, then
\begin{align*}
    \lambda_{1}(\sfKmat_{n}) - (n-1)\lambda_{1}(\sfK)\dto\zeta_{\infty}
\end{align*}
where  
\begin{align}\label{e:zeta}
\zeta_\infty:= \sum_{\lambda\in \sigma(\sfK)\setminus\{\lambda_{1}(\sfK)\}}\frac{\lambda_{1}(\sfK)\lambda}{\lambda_{1}(\sfK) - \lambda}(Z_{\lambda}^2-1) + \sum_{\lambda\in \sigma(\sfK)\setminus\{\lambda_{1}(\sfK)\}}\frac{\lambda^2}{\lambda_{1}(\sfK) - \lambda},
\end{align}
and $\{Z_\lambda: \lambda\in \sigma(\sfK)\}$ are generated independently from the standard normal distribution.
\end{enumerate}
\end{theorem}

\cblue
Under the assumption $\|\mathsf{K}\|_{\infty} \le 1$, it can be easily seen that the infinite series in \eqref{e:zeta} converges in the $L^2$ sense.
\cblack
Notice that the above Theorem holds true whenever the kernel $\sfK$ satisfies Assumption \ref{assmp:assumptionkernel} and the matrix $\sfKmat_n$ has zero as \blue{the} diagonal elements. The result can be easily modified whenever the diagonal entries are given by $\sfK(U_i, U_i)$ for all $1\leq i\leq n$. 

\begin{corollary}\label{cor:diag}
    Adopt Assumption \ref{assmp:assumptionkernel}, and consider the kernel matrix $$\sfKmat_n = (K(U_i,U_j))_{i,j=1}^n.$$ Then for the largest eigenvalue $\lambda_1(\sfKmat_n)$,

    \begin{enumerate}
        \item If  $\sfK$ is not degenerate, namely \blue{$\phi_{1,\sfK}$} is not a constant function, then 
        \begin{align*}
            \sqrt{n}\bigg(\frac{\lambda_{1}(\sfKmat_{n})}{n} - \lambda_{1}(\sfK)\bigg)\dto \cN\left(0,\lambda_{1}(\sfK)^2\Var\left(\blue{\phi_{1, \sfK}^2(U)}\right)\right),
        \end{align*}
        where $U{\sim}\emph{Unif}[0,1]$.
        \item If $\sfK$ is degenerate, namely \blue{$\phi_{1,\sfK}$} is a constant function and $\sum_{\lambda\in \sigma(\sfK)}|\lambda|<\infty$, then
        \begin{align*}
            \lambda_{1}(\sfKmat_{n}) - (n-1)\lambda_{1}(\sfK)\dto\sum_{\lambda\in \sigma(\sfK)\setminus\{\lambda_{1}(\sfK)\}}\frac{\lambda_{1}(\sfK)\lambda}{\lambda_{1}(\sfK) - \lambda}Z_{\lambda}^2
        \end{align*}
        where $\{Z_\lambda: \lambda\in \sigma(\sfK)\}$ are generated independently from the standard normal distribution.
    \end{enumerate}
\end{corollary}
In contrast to Theorem \ref{t:main1}(2) the additional assumption on summability of eigenvalues of the operator $\sfK$ in Corollary \ref{cor:diag}(2) is needed to ensure existence of the asymptotic distribution.

\begin{remark}
We remark that if $\sfK$ is degenerate, namely \blue{$\phi_{1,\sfK}$} is a constant function, then $\Var(\blue{\phi_{1,\sfK}^2(U)})=0$ for $U{\sim}\text{Unif}[0,1]$, and the righthand side of \eqref{e:normal} degenerates. The limit $\zeta_\infty$ from \eqref{e:zeta} degenerates, namely $\zeta_\infty\equiv0$, only if $\lambda=0$ for all $\lambda\in \sigma(\sfK)\setminus\{\lambda_{1}(\sfK)\}$. In this case $\sfK=\lambda_1(\sfK)\bm1$ is a constant kernel. 

\end{remark}

\begin{remark}
    \red{Our proofs can be easily adapted to extend the results in \Cref{t:main1} and \Cref{cor:diag} to other eigenvalues of $\sfKmat_n$. For $t>1$, denote $\lambda_t(\sfKmat_n)$ as the $t$-th largest eigenvalue of $\sfKmat_n$ and let $\phi_{t,\sfK}$ be the $t^{th}$ eigenfunction of $T_\sfK$. If $|\lambda_t(\sfK)-\lambda_{t-1}(\sfK)|, |\lambda_t(\sfK)-\lambda_{t+1}(\sfK)|>0$, then we will have similar dichotomous distributional convergence results for $\lambda_t(\sfKmat_n)$ as in \Cref{t:main1} and \Cref{cor:diag} with $\lambda_1(\sfK)$ replaced by $\lambda_t(\sfK)$. Furthermore, in the non-degenerate setting, that is where $\phi_{t,\sfK}$ is not a constant function, in the limiting distribution $\phi_{1,\sfK}$ is replaced by $\phi_{t,\sfK}$. In the degenerate case, that is where $\phi_{t,\sfK}$ is a constant function, the sum in the limiting distribution is now taken over $\lambda\in \sigma(\sfK)\setminus\{\lambda_{t}(\sfK)\}$.}
\end{remark}


\begin{remark}\label{r:cond}
The convergence to the normal distribution \eqref{e:normal} has been proven in \cite[Theorem 5.1]{koltchinskii2000random} for all eigenvalues under the following assumptions: there exists a sequence $R_n\rightarrow \infty$ such that
\begin{align}\label{e:c1}
\sum_{|r|>R_n}\lambda_r^2(\sfK)=o(n^{-1})
\end{align}
and 
\begin{align}\label{e:c2}
\sum_{|r|\leq R_n, |s|\leq R_n}\int_0^1 \phi_r^2 \phi_s^2 \rd x \sum_{|r|\leq R_n, |s|\leq R_n}(\lambda_r(\sfK)^2+\lambda_s(\sfK)^2)\int_0^1 \phi_r^2 \phi_s^2 \rd x=o(n).
\end{align}
The conditions \eqref{e:c1} and \eqref{e:c2} are not easy to check. Our main result \Cref{t:main1} only requires that $\sfK$ is Lipschitz, which is easier to check. We remark that Lipschitz kernels in general does not satisfy the assumptions \eqref{e:c1} and \eqref{e:c2}. Since for Lipschitz kernels, the eigenvalues decay like $1/n^{3/2+}$ \cite[Section 4]{reade1983eigen}, so we can take $R_n=\sqrt n$ in the \eqref{e:c1}.  Then in \eqref{e:c2}, if the eigenvector integrals are atleast $O(1)$, then the lefthand side of \eqref{e:c2} simplifies
\begin{align*}
&\phantom{{}={}}\sum_{|r|\leq R_n, |s|\leq R_n}\int_0^1 \phi_r^2 \phi_s^2 \rd x \sum_{|r|\leq R_n, |s|\leq R_n}(\lambda_r^2+\lambda_s^2)\int_0^1 \phi_r^2 \phi_s^2 \rd x\\
&\asymp \sum_{|r|\leq R_n, |s|\leq R_n} 1 \sum_{|r|\leq R_n, |s|\leq R_n}(\frac{1}{r^{3}}+\frac{1}{s^{3}})\asymp R_n^3=n^{3/2}.\nonumber
\end{align*}
This fails the assumption \eqref{e:c2}. In \Cref{assmp:assumptionkernel}, we assume that $\sfK$ is Lipschitz, which can possibly be weakend to piecewise Lipschitz, or even piecewise H{\"o}lder continuous. But we will pursue it in the future work. 
\end{remark}

\begin{remark}
More generally, we can consider any probability space $(\bR, \cB, \mu)$, where $\cB$ is the Borel sigma algebra on $\bR$ and $\mu$ is a probability measure on $\bR$. Let $\sfH:\Omega^2 \mapsto \bR$ be a symmetric kernel, that is, a measurable function symmetric in its two entries. Let $\bm X_n=(X_1, X_2, \cdots, X_n)\overset{i.i.d.}{\sim} \mu$, and we can construct the following random matrix, 
\begin{align*}
(\bm\sfH_{n})_{i,j}=\sfH(X_i, X_j)\delta_{i\neq j},\quad 1\leq i,j\leq n.
\end{align*}
Our result \Cref{t:main1} gives fluctuation of the largest eigenvalue of $\bm\sfH_{n}$.
Denote the cumulative density function of $\mu$ as $F_\mu$, and its functional inverse as $F_\mu^{-1}$, then $F_\mu^{-1}(U_i)$ has the same law as $X_i$, where $U_1, U_2, \cdots, U_n$ are i.i.d. uniform distributed on $[0,1]$. Denote the pull back kernel under $F_\mu^{-1}$ as
\begin{align}\label{e:constructK}
\sfK(\cdot, \cdot)=\sfH(F_\mu^{-1}(\cdot), F_\mu^{-1}(\cdot)),
\end{align}
and the corresponding random kernel matrix
\begin{align*}
(\bm\sfK_{n})_{i,j}=\sfK(U_i, U_j)\delta_{i\neq j}=\sfH(F_\mu^{-1}(U_i), F_\mu^{-1}(U_j))\delta_{i\neq j},\quad 1\leq i,j\leq n.
\end{align*}
Then $\bm\sfK_{n}$ has the same law as $\bm\sfH_{n}$, and \Cref{t:main1} holds for $\bm\sfH_{n}$, provided that $\sfK$ constructed in \eqref{e:constructK} satisfies \blue{\Cref{assmp:assumptionkernel}.}
\end{remark}

Our second main result concerns the largest eigenvalue of the adjacency matrix coming from a graphon $W$. \blue{Before stating the results, we first define the graphon $W$ and the adjacency matrix $\bA_n$ coming from $W$.
\begin{definition}[Graphon]\label{def:graphon}
    A graphon is measurable function $W:[0,1]^2\mapsto [0,1]$ which is symmetric, that is for all $x,y\in [0,1]$, $W(x,y)=W(y,x)$.
\end{definition}}
Note that the graphon $W$ can be considered as a kernel and thus we assume that $W$ satisfies \Cref{assmp:assumptionkernel}.
\blue{Suppose $U_{1},\ldots, U_{n}$ are} generated independently from $\blue{\text{Unif}\ [0,1]}$. Then we consider an adjacency matrix $\bA_{n}$ defined as
\begin{align}\label{eq:defAn}
    \bA_{n}(i,j)\sim \text{Ber}(W(U_{i}, U_{j})),\ 1\leq i<j\leq n.
\end{align}
In this section we consider the fluctuation of the eigenvalues of $\bA_{n}$, in particular the largest eigenvalue $\lambda_{1}(\bA_{n})$. 

\blue{We begin by introducing the notion of the degree function of a graphon and relate it to the largest eigenfunction of $W$.
\begin{definition}
    The degree function of a graphon $W$ is defined as
    \begin{align*}
        d_W(x) = \int W(x,y)\rd y,\quad x\in [0,1].
    \end{align*}
\end{definition}
Notice that if the largest eigenfunction $\phi_{1,W}$ is a constant function, then by definition the degree function $d_W(x)$ is also a constant function. On the other hand, if the degree function $d_W(x)$ is a constant function, say $d_W(x) \equiv C\geq 0$ then for any eigenvalue $\lambda\in \sigma(W)$ with corresponding orthonormal eigenfunction $\phi_\lambda$, using Cauchy-Schwarz inequality we have,
\begin{align}\label{eq:eigenbddeg}
\lambda = \int \phi_\lambda(x)W(x,y)\phi_\lambda(y)\rd y\rd x \leq \left(\int \phi_\lambda(x)^2W(x,y)\rd y\rd x\right) = C.
\end{align}
Thus $C$ is the largest eigenvalue of $W$ with the constant eigenfunction $1$. This shows that if $d_W$ is a constant function then $\phi_{1,W}$ is also a constant function. With the above relation we are now ready to state our second main result.}

\begin{theorem}\label{t:main2}
Fix a graphon $W$ satisfying \Cref{assmp:assumptionkernel}, denote its largest eigenvalue as $\lambda_1(W)$ and the associated eigenfunction \blue{$\phi_{1,W}$}. We consider the adjacency matrix $\bA_{n}$ corresponding to the graphon $W$ as in \eqref{eq:defAn}, and denote its largest eigenvalue as $\lambda_1( \bA_{n})$, then 
\begin{enumerate}
\item If  the degree function of $W$ is not a constant, namely \blue{$\phi_{1,W}$} is not a constant function, then 
\begin{align*}
    \sqrt{n}\bigg(\frac{\lambda_{1}(\bA_{n})}{n} - \lambda_{1}(W)\bigg)\dto \cN\left(0,\lambda_{1}(W)^2\Var\left(\phi_{1}^2(U)\right)\right),
\end{align*}
where $U{\sim}\emph{Unif}[0,1]$.
\item If the degree function of $W$ is a constant, namely \blue{$\phi_{1,W}$} is a constant function, then
\begin{align}\label{eq:secondconvg}
    \lambda_{1}(\bA_{n}) - (n-1)\lambda_{1}(W)\dto\zeta_{\infty} + \cN(\alpha,\sigma^2)
\end{align}
where  
\begin{align*}
\zeta_\infty:= \sum_{\lambda\in \sigma(W)\setminus\{\lambda_{1}(W)\}}\frac{\lambda_{1}(W)\lambda}{\lambda_{1}(W) - \lambda}(Z_{\lambda}^2-1) + \sum_{\lambda\in \sigma(W)\setminus\{\lambda_{1}(W)\}}\frac{\lambda^2}{\lambda_{1}(W) - \lambda},
\end{align*}
$\{Z_\lambda: \lambda\in \sigma(W)\}$ are generated independently from the standard normal distribution, and $\cN(\alpha,\sigma^2)$ represents an independent normal distribution with mean $\alpha$ and variance $\sigma^2$ given by,
\begin{align*}
&\alpha=\frac{1}{\lambda_{1}(W)}\int\frac{\blue{\phi_{1,W}^{2}(x) + \phi_{1,W}^{2}(y)}}{2}W(x,y)(1-W(x,y))\rd x\rd y,\\
&\sigma^2= 2\int \blue{\phi_{1,W}^{2}(x)\phi_{1,W}^{2}(y)}W(x,y)(1-W(x,y))\rd x\rd y.
\end{align*}
\end{enumerate}
\end{theorem}

\begin{remark}
When the graphon $W$ has a constant degree profile, the largest eigenvalue of the adjacency matrix fluctuates on the scale $\Omega(1)$. When the graphon $W$ has an irregular degree profile, the largest eigenvalue fluctuates on a much larger scale, $\Omega(\sqrt n)$.
\end{remark}

\begin{remark}
\blue{Our result can be extended to other eigenvalues of  $\bA_n$. Once again for $t>1$, denote $\lambda_t(\bA_n)$ as the $t$-th largest eigenvalue of $\bA_n$ and assume $|\lambda_t(W)-\lambda_{t-1}(W)|, |\lambda_t(W)-\lambda_{t+1}(W)|>0$. In contrast to \Cref{t:main2}, in this case we will not have a dichotomy in the limiting distribution. This follows by noticing that if $\phi_{t,W}$ is a constant function, then the degree function $d_W\equiv \lambda_t(W)$ and hence by \eqref{eq:eigenbddeg} we must have that $\lambda_t(W) = \lambda_{1}(W)$ which violates the spectral gap assumption from above. Instead, we will always have the convergence to the normal distribution:
\begin{align*}
    \sqrt{n}\bigg(\frac{\lambda_{t}(\bA_n)}{n} - \lambda_{t}(W)\bigg)\dto \cN\left(0,\lambda_{t}(W)^2\Var\left(\phi_{t, W}^2(U)\right)\right),
\end{align*}
where $U{\sim}\mathrm{Unif}[0,1]$.}
\end{remark}

\begin{remark}
    A similar dichotomy of distributional convergence is also present for motif counts in random graphs generated as in \eqref{eq:defAn}. In particular, \cite{hladky2021limit} extended the notion of edge-regularity \blue{(constant degree function)} to clique-regularity and showed that if a Graphon $W$ is regular with respect to a clique $K_{r}$ then the asymptotic distribution of $K_{r}$ counts in the random graph in $\eqref{eq:defAn}$ has a structure similar to \eqref{eq:secondconvg} with a centered Gaussian and a Non-Gaussian component, where the non-Gaussian component is a weighted sum of independent chi-squared random variables with the weights related to the spectrum of a graphon derived from $W$. On the other hand, for $K_{r}$-irregular graphons, we get the familiar gaussian convergence. This result was further extended for general subgraphs by \cite{bhattacharya2023fluctuations}, who extended the notion of clique regularity to general subgraph regularity and showed a similar dichotomous asymptotic distribution. 
\end{remark}

\section{Simulations}\label{sec:simulations}
In this section we validate the asymptotic distributions from Theorem \ref{t:main1} and Theorem \ref{t:main2}. In particular we construct example of Graphons (which also acts as kernels) satisfying Assumption \ref{assmp:assumptionkernel} and the conditions of Theorems \ref{t:main1} and \ref{t:main2}. Define, $\phi_1(x) = 1$, $\phi_2(x) = \sqrt{3}(2x-1)$, $\phi_3(x) = \sqrt{5}(6x^2 - 6x + 1)$ and $\phi_4(x) = \sqrt{7}\left(20x^3 - 30x^2 + 12x - 1\right)$. Notice that $\phi_i, 1\leq i\leq 4$ are the first four ``Shifted'' Legendre Polynomial. By definition it is easy to notice that the collection $\{\phi_i, 1\leq i\leq 5\}\in L^2[0,1]$ and are orthonormal. Now we define the graphons as follows,
\begin{align*}
    W_1(x,y) = \frac{1}{2}\phi_2(x)\phi_2(y) + \frac{1}{9}\phi_3(x)\phi_3(y) + \frac{1}{30}\phi_4(x)\phi_4(y),
\end{align*}
and
\begin{align*}
    W_2(x,y) = \frac{1}{5}\phi_1(x)\phi_1(y) + \frac{1}{9}\phi_2(x)\phi_2(y) + \frac{1}{30}\phi_3(x)\phi_3(y).
\end{align*}
Notice that by construction $W_1$ and $W_2$ satisfies assumption \ref{assmp:assumptionkernel} and $W_2$ has constant largest eigenfunction, while for $W_1$ the largest eigenfunction is non-constant.\\

\subsection{Largest Eigenvalue of Kernel Matrix} For $U_1,\ldots,U_n$ generated randomly from $\mathrm{Unif}[0,1]$ distribution we consider the asymptotic distribution of largest eigenvalue of the kernel matrices constructed using $W_1$ and $W_2$ as in \eqref{e:defKn}. For $W_1$ and $W_2$ we consider $n=1000$ and $n=100$ respectively, and repeat the experiment $500$ times to get repeated samples of the largest eigenvalue and construct histogram of the properly scaled samples (according to Theorem \ref{t:main1}). We consider the asymptotic distribution from Theorem \ref{t:main1} and generate $10^5$ samples from it to provide a histogram. The comparison between sample distribution and asymptotic distribution is provided in Figure \ref{fig:kernmat}. The comparison presented in Figures \ref{fig:sub1} and \ref{fig:sub2} validates the asymptotic distribution presented in Theorem \ref{t:main1}.\\

\begin{figure}
    \centering
    \begin{subfigure}{7cm}
      \centering
      \includegraphics[scale = 0.43]{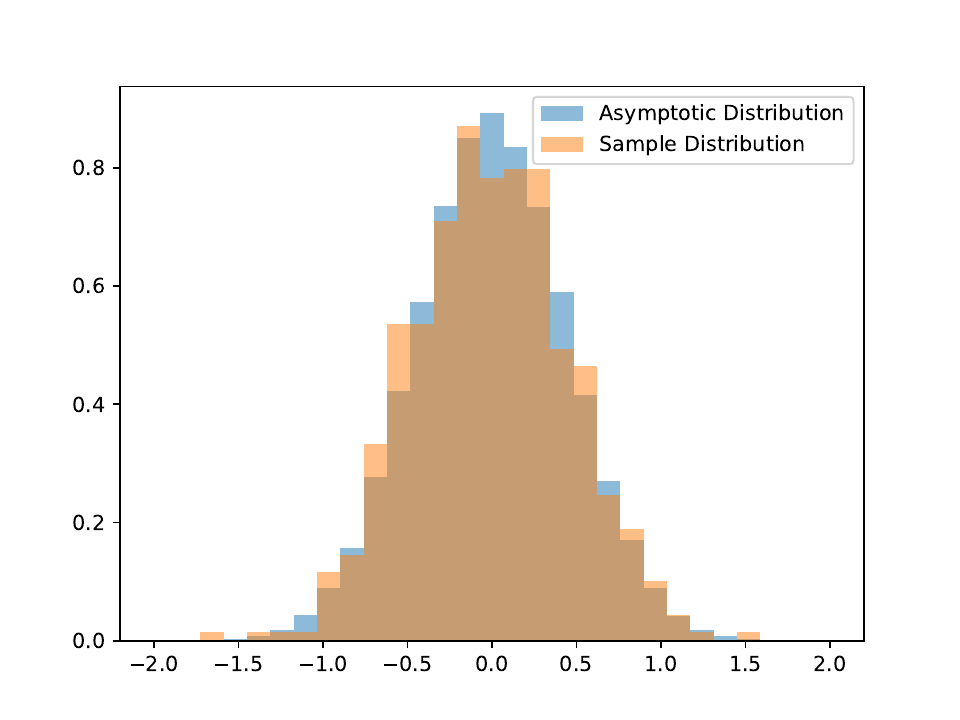}
      \caption{Kernel Matrix with $W_1$.}
      \label{fig:sub1}
    \end{subfigure}%
    \begin{subfigure}{7cm}
      \centering
      \includegraphics[scale = 0.43]{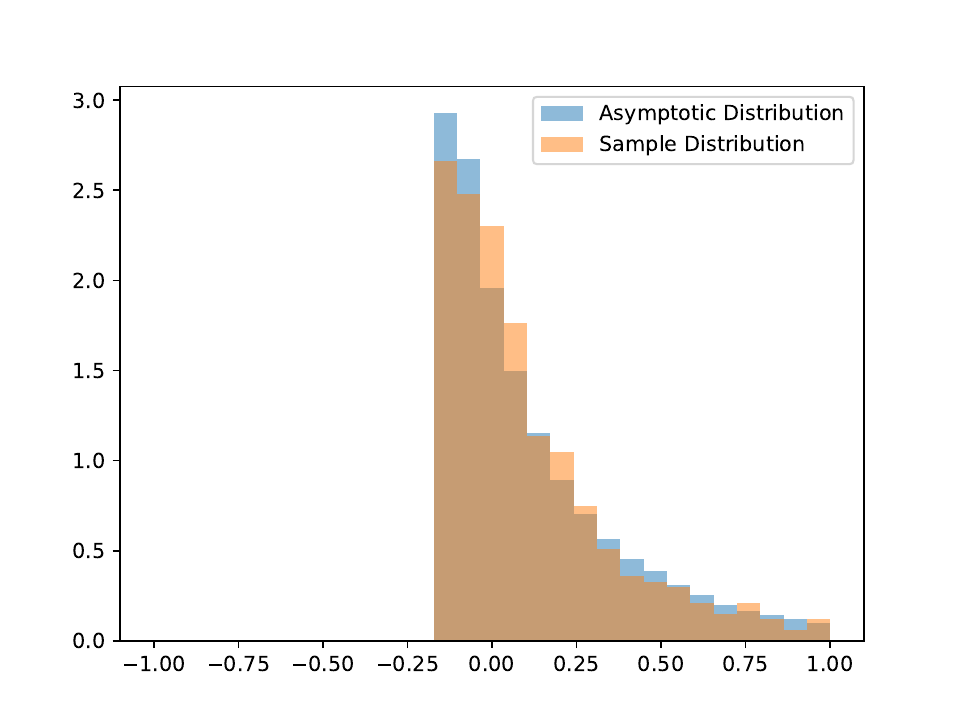}
      \caption{Kernel Matrix with $W_2$.}
      \label{fig:sub2}
    \end{subfigure}
    \caption{Sample and Asymptotic Distribution of Largest Eigenvalue of Kernel Matrices.}
    \label{fig:kernmat}
\end{figure}

\subsection{Largest Eigenvalue of Adjacency Matrix} Here we once again generate $U_1,\ldots,U_n$ randomly from $\mathrm{Unif}[0,1]$ and construct Adjacency Matrix using the Graphon $W_1$ and $W_2$ following \eqref{eq:defAn}. Once again as above we consider $n=1000$ and $n=100$ for $W_1$ and $W_2$ respectively, and calculate the largest eigenvalue of the Adjacency matrix. We repeat the experiment $500$ times to have $500$ sample for the largest eigenvalue and follow the scalings from Theorem \ref{t:main2} to provide the histogram of the samples. To compare with the asymptotic distribution we once again generate $10^5$ samples from the asymptotic distribution and provide histogram using the samples. The comparison between sample and asymptotic distribution is provided in Figure \ref{fig:Adjmat}, in particular the asymptotic distribution presented in Theorem \ref{t:main2} is validated by the comparison from Figures \ref{fig:sub1Adj} and \ref{fig:sub2Adj}.

\begin{figure}
    \centering
    \begin{subfigure}{7cm}
      \centering
      \includegraphics[scale = 0.43]{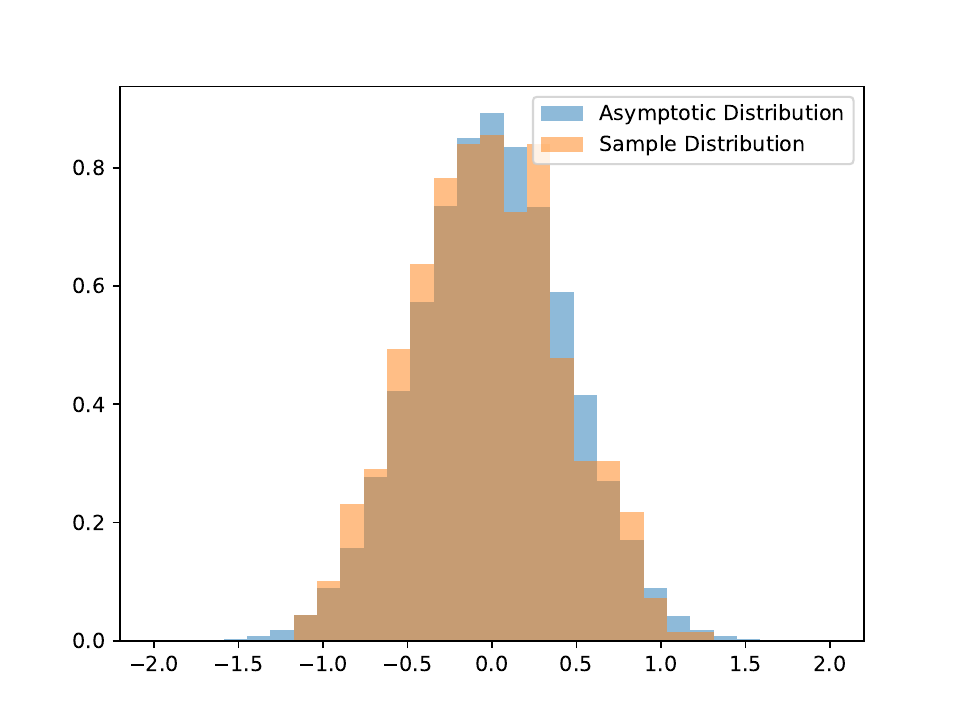}
      \caption{Adjacency Matrix with $W_1$.}
      \label{fig:sub1Adj}
    \end{subfigure}%
    \begin{subfigure}{7cm}
      \centering
      \includegraphics[scale = 0.43]{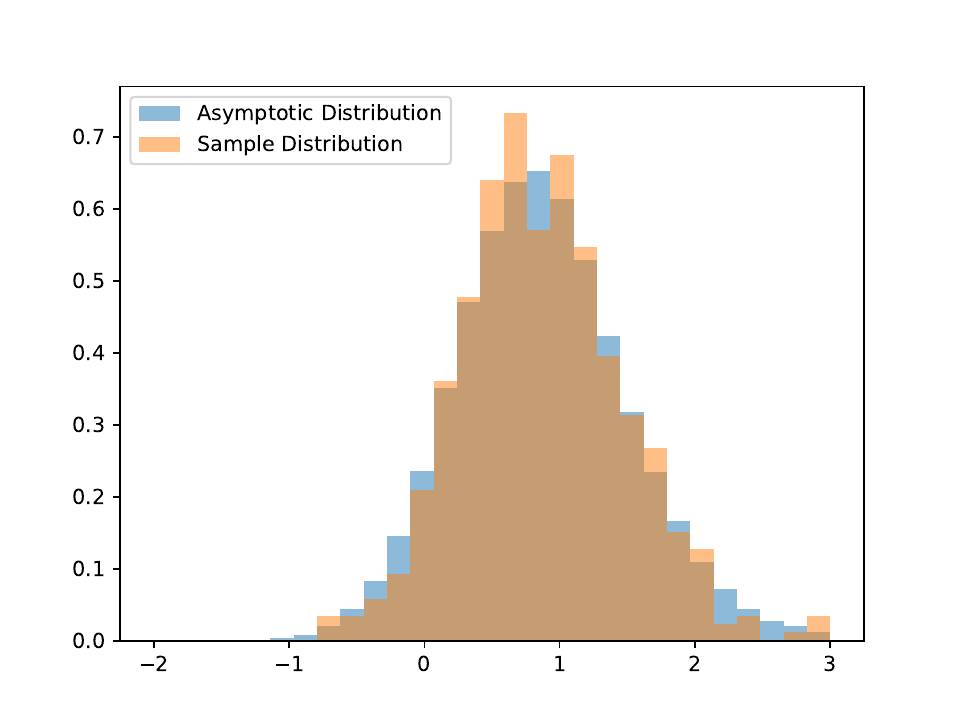}
      \caption{Adjacency Matrix with $W_2$.}
      \label{fig:sub2Adj}
    \end{subfigure}
    \caption{Sample and Asymptotic Distribution of Largest Eigenvalue of Adjacency Matrices.}
    \label{fig:Adjmat}
\end{figure}

\section{Proof Outlines}\label{s:outline}

\subsection{Proof of Theorem \ref{t:main1}}\label{s:prft1}
We recall the notations and assumptions on the kernel $\sfK$ from assumption \ref{assmp:assumptionkernel}. In the following for ease of exposition we suppress the dependence of eigenfunctions on the kernel $\sfK$ and write $\phi_j:=\phi_{j, \sfK}$ and similarly for $\phi_{j,\sfK}^\prime$. 

Since $\sfK$ is a self-adjoint integral operator (which is compact), we have the expansion
\begin{align}\label{e:Wdecomp}
    \sfK(x,y) = \sum_{j=1}^{\infty}\lambda_{j}(\sfK)\phi_{j}(x)\phi_{j}(y) + \sum_{j=1}^{\infty}\lambda_{j}^{\prime}(\sfK)\phi_{j}^{\prime}(x)\phi_{j}^{\prime}(y),
\end{align}
where the equality holds in $L_{2}$ sense. In this section we provide an outline of the proof of Theorem \ref{t:main1} by studying the largest eigenvalues of the $n\times n$ kernel matrix $\sfKmat_{n}$ defined in \eqref{e:defKn}. Recalling the decomposition \eqref{e:Wdecomp}, we can rewrite $\sfKmat_n$ as
\begin{align*}
    \sfKmat_{n} = \sum_{j=1}^{\infty}\lambda_{j}(\sfK)\left(\Phi_{j}(\bU_{n})\Phi_{j}(\bU_{n})^{\top} - \bD_{n,j}\right) + \sum_{j=1}^{\infty}\lambda_{j}^{\prime}(\sfK)\left(\Phi_{j}^{\prime}(\bU_{n})\Phi_{j}^{\prime}(\bU_{n})^{\top} - \bD_{n,j}^{\prime}\right),
\end{align*}
where $\Phi_{j}(\bU_{n}) := (\phi_{j}(U_{1}),\cdots,\phi_{j}(U_{n}))$, $\bD_{n,j}:= \diag(\phi_{j}^{2}(U_{1}),\cdots,\phi_{j}^{2}(U_{n}))$, and $\Phi_{j}^{\prime}(\bU_{n})$, $\bD_{n,j}^{\prime}$ are defined similarly through $\phi_{j}^{\prime}$ and the equality is in coordinate-wise $L_{2}$ sense. By definition the largest eigenvalue $\lambda_{1}(\sfKmat_{n})$ of $\sfKmat_{n}$ satisfies the equation,
\begin{align}\label{e:la1}
    \det\left(\lambda_{1}(\sfKmat_{n})\bI_{n} - \sfKmat_{n}\right) = 0.
\end{align}
In the following we will use \eqref{e:la1} as a starting point to get a simple equation of $\lambda_1(\sfKmat_n)$ (see \eqref{eq:1AB1=12}). We first start with a weak estimate of $\lambda_{1}(\sfK)$, which can be viewed as a law of large number statement. \cblue The following lemma follows as a direct consequence of Lemma \ref{lemma:lambda1Wnconc}.
\begin{lemma}\label{l:la1est}
    Under the assumptions of Theorem \ref{t:main1}, the following estimate holds,
    \begin{align*}
    \lambda_{1}(\sfKmat_{n})/n\pto \lambda_{1}(\sfK).
    \end{align*}
\end{lemma}
\cblack
To extend the result to the fluctuation of $\lambda_1(\sfKmat_n)$, we need to introduce the following notations,
\begin{align}\begin{split}\label{e:defAB}
    \bcA_{n}&:= \lambda_{1}(\sfKmat_{n})\bI_{n} + \sum_{\ell=1}^{\infty}\lambda_{\ell}(\sfK)\bD_{n,j} + \sum_{\ell=1}^{\infty}\lambda_{j}^{\prime}(\sfK)\bD_{n,\ell}^{\prime},\\
    \bcB_{n}&:= \sum_{\ell=2}^{\infty}\lambda_{\ell}(\sfK)\Phi_{\ell}(\bU_{n})\Phi_{\ell}(\bU_{n})^{\top} + \sum_{\ell=1}^{\infty}\lambda_{\ell}^{\prime}(\sfK)\Phi_{\ell}^{\prime}(\bU_{n})\Phi_{\ell}^{\prime}(\bU_{n})^{\top}.
\end{split}\end{align}
We remark that $\bcA_n$ is a diagonal matrix, and  the last two infinite sum in its definition gives the diagonal matrix $\diag (\{\sfK(U_i, U_i)\}_{1\leq i\leq n})$; $\bcB_{n}$ is the kernel matrix of $\sfK$ \blue{(with diagonal terms)} with the first eigenvalue removed. With the above notations we can now rewrite \eqref{e:la1} as,
\begin{align}\label{e:AB}
    \lambda_{1}(\sfKmat_{n})\bI_{n} - \sfKmat_{n} = \bcA_{n} - \bcB_{n} - \lambda_{1}(\sfK)\Phi_{1}(\bU_{n})\Phi_{1}(\bU_{n})^{\top}.
\end{align}
To further simplify \eqref{e:la1} we first show that with high probability $\bcA_n-\bcB_n$ is invertible. 
\begin{lemma}
Let $\bcA_n, \bcB_n$ be as defined in \eqref{e:defAB}, then $\bcA_n-\bcB_n$ is invertible with probability at least $1-16n\exp\left(-\frac{1}{6}(\log n)^2\right)$.
\end{lemma}
\begin{proof}
By Proposition \ref{lemma:AnBninvertible} and Lemma \ref{lemma:lambda1Wnconc}, we have that  $\bcA_n-\bcB_n$ is invertible with probability at least $1-16n\exp\left(-\frac{1}{6}(\log n)^2\right)$. 
\end{proof}

\cblue
Now plugging \eqref{e:AB} into \eqref{e:la1} and using the Weinstein-Aronszajn identity, we conclude that with probability at least $1-16n\exp\left(-\frac{1}{6}(\log n)^2\right)$, $\lambda_{1}(\sfKmat_{n})$ is characterized by the equation,
\begin{align}\label{eq:1AB1=12}
    \lambda_{1}(\sfK)\Phi_{1}(\bU_{n})^{\top}(\bcA_{n} - \bcB_{n})^{-1}\Phi_{1}(\bU_{n}) = 1.
\end{align}
\red{Notice that here we need to invert the matrix $\bcA_n - \bcB_n$, which is potentially full rank matrix. To analyze \eqref{eq:1AB1=12} we will now provide an approximation using finite (fixed) rank objects.} Towards that, for $m\geq 2$, we define the following finite (fixed) rank approximations of $\bcA_n$ and $\bcB_n$,
\begin{align}\begin{split}\label{e:defABk}
    \bcA_{n}^{(m)} &:= \lambda_{1}(\sfKmat_{n})\bI_{n} + \sum_{\ell=1}^{m}\lambda_{\ell}(\sfK)\bD_{n,j}^{\prime} + \sum_{\ell=1}^{m}\lambda_{j}^{\prime}(\sfK)\bD_{n,\ell}^{\prime},\\
    \bcB_{n}^{(m)} &:= \sum_{\ell=2}^{m}\lambda_{\ell}(\sfK)\Phi_{\ell}(\bU_{n})\Phi_{\ell}(\bU_{n})^{\top} + \sum_{\ell=1}^{m}\lambda_{\ell}^{\prime}(\sfK)\Phi_{\ell}^{\prime}(\bU_{n})\Phi_{\ell}^{\prime}(\bU_{n})^{\top}.
\end{split} \end{align}
\red{For simplification we assume that for all $1\leq \ell\leq m$, $\lambda_\ell(\sfK)>0$ and $\lambda_\ell^\prime(\sfK)<0$. The following proofs go through without this assumption, by defining the matrices $\bcA_{n}^{(m)}$ and $\bcB_{n}^{(m)}$ using only the non-zero eigenvalues up to index $m$ but with additional notational complications.} The following Proposition \ref{prop:IIandIIIreplace} with proof given in Section \ref{sec:proofofIIandIIIreplace} states that for sufficiently large enough $m$, we can replace $\bcA_{n} - \bcB_{n}$ in \eqref{eq:1AB1=12} by $\bcA_{n}^{(m)} - \bcB_{n}^{(m)}$ with arbitrarily small error.

\begin{prop}\label{prop:IIandIIIreplace}
    Recall the matrices $\bcA_{n}, \bcB_{n}, \bcA_{n}^{(m)} ,\bcB_{n}^{(m)}$ from \eqref{e:defAB} and \eqref{e:defABk}. For any fixed $\vep>0$ there exists $m(\vep)\in\mathbb{N}$ such that for any fixed $m\geq m(\vep)$ and for all $n\geq n(m,\vep)$ the following holds with probability $1-8\sqrt{\vep}$,
    \begin{align}\label{eq:reducedtok}
        \left|\frac{1}{\lambda_{1}(\sfK)} - \Phi_{1}(\bU_{n})^{\top}\left(\bcA_{n}^{(m)} - \bcB_{n}^{(m)}\right)^{-1}\Phi_{1}(\bU_{n})\right|\lesssim_{\sfK} \frac{\sqrt{\vep}}{n}.
    \end{align}
\end{prop}

Comparing \eqref{eq:reducedtok} with \eqref{eq:1AB1=12}, we need to invert the matrix $\bcA_{n}^{(m)} - \bcB_{n}^{(m)}$ instead of $\bcA_{n} - \bcB_{n}$. The advantage here is that because of the finite (fixed) rank, namely rank at most $2m$, and we can use the Woodbury formula to invert $\bcA_{n}^{(m)} - \bcB_{n}^{(m)}$. This leads to the following proposition, and we postpone its proof to Section \ref{sec:proofofla1decomposition}. We begin by introducing the following notations,
\begin{align*}
        r_{n,\ell}(\bU_n)= \lambda_{\ell}(\sfK)\sum_{i=1}^{n}\phi_{1}(U_{i})^2\phi_{\ell}(U_{i})^2\text{ and } s_{n,\ell}(\bU_n)= \frac{\lambda_{\ell}(\sfK)\lambda_{1}(\sfK)}{\lambda_{1}(\sfK) - \lambda_{\ell}(\sfK)}\left(\sum_{i=1}^{n}\phi_{1}(U_{i})\phi_{\ell}(U_{i})\right)^2,
\end{align*}
\normalsize
and define $r_{n,\ell}^\prime,s_{n,\ell}^\prime$ analogously using $\phi_\ell^\prime$ and $\lambda_{\ell}^\prime(\sfK)$ in place of $\phi_{\ell}$ and $\lambda_{\ell}(\sfK)$ respectively.

\begin{prop}\label{p:la1decomposition}
    We introduce the following quantities $T_{n,m,1}$ and $T_{n,m,2}$ given by
    \begin{align*}
        T_{n,m,1} := \frac{\lambda_{1}(\sfK)}{\lambda_{1}(\sfKmat_{n})}\sum_{\ell=1}^{m}r_{n,\ell}(\bU_n) + r_{n,\ell}^\prime(\bU_n) \text{ and } T_{n,m,2} := \frac{\lambda_{1}(\sfK)}{\lambda_{1}(\sfKmat_{n})}\sum_{\ell=2}^{m}s_{n,\ell}(\bU_n) + \sum_{\ell=1}^{m}s_{n,\ell}^\prime(\bU_n).
    \end{align*}
    \normalsize
    Fix any small $\vep>0$, choose $m(\vep)$ satisfying \eqref{eq:reducedtok} and fix $m\geq m(\vep)$. Then there exists $n(m,\vep)$ satisfying \eqref{eq:reducedtok} such that for any $n\geq n(m,\vep)$, the following holds with probability at least $1-9\sqrt{\vep}$,
    \begin{align}\label{eq:lambdanflucexpand}
        \left|\frac{\lambda_{1}(\sfKmat_{n})}{n} - \lambda_{1}(\sfK) - \lambda_{1}(\sfK)\left[\frac{\left\|\Phi_{1}(\bU_{n})\right\|_{2}^2}{n} - 1\right] - \frac{1}{n}\left(T_{n,m,2} - T_{n,m,1}\right) - \frac{\lambda_{1}(\sfKmat_{n})}{n}t_{n}\right|\lesssim_{\sfK}\frac{\sqrt{\vep}}{n}
    \end{align}
    where $|t_{n}|\lesssim_{\sfK,m} n^{-3/2}(\log n)^3$.
\end{prop}
The proof is now completed by analysing the fluctuation of the terms $\left[\|\Phi_{1}(\bU_{n})\|_{2}^2/n - 1\right]$, $T_{n,m,1}$ and $T_{n,m,2}$. We postpone the technical details to Section \ref{sec:completeproofkernel} where we show that under non-degeneracy of $\sfK$, the terms $T_{n,m,1}$ and $T_{n,m,2}$ are $o_{p}(\sqrt{n})$ and the dominant contribution is coming from $\left[\|\Phi_{1}(\bU_{n})\|_{2}^2/n - 1\right]$, whereas under degeneracy of $\sfK$, the term $T_{n,m,2} - T_{n,m,1} + \lambda_1(\sfK)$ converges to the limiting distribution $\zeta_\infty$ as in Theorem \ref{t:main1}. 
\cblack
\subsection{Proof of Corollary \ref{cor:diag}}
In the following, we sketch the proof of Corollary \ref{cor:diag}. The proof of part $(1)$ follows immediately from part $(1)$ in Theorem \ref{t:main1} and Weyl's inequality. The conclusion from part $(2)$ can be proved along the lines of proof of part $(2)$ in Theorem \ref{t:main1}. Hence, in the following we present a sketch of the proof for part $(2)$. Notice that the largest eigenvalue satisfies the equation,
\begin{align*}
    \det\left(\lambda_{1}(\sfKmat_n)\one_n - \sfKmat_n\right) = 0.
\end{align*}
As in \eqref{e:AB} the above equation can be rewritten as,
\begin{align*}
    \det\left(\lambda_{1}(\sfKmat_{n})\bI_{n} - \sfKmat_{n}\right) =\det\left( \bcA_{n} - \bcB_{n} - \lambda_{1}(\sfK)\Phi_{1}(\bU_{n})\Phi_{1}(\bU_{n})^{\top}\right) = 0,
\end{align*}
where $\bcA_n = \lambda_1(\sfKmat_n)\one_n$ and $\bcB_n$ is defined in \eqref{e:defAB}. Now, we can replicate the proof of Theorem \ref{t:main1} and for given $\varepsilon>0$ there exists $m(\vep)\geq 1$ such that for all $m\geq m(\vep)$ and $n\geq n(m,\vep)$ we get
\begin{align*}
    \left|\frac{\lambda_{1}(\sfKmat_{n})}{n} - \lambda_{1}(\sfK)- \frac{T_{n,m,2}}{n} - \frac{\lambda_{1}(\sfKmat_{n})}{n}t_{n}\right|\lesssim_{\sfK}\frac{\sqrt{\vep}}{n},
\end{align*}
where $T_{n,m, 2}$ is defined in Proposition \ref{p:la1decomposition} and $|t_{n}|\lesssim_{\sfK,m} n^{-3/2}(\log n)^3$. Notice that the above equation is similar to the one in \eqref{eq:lambdanflucexpand2} with $T_{n,m, 1} = 0$. This follows by recalling the proof of Proposition \ref{p:la1decomposition} and noticing that the term $T_{n,m, 1}$ was contributed because of the adjustment coming from the missing diagonal terms. The rest of the proof now follows along the arguments presented in \eqref{e:TTdiff}, \eqref{e:zetainf} and \eqref{e:mid1}. 

\subsection{Proof of Theorem \ref{t:main2}}\label{s:prft2}
In this section we \blue{provide} the proof of our main result, Theorem \ref{t:main2}. Define $\bW_n$ to be the $n\times n$ matrix with the $(i,j)^{th}$ entry given by, $W(U_i, U_j)$ for all $1\leq i\neq j\leq n$ and with empty diagonal. To analyse the fluctuation of the largest eigenvalue $\lambda_1(\bA_n)$, we will use the following decomposition,
\begin{align*}
    \lambda_1(\bA_n) = \lambda_1(\bA_n) - \lambda_1(\bW_n) + \lambda_1(\bW_n).
\end{align*}
Since $W$ satisfies \Cref{assmp:assumptionkernel} then from \Cref{t:main1} we know the fluctuation of $\lambda_{1}(\bW_{n})$. Thus, here we first proceed with finding out the fluctuation of the eigenvalue difference $\lambda_{1}(\bA_{n}) - \lambda_{1}(\bW_{n})$. Consider the eigendecomposition of $\bW_{n}$ as
\begin{align*}
    \bW_{n} = \sum_{i=1}^{n}\lambda_{i}(\bW_{n})\bv_{i}\blue{\bv_i^\top},
\end{align*}
where $\lambda_{1}(\bW_{n})\geq \lambda_{2}(\bW_{n})\geq \cdots\geq \lambda_{n}(\bW_{n})$ are the eigenvalues of the matrix $\bW_{n}$ with orthonormal eigenvectors $\bv_{1}, \bv_{2},\ldots, \bv_{n}$ respectively. Then define
\begin{align}\label{eq:defAtilde}
    \widetilde{\bA}_{n} = \sum_{i=2}^{n}\lambda_{i}(\bW_{n})\bv_{i}\bv_{i}^\top + \bA_{n} - \bW_{n},
\end{align}
and note that $\bA_{n} = \lambda_{1}(\bW_{n})\bv_{1}\bv_{1}^{\top} + \widetilde{\bA}_{n}.$ The following lemma, with proof provided in Section \ref{sec:proofofAninv}, states that with high probability $\lambda_{1}(\bA_{n})$ is not an eigenvalue of the matrix $\widetilde{\bA}_{n}$. 
\begin{lemma}\label{lemma:Aninv}
    Consider the matrix $\widetilde{\bA}_{n}$ defined in \eqref{eq:defAtilde} and consider $\sigma(\widetilde{\bA}_n)$ to be the eigenvalues of $\widetilde{\bA}_n$. Then
    \begin{align*}
        \inf\left\{|\lambda_{1}(\bA_{n}) - \lambda|:\lambda\in \sigma(\widetilde{\bA}_{n})\right\}\geq \frac{|\lambda_{1}(W) - \lambda_{2}(W)|}{4},
    \end{align*}
    with probability at least $1-Cn\exp\left(-\frac{1}{6}(\log n)^2\right)$.
\end{lemma}
Then with probability at least $1-Cn\exp\left(-\frac{1}{6}(\log n)^2\right)$, we can invert the matrix $\lambda_1(\bA_n)\bI_n - \widetilde{\bA}_n$ and by arguments similar to \eqref{eq:1AB1=12} we get,
\begin{align}\label{eq:Anresolvent}
    \lambda_{1}(\bW_{n})\bv_{1}^{\top}(\lambda_{1}(\bA_{n})\bI_{n} - \widetilde{\bA}_{n})^{-1}\bv_{1} = 1.
\end{align}
Now to invert the matrix $\lambda_{1}(\bA_{n})\bI_{n} - \widetilde{\bA}_{n}$ we introduce the following notations,
\begin{align}\label{eq:defCn}
    \bmB_{n} = \bA_{n} - \bW_{n}\text{ and }\bC_{n} = \lambda_{1}(\bA_{n})\bI_{n} - \sum_{j\neq 1}\lambda_{j}(\bW_{n})\bv_{j}\bv_{j}^{\top}.
\end{align}
Then we can rewrite the matrix in \eqref{eq:Anresolvent} as $\lambda_{1}(\bA_{n})\bI_{n} - \widetilde{\bA}_{n}=\bC_n-\bmB_n$. The following proposition collects some properties of $\bmB_n, \bC_n$. We postpone its proof to Section \ref{sec:proofofCinvnorm}.
\begin{prop}\label{p:Cninvnorm}
    Recall the matrices $\bmB_n, \bC_{n}$ from \eqref{eq:defCn}. Then with probability at least $1-Cn\exp\left(-\frac{1}{6}(\log n)^2\right)$, the norm of $\bmB_n$ is bounded
    \begin{align}\label{eq:Bnbddvershynin}
    \|\bmB_{n}\|_{2\ra 2}\lesssim_{W}\sqrt{n},
    \end{align}
    and
     $\bC_{n}$ is invertible and,
    \begin{align}\label{eq:Cnorm}
        \|\bC_{n}^{-1}\|_{2\ra 2}\lesssim_{W}\frac{1}{n}.
    \end{align}
The equation from \eqref{eq:Anresolvent} combined with the above estimates implies
\begin{align}\label{eq:taylorfirstdiff}
    \frac{\lambda_{1}(\bA_{n})}{\lambda_{1}(\bW_{n})}(\lambda_{1}(\bA_{n}) - \lambda_{1}(\bW_{n})) = \bv_{1}^{\top}\bmB_{n}\bv_{1} + \bv_{1}^{\top}\bmB_{n}\bC_{n}^{-1}\bmB_{n}\bv_{1} + O_{W}\left(\frac{1}{\sqrt{n}}\right),
\end{align}
with probability at least $1-Cn\exp\left(-\frac{1}{6}(\log n)^2\right)$.
\end{prop}
\cblue
The identity from \eqref{eq:taylorfirstdiff} follows by a Taylor expansion of \eqref{eq:Anresolvent} and using the estimates from \Cref{p:Cninvnorm}. Now, to analyse the fluctuation of $\lambda_1(\bA_n) - \lambda_1(\bW_n)$, in the next proposition, we consider simplification of the first two terms $\bv_{1}^{\top}\bmB_{n}\bv_{1}$ and $\bv_{1}^{\top}\bmB_{n}\bC_{n}^{-1}\bmB_{n}\bv_{1}$ on the righthand side of \eqref{eq:taylorfirstdiff}. We postpone its proof to Section \ref{sec:proofofvphiconc}.
\cblack
\begin{prop}\label{prop:vphiconc}
 Recall the matrices $\bmB_{n}$ and $\bC_{n}$ from \eqref{eq:defCn}.
    Denote $\phi_{1}$ the eigenfunction of $W$ corresponding to the eigenvalue $\lambda_{1}(W)$ as in Assumption \ref{assmp:assumptionkernel}, and define
    \begin{align}\label{e:defPhi}
        \bPhi_{1} = \frac{1}{\sqrt{n}}\left(\phi_{1}(U_{1}),\cdots,\phi_{1}(U_{n})\right)^{\top}
    \end{align}
    where $U_{1},\ldots, U_{n}$ are as considered in \eqref{eq:defAn}. Then
    \begin{align}\label{e:replace}
        \left|\bv_{1}^{\top}\bmB_{n}\bv_{1} - \bPhi_{1}^{\top}\bmB_{n}\bPhi_{1}\right|\lesssim_{W}\left(\frac{\log^{3}n}{\sqrt{n}}\right)^{1/2},
    \end{align}
     and
       \begin{align}\label{e:vBCBv}
        \left|\bv_{1}^{\top}\bmB_{n}\bC_{n}^{-1}\bmB_{n}\bv_{1} - \frac{1}{\lambda_{1}(W)}\int\frac{\phi_{1}^{2}(x) + \phi_{1}^{2}(y)}{2}W(x,y)(1-W(x,y))\rd x\rd y\right|\lesssim_{W}\left(\frac{\log n}{\sqrt{n}}\right)^{\frac{1}{2}}
    \end{align}
    with probability at least $1-Cn\exp\left(-\frac{1}{6}(\log n)^2\right)$.
\end{prop}
\cblue
Now, by applying the decompositions given in 
\eqref{eq:taylorfirstdiff}, \eqref{e:replace}, and \eqref{e:vBCBv}, we complete the proof by establishing the asymptotic normality of the term $\bPhi_{1}^{\top}\bmB_{n}\bPhi_{1}$. More precisely, by conditioning on \(\bU_n\) and invoking a Gaussian distributional convergence result for \(\bPhi_{1}^{\top}\bmB_{n}\bPhi_{1}\), we deduce the convergence in distribution of 
\(
\lambda_1(\bA_n) - \lambda_1(\bW_n).
\)
This result is formalized in the following proposition.
\cblack
\begin{prop}\label{p:main3}
    Fix a graphon $W$ satisfying \Cref{assmp:assumptionkernel}. We consider the adjacency matrix $\bA_{n}$ corresponding to the graphon $W$ as in \eqref{eq:defAn}, and denote its largest eigenvalue as $\lambda_1( \bA_{n})$, then there exists a set $\cA$ of $(U_1, U_2,U_3,\cdots)$ such that $\mathbb P(\cA)=1$ on the set $\cA$,
    \begin{align}\label{e:cltla}
        (\lambda_{1}(\bA_{n}) - \lambda_{1}(\bW_n))|\bU_n\dto \cN(\alpha,\sigma^2),
    \end{align}
    where
    \begin{align*}
    &\alpha=\frac{1}{\lambda_{1}(W)}\int\frac{\phi_{1}^{2}(x) + \phi_{1}^{2}(y)}{2}W(x,y)(1-W(x,y))\rd x\rd y,\\
    &\sigma^2= 2\int \phi_{1}^{2}(x)\phi_{1}^{2}(y)W(x,y)(1-W(x,y))\rd x\rd y.
    \end{align*}
\end{prop}

\section{Preliminary Results on Kernel Matrices}
\label{s:kernel}
\blue{In this section we collect some preliminary results on the kernel matrices. The proofs of these results are given in Appendix A and Appendix B. We start with the definition of the Hilbert Schmidt operator associated with a symmetric Lipschitz function.} For any symmetric Lipschitz continuous function $f:[0,1]^2\ra\mathbb R$, we associate it with an integral operator $T_f: L^2[0,1]\mapsto L^2[0,1]$:
\begin{align*}
T_f(\phi)(x)=\int_0^1 f(x,y)\phi(y)\rd y.
\end{align*}

The following lemma gives bounds on the eigenfunctions of the Hilbert Schmidt operator derived from the symmetric Lipschitz function $f:[0,1]^2\ra \bR$.
\begin{lemma}\label{lemma:efuncbdd}
    Consider a symmetric Lipschitz continuous function $f:[0,1]^2\ra\mathbb R,$ with Lipschitz constant $L_{f}$ such that $|f|\leq B_{f}$ and suppose,
    \begin{align*}
        \lambda_{1}(f)> \lambda_{2}(f)\geq\cdots\geq 0,\ \lambda_{1}^{\prime}(f)\leq\lambda_{2}^{\prime}(f)\leq\cdots\leq0
    \end{align*}
    be the eigenvalues of $T_{f}$ with corresponding eigenfunctions $\phi_{i}$ and $\phi_{i}^{\prime}$ for $i\geq 1$. \blue{Then whenever $\lambda_j(f),\lambda_j^\prime(f)\neq 0$, the following holds.}
    \begin{enumerate}
    \item[(a)] The eigenfunctions $\phi_j(x)$ and $\phi'_j(x)$ are uniformly bounded by $\frac{\blue{B_f}}{|\lambda_{j}(f)|}$ and $\frac{\blue{B_f}}{|\lambda_{j}^{\prime}(f)|}$ respectively. 
    \item[(b)] The eigenfunctions $\phi_j(x)$ and $\phi'_j(x)$ are Lipschitz with Lipschitz constant $\frac{L_{f}}{|\lambda_{j}(f)|}$ and $\frac{L_{f}}{|\lambda_{j}^{\prime}(f)|}$ respectively.
    \end{enumerate}
\end{lemma}

Consider $U_{1},U_{2},\ldots,U_{n}$ to be randomly drawn samples from the Uniform distribution on $[0,1]$. Let $U_{(1)}\leq U_{(2)}\leq \cdots\leq U_{(n)}$ be the arrangement of $\{U_{i}:1\leq i\leq n\}$ in increasing order. We consider a $n\times n$ matrix with elements $f(U_{(i)}, U_{(j)})$ and study the concentration of an operator derived from such a matrix by embedding it in $[0,1]^2$. 

   The following lemma \blue{implies} that spectrum of the above matrix $f(U_{(i)}, U_{(j)})$ is the same as that of $f(U_{i}, U_{j})$.

    \begin{lemma}\label{lemma:permmat}
    Consider a function $f:[0,1]^2\ra\mathbb{R}$ and let $U_{1},U_{2}.\ldots, U_{n}$ generated randomly from $\blue{\text{Unif}\ [0,1]}$. Then there exists a permutation matrix $\Pi_{n}$ such that,
    \begin{align}\label{eq:UorderPieq}
        \left(\left(f(U_{(i)}, U_{(j)})\right)\right)_{i\neq j} = \Pi_{n}\left(\left(f(U_{i}, U_{j})\right)\right)_{i\neq j}\Pi_{n}^{\top}
    \end{align}
    where $U_{(1)}\leq U_{(2)}\leq \cdots\leq U_{(n)}$.
    \end{lemma}

In the following lemma we show that largest eigenvalue of the sample kernel matrix is close to the largest eigenvalue of the operator $T_{f}$ with high probability.

    \begin{lemma}\label{lemma:lambda1Wnconc}
    Let $f$ be a Lipschitz continuous symmetric function such that $|f|\leq B_{f}$ and Lipschitz constant $L_{f}$. Consider $U_{1},U_{2}.\ldots, U_{n}$ generated randomly from $\blue{\text{Unif}\ [0,1]}$ and let
    \begin{align*}
        \bm F_{n}:= ((f(U_{i},U_{j})))_{i\neq j=1}^{n}.
    \end{align*}
    Further suppose,
    \begin{align*}
        \lambda_{1}(f)> \lambda_{2}(f)\geq\cdots\geq 0,\ \lambda_{1}^{\prime}(f)\leq\lambda_{2}^{\prime}(f)\leq\cdots\leq0
    \end{align*}
    be the eigenvalues of $T_{f}$ and let $\lambda_{1}(\bm F_{n})$ to be the largest eigenvalue of $\bm F_{n}$. Then,
    \begin{align*}
        \left|\frac{\lambda_{1}(\bm F_{n})}{n} - \lambda_{1}(f)\right|\lesssim_{f}\frac{\log n}{\sqrt{n}}
    \end{align*}
    with probability at least $1-8n\exp\left(-\frac{1}{6}(\log n)^2\right)$.
    \end{lemma}
    
 In the next lemma, we show that the integral operator $T_f$, can be approximated by the integral operator associated with a discrete approximation of $f$ obtained by embedding $\{f(U_{(i)}, U_{(j)})\}_{i\neq j}$ in $[0,1]^2$.

 \begin{lemma}\label{lemma:TW0minusTW}
        For a Lipschitz continuous and symmetric function $f$ such that $|f|\leq B_{f}$ with Lipschitz constant $L_{f}$ and $U_{1},U_{2}.\ldots, U_{n}$ generated randomly from $\blue{\text{Unif}\ [0,1]}$ define,
        \begin{align}\label{eq:Wncirc}
            f_{n}^{\circ}(x,y) = \sum_{i\neq j}f\left(U_{(i)}, U_{(j)}\right)\one\left\{\frac{i-1}{n}<x\leq \frac{i}{n}, \frac{j-1}{n}<y\leq \frac{j}{n}\right\}.
        \end{align}
        Then,
        \begin{align*}
            \|T_f - T_{f_{n}^{\circ}}\|_{2\ra 2}\lesssim_{f} \frac{\log n}{\sqrt{n}} + \frac{1}{\sqrt{n}}\text{ with probability at least } 1 - 4n\exp\left(-\frac{1}{6}(\log n)^2\right).
        \end{align*}
    \end{lemma}

In the next proposition, we study the matrix $\bF_n:=\{f(U_{i}, U_{j})\}_{i\neq j}$ in $[0,1]^2$. It roughly says that the largest eigenvalue of  $\bF_n$ is well separated from its other eigenvalues.

\begin{prop}\label{lemma:AnBninvertible}
    Consider a symmetric Lipschitz continuous function $f:[0,1]^2\ra\mathbb R,$ with Lipschitz constant $L_{f}$ such that $|f|\leq B_{f}$ and suppose,
    \begin{align*}
        \lambda_{1}(f)> \lambda_{2}(f)\geq\cdots\geq 0,\ \lambda_{1}^{\prime}(f)\leq\lambda_{2}^{\prime}(f)\leq\cdots\leq0
    \end{align*}
    be the eigenvalues of $T_{f}$ with corresponding eigenfunctions $\phi_{i}(x)$ and $\phi_{i}^{\prime}(x)$ for $i\geq 1$. Let, $\Phi_{1}(\bU) = (\phi_{1}(U_{1}),\ldots,\phi_{1}(U_{n}))^{\top}$ and let $\bF_{n}$ to be a $n\times n$ matrix with $0's$ on the diagonal and the $(i,j)^{th}$ entry given by $f(U_i, U_j)$ for all $1\leq i\neq j\leq n$ where $U_{1},\ldots,U_{n}$ are generated independently from $\blue{\text{Unif}\ [0,1]}$. Further consider $\blambda_n\in \bR$ such that,
    \begin{align}\label{eq:blambdanconc}
        \left|\frac{\blambda_{n}}{n} - \lambda_{1}(f)\right|\lesssim_{f}\frac{\log n}{\sqrt{n}}
    \end{align} 
    with probability at least $1-8n\exp\left(-\frac{1}{6}(\log n)^2\right)$. Define,
    \begin{align}\label{e:defXn}
        \bm{X}_{n} := \blambda_{n}\mathbb{I}_{n} + \lambda_{1}(f)\bD_{n} - (\bF_{n} - \lambda_{1}(f)(\Phi_{1}(\bU)\Phi_{1}(\bU)^{\top} - \bD_{n}))
    \end{align}
    where $\bD_{n} = \diag(\phi_{1}^{2}(U_{1}),\ldots,\phi_{1}^{2}(U_{n}))$. Then for large enough $n$, $\bm X_{n}$ is invertible and,
    \begin{align}\label{eq:normXninvbd}
        \|\bm X_{n}^{-1}\|_{2\ra 2} \leq \frac{|\lambda_{1}(f) - \lambda_{2}(f)|}{2n}.
    \end{align}
    with probability at least $1-16n\exp\left(-\frac{1}{6}(\log n)^2\right)$.
\end{prop}

\section{Proof of results from Section \ref{s:prft1}}\label{s:Kerneleig}
In this section we complete the proof of \Cref{t:main1}. We start with the postponed analysis from the end of Section \ref{s:prft1} and then provide proofs of Proposition \ref{prop:IIandIIIreplace} and Proposition \ref{p:la1decomposition}.

\subsection{Completing the proof of \Cref{t:main1}}\label{sec:completeproofkernel}
\cblue
We begin by recalling the conclusion of \Cref{prop:IIandIIIreplace} and \Cref{p:la1decomposition}. From \Cref{p:la1decomposition} recall the notations,
\begin{align*}
    T_{n,m,1} = \frac{\lambda_{1}(\sfK)}{\lambda_{1}(\sfKmat_{n})}\sum_{\ell=1}^{m}\left(\lambda_{\ell}(\sfK)\sum_{i=1}^{n}\phi_{1}(U_{i})^2\phi_{\ell}(U_{i})^2 + \lambda_{\ell}^{\prime}(\sfK)\sum_{i=1}^{n}\phi_{1}(U_{i})^2\phi_{\ell}^{\prime}(U_{i})^2\right)
\end{align*}
and,
\begin{align*}
    T_{n,m,2} = \frac{\lambda_{1}(\sfK)}{\lambda_{1}(\sfKmat_{n})}\Bigg[\sum_{\ell=2}^{m}\frac{\lambda_{\ell}(\sfK)\lambda_{1}(\sfK)}{\lambda_{1}(\sfK) - \lambda_{\ell}(\sfK)}
    & \left(\sum_{i=1}^{n}\phi_{1}(U_{i})\phi_{\ell}(U_{i})\right)^2\\
    &+ \sum_{\ell=1}^{m}\frac{\lambda_{\ell}^{\prime}(\sfK)\lambda_{1}(\sfK)}{\lambda_{1}(\sfK) - \lambda_{\ell}^{\prime}(\sfK)}\left(\sum_{i=1}^{n}\phi_{1}(U_{i})\phi_{\ell}^{\prime}(U_{i})\right)^2\Bigg].
\end{align*}
\normalsize
Now fix $\vep>0$, then there exists $m(\vep)\in \N$ such that for all $m\geq m(\vep)$ and $n\geq n(m,\vep)$, we have,
\begin{align*}
    \left|\frac{\lambda_{1}(\sfKmat_{n})}{n} - \lambda_{1}(\sfK) - \lambda_{1}(\sfK)\left[\frac{\left\|\Phi_{1}(\bU_{n})\right\|_{2}^2}{n} - 1\right] - \frac{1}{n}\left(T_{n,m,2} - T_{n,m,1}\right) - \frac{\lambda_{1}(\sfKmat_{n})}{n}t_{n}\right|\lesssim_{\sfK}\frac{\sqrt{\vep}}{n}
\end{align*}
where $|t_{n}|\lesssim_{\sfK,m} n^{-3/2}(\log n)^3$.\cblack
By Lemma \ref{l:la1est} note that $\lambda_{1}(\sfKmat_{n})/n\pto \lambda_{1}(\sfK)$. Then, for fixed $m$, by a weak law of large numbers argument, $T_{n,m,1} = o_{p}(\sqrt{n})$ as $n\ra\infty$. By the standard central limit theorem we have,
\begin{align*}
    \left(\sum_{i=1}^{n}\phi_{1}(U_{i})\phi_{\ell}(U_{i})\right)^2 = O_{p}(n), 2\leq \ell\leq m\text{ and }\left(\sum_{i=1}^{n}\phi_{1}(U_{i})\phi_{\ell}^{\prime}(U_{i})\right)^2 = O_{p}(n), 1\leq \ell\leq m.
\end{align*}
\normalsize
Together with $\lambda_{1}(\sfKmat_{n})/n\pto \lambda_{1}(\sfK)$ (from Lemma \ref{l:la1est}) we conclude $T_{n,m,2} = o_{p}(\sqrt{n})$, and 
\begin{align}\label{eq:Tnm12opsqrtn}
    -T_{n,m,1} + T_{n,m,2} = o_{p}(\sqrt{n}).
\end{align}
for fixed $m\geq 2$. Again by the standard central limit theorem
\begin{align}\label{e:laclt}
\sqrt{n} \lambda_{1}(\sfK)\left[\frac{\left\|\Phi_{1}(\bU_{n})\right\|_{2}^2}{n} - 1\right] 
= \lambda_{1}(\sfK)\frac{\sum_{i=1}^n (\phi_1(U_i)^2-1)}{\sqrt n}\rightarrow \cN\left(0,\lambda_{1}(\sfK)^2\Var\left(\phi_{1}^2(U)\right)\right).
\end{align}
\Cref{p:la1decomposition} gives that
\begin{align}\begin{split}\label{e:nla1}
    &\P\bigg(\sqrt{n}\bigg(\frac{\lambda_{1}(\sfKmat_{n})}{n}- \lambda_{1}(\sfK)\bigg)\leq t\bigg)\\
    &\leq \P\left(\sqrt{n}\lambda_{1}(\sfK)\left[\frac{\left\|\Phi_{1}(\bU_{n})\right\|_{2}^2}{n} - 1\right] + \frac{1}{\sqrt{n}}\left(-T_{n,m,1} + T_{n,m,2}\right) + \frac{\lambda_{1}(\sfKmat_{n})}{\sqrt{n}}t_{n}\leq t + O_{\sfK}\left(\frac{\sqrt{\vep}}{\sqrt{n}}\right)\right)\\
    & + 9\sqrt{\vep}.
\end{split}\end{align}
\normalsize
Taking $n\ra\infty$ and recalling \Cref{l:la1est},  \eqref{eq:Tnm12opsqrtn} and \eqref{e:laclt} it is now easy to see that,
\small
\begin{align}\label{e:nla2}
    \P\left(\sqrt{n}\lambda_{1}(\sfK)\left[\frac{\left\|\Phi_{1}(\bU_{n})\right\|_{2}^2}{n} - 1\right] + \frac{1}{\sqrt{n}}\left(-T_{n,m,1} + T_{n,m,2}\right) + \frac{\lambda_{1}(\sfKmat_{n})}{\sqrt{n}}t_{n}\leq t + O_{\sfK}\left(\frac{\sqrt{\vep}}{\sqrt{n}}\right)\right)\ra\P(Z\leq t)
\end{align}
\normalsize
where $Z\sim \cN\left(0,\lambda_{1}(\sfK)^2\Var\left(\phi_{1}^2(U)\right)\right)$. Thus \eqref{e:nla1} and \eqref{e:nla2} together imply,
\begin{align*}
    \limsup_{n\ra\infty}\P\bigg(\sqrt{n}\bigg(\frac{\lambda_{1}(\sfKmat_{n})}{n} - \lambda_{1}(\sfK)\bigg)\leq t\bigg)\leq \P(Z\leq t) + 9\sqrt{\vep}.
\end{align*}
Similarly one can show that,
\begin{align*}
    \limsup_{n\ra\infty}\P\bigg(\sqrt{n}\bigg(\frac{\lambda_{1}(\sfKmat_{n})}{n} - \lambda_{1}(\sfK)\bigg)> t\bigg)\leq \P(Z> t) + 9\sqrt{\vep}.
\end{align*}
and recalling that $\vep$ is chosen arbitrarily small, we conclude,
\begin{align*}
    \sqrt{n}\bigg(\frac{\lambda_{1}(\sfKmat_{n})}{n} - \lambda_{1}(\sfK)\bigg)\dto \cN\left(0,\lambda_{1}(\sfK)^2\Var\left(\phi_{1}^2(U)\right)\right).
\end{align*}
This finishes proof of the first statement in \Cref{t:main1}. For the second statement in \Cref{t:main1} we know that $\phi_{1}^2\equiv 1$, and note that \eqref{eq:lambdanflucexpand} simplifies to,
\begin{align}\label{eq:lambdanflucexpand2}
    \left|\frac{\lambda_{1}(\sfKmat_{n})}{n} - \lambda_{1}(\sfK)- \frac{1}{n}\left(-T_{n,m,1} + T_{n,m,2}\right) - \frac{\lambda_{1}(\sfKmat_{n})}{n}t_{n}\right|\lesssim_{\sfK}\frac{\sqrt{\vep}}{n}
\end{align}
where now
\begin{align*}
    T_{n,m,1} = n\frac{\lambda_{1}(\sfK)^2}{\lambda_{1}(\sfKmat_{n})} + \frac{\lambda_{1}(\sfK)}{\lambda_{1}(\sfKmat_{n})}\sum_{\ell=2}^{m}\lambda_{\ell}(\sfK)\sum_{i=1}^{n}\phi_{\ell}(U_{i})^2 + \frac{\lambda_{1}(\sfK)}{\lambda_{1}(\sfKmat_{n})}\sum_{\ell=1}^{m}\lambda_{\ell}^{\prime}(\sfK)\sum_{i=1}^{n}\phi_{\ell}^{\prime}(U_{i})^2
\end{align*}
and,
\small
\begin{align*}
    T_{n,m,2} := \frac{\lambda_{1}(\sfK)}{\lambda_{1}(\sfKmat_{n})}\left[\sum_{\ell=2}^{m}\frac{\lambda_{\ell}(\sfK)\lambda_{1}(\sfK)}{\lambda_{1}(\sfK) - \lambda_{\ell}^{\prime}(\sfK)}\left(\sum_{i=1}^{n}\phi_{\ell}(U_{i})\right)^2+ \sum_{\ell=1}^{m}\frac{\lambda_{\ell}^{\prime}(\sfK)\lambda_{1}(\sfK)}{\lambda_{1}(\sfK) - \lambda_{\ell}^{\prime}(\sfK)}\left(\sum_{i=1}^{n}\phi_{\ell}^{\prime}(U_{i})\right)^2\right].
\end{align*}
\normalsize
The following proposition states that $ T_{n,m,1}$ converges in probability, and $ T_{n,m,2}$ converges to a chi-square distribution. Its proof is postponed to Section \ref{sec:proofofTnm12convg}.
\begin{prop}\label{prop:Tnm12convg}
    Fix $m\geq 2$. Then,
    \begin{align*}
       T_{n,m,1}\pto  \sum_{\ell=1}^{m}\left[\lambda_{\ell}(\sfK) + \lambda_{\ell}^{\prime}(\sfK)\right]\text{ and }T_{n,m,2}\dto \sum_{\ell=2}^{m}\frac{\lambda_{\ell}(\sfK)\lambda_{1}(\sfK)}{\lambda_{1}(\sfK) - \lambda_{\ell}(\sfK)}Z_{\ell}^2 +\sum_{\ell=1}^{m}\frac{\lambda_{\ell}^{\prime}(\sfK)\lambda_{1}(\sfK)}{\lambda_{1}(\sfK) - \lambda_{\ell}^{\prime}(\sfK)}\widetilde{Z}_{\ell}^2
    \end{align*}
    where $Z_{2},\ldots,Z_{m}$ and $\widetilde{Z}_{1},\ldots, \widetilde{Z}_{m}$ are independently generated from $\cN(0,1)$.
\end{prop}

 Applying the convergences from Proposition \ref{prop:Tnm12convg} along with Slutsky's Lemma shows,
\begin{align}\begin{split}\label{e:TTdiff}
    T_{n,m,2} - T_{n,m,1}+\lambda_1(\sfK)\dto\zeta_{m} :=
    &\sum_{\ell=2}^{m}\frac{\lambda_{\ell}(\sfK)\lambda_{1}(\sfK)}{\lambda_{1}(\sfK) - \lambda_{\ell}(\sfK)}(Z_{\ell}^2-1) +\sum_{\ell=1}^{m}\frac{\lambda_{\ell}^{\prime}(\sfK)\lambda_{1}(\sfK)}{\lambda_{1}(\sfK) - \lambda_{\ell}^{\prime}(\sfK)}(\widetilde{Z}_{\ell}^2-1) \\
    & + \sum_{\ell=2}^{m}\frac{\lambda_{\ell}(\sfK)^2}{\lambda_{1}(\sfK) - \lambda_{\ell}(\sfK)} +  \sum_{\ell=1}^{m}\frac{\lambda_{\ell}^{\prime}(\sfK)^2}{\lambda_{1}(\sfK) - \lambda_{\ell}^{\prime}(\sfK)}
\end{split}\end{align}
\normalsize
where $Z_{2},\ldots,Z_{m}$ and $\widetilde{Z}_{1},\ldots, \widetilde{Z}_{m}$ are independently generated from $\cN(0,1)$. 
Recalling that $\sfK\in L_{2}[0,1]^2$ it is easy to conclude that as $m\ra\infty$,
\begin{align}\begin{split}\label{e:zetainf}
    \zeta_{m}\dto \zeta_{\infty}:= &\sum_{\ell=2}^{\infty}\frac{\lambda_{\ell}(\sfK)\lambda_{1}(\sfK)}{\lambda_{1}(\sfK) - \lambda_{\ell}(\sfK)}(Z_{\ell}^2-1) +\sum_{\ell=1}^{\infty}\frac{\lambda_{\ell}^{\prime}(\sfK)\lambda_{1}(\sfK)}{\lambda_{1}(\sfK) - \lambda_{\ell}^{\prime}(\sfK)}(\widetilde{Z}_{\ell}^2-1) \\
    & + \sum_{\ell=2}^{\infty}\frac{\lambda_{\ell}(\sfK)^2}{\lambda_{1}(\sfK) - \lambda_{\ell}(\sfK)} +  \sum_{\ell=2}^{\infty}\frac{\lambda_{\ell}^{\prime}(\sfK)^2}{\lambda_{1}(\sfK) - \lambda_{\ell}^{\prime}(\sfK)}\\
    & = \sum_{\lambda\in \sigma(\sfK)\setminus\{\lambda_{1}(\sfK)\}}\frac{\lambda_{1}(\sfK)\lambda}{\lambda_{1}(\sfK) - \lambda}(Z_{\lambda}^2-1) + \sum_{\lambda\in \sigma(\sfK)\setminus\{\lambda_{1}(\sfK)\}}\frac{\lambda^2}{\lambda_{1}(\sfK) - \lambda}
\end{split}\end{align}
where $\{Z_{\lambda}:\lambda\in \sigma(\sfK)\setminus\{\lambda_{1}(\sfK)\}\}$ are independently generated from $\cN(0,1)$.  We can rewrite \eqref{eq:lambdanflucexpand2} as
\begin{align*}
    \P\left(\lambda_{1}(\sfKmat_{n}) - n\lambda_{1}(\sfK)\leq t\right)\leq \P\left(-T_{n,m,1} + T_{n,m,2} + \lambda_{1}(\sfKmat_{n})t_{n}\leq t + O_{\sfK}(\sqrt{\vep})\right) + 9\sqrt{\vep}.
\end{align*}
Recalling $\zeta_m$ from \eqref{e:TTdiff},  we can rewrite the above expression as
\begin{align}\label{e:mid1}
    \limsup_{n\ra\infty}\P\left(\lambda_{1}(\sfKmat_{n}) - n\lambda_{1}(\sfK)\leq t\right)\leq\P\left(-\lambda_{1}(\sfK) + \zeta_{m}\leq t + O_{\sfK}(\sqrt{\vep})\right) + 9\sqrt{\vep}.
\end{align}
As $m\rightarrow \infty$, $\zeta_m \dto  \zeta_\infty$ as constructed in \eqref{e:zetainf}. 
Then recalling that $m\geq m(\vep)$ in \eqref{e:mid1} was arbitrarily chosen we get by taking $m\ra\infty$,
\begin{align*}
    \limsup_{n\ra\infty}\P\left(\lambda_{1}(\sfKmat_{n}) - n\lambda_{1}(\sfK)\leq t\right)\leq \P\left(\zeta_{\infty}\leq t + O_{\sfK}(\sqrt{\vep}) + \lambda_{1}(\sfK)\right) + 9\sqrt{\vep}.
\end{align*}
Finally recalling that $\vep>0$ was chosen arbitrarily small we get,
\begin{align*}
    \limsup_{n\ra\infty}\P\left(\lambda_{1}(\sfKmat_{n}) - n\lambda_{1}(\sfK)\leq t\right)\leq \P\left(\zeta_{\infty} - \lambda_{1}(\sfK)\leq t\right).
\end{align*}
\cblue
Similarly we can show that,
\begin{align*}
    \limsup_{n\ra\infty}\P\left(\lambda_{1}(\sfKmat_{n}) - n\lambda_{1}(\sfK)> t\right)\leq \P\left(\zeta_{\infty} - \lambda_{1}(\sfK)> t\right).
\end{align*}
\cblack
Thus we conclude that,
\begin{align*}
    \lambda_{1}(\sfKmat_{n}) - (n-1)\lambda_{1}(\sfK)\dto\zeta_{\infty}.
\end{align*}
This finishes the proof of the second statement in \Cref{t:main1}.

\subsubsection{Proof of Proposition \ref{prop:Tnm12convg}}\label{sec:proofofTnm12convg}
Recall that all the eigenfunctions of $\sfK$ are orthonormal. Then the in probability convergence of $T_{n,m,1}$ is immediate by \eqref{eq:lambda1Wnconc} and the weak law of large numbers. Next we show the in distribution convergence of $T_{n,m,2}$. For $\phi_{1}\equiv 1$ almost surely, recalling the definition of $T_{n,m,2}$ we get,
\begin{align*}
    T_{n,m,2} = \frac{n\lambda_{1}(\sfK)}{\lambda_{1}(\sfKmat_{n})}\left[\sum_{\ell=2}^{m}\frac{\lambda_{\ell}(\sfK)\lambda_{1}(\sfK)}{\lambda_{1}(\sfK) - \lambda_{\ell}^{\prime}(\sfK)}\left(\frac{1}{\sqrt{n}}\sum_{i=1}^{n}\phi_{\ell}(U_{i})\right)^2+ \sum_{\ell=1}^{m}\frac{\lambda_{\ell}^{\prime}(\sfK)\lambda_{1}(\sfK)}{\lambda_{1}(\sfK) - \lambda_{\ell}^{\prime}(\sfK)}\left(\frac{1}{\sqrt{n}}\sum_{i=1}^{n}\phi_{\ell}^{\prime}(U_{i})\right)^2\right].
\end{align*}
\normalsize
Define,
\begin{align*}
    \bm\phi_{m}(U_{i}) = (\phi_{2}(U_{i}),\cdots,\phi_{m}(U_{i}),\phi_{1}^{\prime}(U_{i}),\cdots,\phi_{m}^{\prime}(U_{i})).
\end{align*}
Then recalling the orthonormality of eigenfunctions and the multivariate CLT we find,
\begin{align*}
    \frac{1}{\sqrt{n}}\sum_{i=1}^{n}\bm\phi_{m}(U_{i})\dto \cN_{2m-1}\left(\bm 0_{2m-1}, \bI_{2m-1}\right).
\end{align*}
The proof is now completed by an application of the continuous mapping theorem, \eqref{eq:lambda1Wnconc} and Slutsky's Lemma.

\cblue
\subsection{Proof of \Cref{prop:IIandIIIreplace}}\label{sec:proofofIIandIIIreplace}
We start by recalling the master equation \eqref{eq:1AB1=12},
\begin{align*}
    \lambda_{1}(\sfK)\Phi_{1}(\bU_{n})^{\top}(\bcA_{n} - \bcB_{n})^{-1}\Phi_{1}(\bU_{n}) = 1.
\end{align*}
By basic algebra, we can reformulate \eqref{eq:1AB1=12} as
\small
\begin{align}\begin{split}\label{e:decomp}
\frac{1}{ \lambda_{1}(\sfK)}
&=\Phi_{1}(\bU_{n})^{\top}(\bcA_{n} - \bcB_{n})^{-1}\Phi_{1}(\bU_{n}) \\
&=\frac{\Phi_{1}(\bU_{n})^{\top}\Phi_{1}(\bU_{n}) }{\lambda_1(\sfKmat_n)}
+\frac{\Phi_{1}(\bU_{n})^{\top}(\lambda_1(\sfKmat_n)-\bcA_n+\bcB_n)(\bcA_n-\bcB_n)^{-1}\Phi_{1}(\bU_{n}) }{\lambda_1(\sfKmat_n)}.
\end{split}\end{align}
\normalsize
We can further decompose the last term on the righthand side of \eqref{e:decomp} as
\begin{align}\label{e:decomposeII}
\frac{\Phi_{1}(\bU_{n})^{\top}(\lambda_1(\sfKmat_n)-\bcA_n+\bcB_n)(\bcA_n-\bcB_n)^{-1}\Phi_{1}(\bU_{n}) }{\lambda_1(\sfKmat_n)}=:II+III,
\end{align}
where
\small
\begin{align}\begin{split}\label{e:defII}
&II:=\frac{\Phi_{1}(\bU_{n})^{\top}(\lambda_1(\sfKmat_n)-\bcA_n+\bcB_n)\Phi_{1}(\bU_{n}) }{\lambda^2_1(\sfKmat_n)},\\
&III:=\frac{\Phi_{1}(\bU_{n})^{\top}(\lambda_1(\sfKmat_n)-\bcA_n+\bcB_n)(\bcA_n-\bcB_n)^{-1}(\lambda_1(\sfKmat_n)-\bcA_n+\bcB_n)\Phi_{1}(\bU_{n}) }{\lambda^2_1(\sfKmat_n)}.
\end{split}\end{align}
\normalsize
In the following proposition we show that we can replace $\bcA_{n} - \bcB_{n}$ in \eqref{e:defII} by $\bcA_{n}^{(m)} - \bcB_{n}^{(m)}$ with a small error. The proof of this proposition is deferred to Section \ref{sec:proofofIIand}
\begin{prop}\label{prop:IIandIIIreplace2}
    Recall the matrices $\bcA_{n}, \bcB_{n}, \bcA_{n}^{(m)} ,\bcB_{n}^{(m)}$ from \eqref{e:defAB} and \eqref{e:defABk}.
           For any fixed $\vep>0$ there exists $m(\vep)\in\mathbb{N}$ such that for any fixed $m\geq m(\vep)$, for all $n\geq n(m,\vep)$ the following holds for $II$ and $III$ from \eqref{e:defII}.
    With probability $1-7\sqrt{\vep}$,
    \begin{align}\label{eq:IIapprox}
        \left|II - \frac{\Phi_{1}(\bU_{n})^{\top}\left(\lambda_{1}(\sfKmat_{n}) - \bcA_{n}^{(m)} + \bcB_{n}^{(m)}\right)\Phi_{1}(\bU_{n})}{\lambda_{1}(\sfKmat_{n})^2}\right|\lesssim_{\sfK}\frac{\vep}{n}
    \end{align}
    and 
    \footnotesize
    \begin{align}\begin{split}\label{e:replace2}
        \left| 
        III-\frac{\Phi_{1}(\bU_{n})^{\top}(\lambda_1(\sfKmat_n)-\bcA^{(m)}_n+\bcB^{(m)}_n)(\bcA^{(m)}_n-\bcB^{(m)}_n)^{-1}(\lambda_1(\sfKmat_n)-\bcA^{(m)}_n+\bcB^{(m)}_n)\Phi_{1}(\bU_{n}) }{\lambda^2_1(\sfKmat_{n})}\right|\lesssim_{\sfK}\frac{\sqrt{\vep}}{n}.
        \end{split}\end{align}
    \normalsize
    \end{prop}

    Then combining \Cref{prop:IIandIIIreplace} with \eqref{e:decomposeII} and \eqref{e:defII}, we can replace $(\bcA_{n} - \bcB_{n})$ in \eqref{e:decomp} by $(\bcA_{n}^{(m)} - \bcB_{n}^{(m)})$ and we get,
    \begin{align*}
        \left|\frac{1}{\lambda_{1}(\sfK)} - \Phi_{1}(\bU_{n})^{\top}\left(\bcA_{n}^{(m)} - \bcB_{n}^{(m)}\right)^{-1}\Phi_{1}(\bU_{n})\right|\lesssim_{\sfK} \frac{\sqrt{\vep}}{n}
    \end{align*}
    with probability at least $1-8\sqrt{\vep}$ for the choice of $\vep, m$ and $n$ as in \Cref{prop:IIandIIIreplace}. 
\cblack
    \subsubsection{Proof of Proposition \ref{prop:IIandIIIreplace2}}\label{sec:proofofIIand}
    Recall the matrices $\bcA_{n}, \bcB_{n}, \bcA_{n}^{(m)} ,\bcB_{n}^{(m)}$ from \eqref{e:defAB} and \eqref{e:defABk}, and  $II, III$ from \eqref{e:defII}.
    Before going ahead with the proof of Proposition \ref{prop:IIandIIIreplace2} we first state the following Lemmas \ref{lemma:inftytok} and \ref{lemma:IIItermOn}, which will be used extensively in the proof of Proposition \ref{prop:IIandIIIreplace2}. In particular, lemma \ref{lemma:inftytok} states that $(\bcA_{n} - \bcB_{n})$ and $\bcA_{n}^{(m)} - \bcB_{n}^{(m)}$ can be arbitrarily close provided we take $m$ large enough and Lemma \ref{lemma:IIItermOn} gives an efficient estimate on the inner product of $\bcA_{n}^{(m)} - \bcB_{n}^{(m)}$ with the vector $\Phi_{1}(\bU_{n})$. Both lemmas will be used to replace $(\bcA_{n} - \bcB_{n})$ in $II, III$ (recall from \eqref{e:defII}) to $\bcA_{n}^{(m)} - \bcB_{n}^{(m)}$ with arbitrarily small error.
    
    \begin{lemma}\label{lemma:inftytok}
        Recall the matrices $\bcA_{n}, \bcB_{n}, \bcA_{n}^{(m)} ,\bcB_{n}^{(m)}$ from \eqref{e:defAB} and \eqref{e:defABk}.
               For any fixed $\vep>0$ there exists $m(\vep)\in\mathbb{N}$ such that for any fixed $m\geq m(\vep)$, for all $n\geq n(m,\vep)$ we have,
            \begin{align}
                &\left|\Phi_{1}(\bU_{n})^{\top}\left((\bcA_{n} - \bcB_{n}) - (\bcA_{n}^{(m)} - \bcB_{n}^{(m)})\right)\Phi_{1}(\bU_{n})\right|\lesssim_{\sfK} \vep n,\label{eq:AnBnAkBk1}\\
                &\left\|\left((\bcA_{n} - \bcB_{n}) - (\bcA_{n}^{(m)} - \bcB_{n}^{(m)})\right)\Phi_{1}(\bU_{n})\right\|_{2}\lesssim_{\sfK}\vep n,\label{eq:AnBnAkBk2}\\
                &\left\|(\bcA_{n} - \bcB_{n}) - (\bcA_{n}^{(m)} - \bcB_{n}^{(m)})\right\|_{2\ra 2}\leq \vep n
                \label{eq:AnBnAkBk3}
            \end{align}
            with probability at least $1-\vep$.
    \end{lemma}
        \begin{lemma}\label{lemma:IIItermOn}
        Recall the matrices $\bcA_{n}, \bcB_{n}, \bcA_{n}^{(m)} ,\bcB_{n}^{(m)}$ from \eqref{e:defAB} and \eqref{e:defABk}.
            For any fixed $\vep>0$, there exists $m(\vep)\in \N$ such that for all $m\geq m(\vep)$, for all $n\geq n(m,\vep)$ we have,
            \begin{align*}
                \left\|\left(\lambda_{1}(\sfKmat_{n}) - \bcA_{n}^{(m)} + \bcB_{n}^{(m)}\right)\Phi_{1}(\bU_{n})\right\|_{2}\lesssim_{\sfK} \frac{n}{\vep^{1/4}} + \sqrt{n}
            \end{align*}
            with probability at least $1-\sqrt{\vep}$.
        \end{lemma}
    
    The proof of Lemmas \ref{lemma:inftytok} and \ref{lemma:IIItermOn} are given in sections \ref{sec:proofofinftytok} and \ref{section:proofofIIItermOn}.
    Having stated the above lemmas we now proceed with the proof of Proposition \ref{prop:IIandIIIreplace2}. First we prove the approximation to the term $II$ in \eqref{eq:IIapprox}. Notice that by Lemma \ref{lemma:lambda1Wnconc},
    \begin{align}\label{eq:lambda1Wnconc}
        \left|\frac{\lambda_{1}(\sfKmat_{n})}{n} - \lambda_{1}(\sfK)\right|\lesssim_{\sfK}\frac{\log n}{\sqrt{n}}
    \end{align}
    with probability $1-8n\exp\left(-\frac{1}{6}(\log n)^2\right)$. Then by Lemma \ref{lemma:inftytok} and \eqref{eq:lambda1Wnconc} for any $\vep>0$ and $m\geq m(\vep)$ there exists $n(m,\vep)\geq 1$ such that for all $n\geq n(m,\vep)$,
    \begin{align*}
        \left|II - \frac{\Phi_{1}(\bU_{n})^{\top}\left(\lambda_{1}(\sfKmat_{n}) - \bcA_{n}^{(m)} + \bcB_{n}^{(m)}\right)\Phi_{1}(\bU_{n})}{\lambda_{1}(\sfKmat_{n})^2}\right|\lesssim_{\sfK} \frac{\vep}{n}
    \end{align*}
    with probability at least $1-8n\exp\left(-\frac{1}{6}(\log n)^2\right) - \vep$. Now to approximate $III$ we first consider the following bound,
    \footnotesize
    \begin{align}\label{eq:replace}
        &\phantom{{}={}}\left| 
        III-\frac{\Phi_{1}(\bU_{n})^{\top}(\lambda_1(\sfKmat_n)-\bcA^{(m)}_n+\bcB^{(m)}_n)(\bcA^{(m)}_n-\bcB^{(m)}_n)^{-1}(\lambda_1(\sfKmat_n)-\bcA^{(m)}_n+\bcB^{(m)}_n)\Phi_{1}(\bU_{n}) }{\lambda^2_1(\sfKmat_{n})}\right|\nonumber\\
        &\leq \left|\frac{\Phi_{1}(\bU_{n})^{\top}(\lambda_1(\sfKmat_n)-\bcA^{(m)}_n+\bcB^{(m)}_n)((\bcA_n-\bcB_n)^{-1}-(\bcA^{(m)}_n-\bcB^{(m)}_n)^{-1})(\lambda_1(\sfKmat_n)-\bcA^{(m)}_n+\bcB^{(m)}_n)\Phi_{1}(\bU_{n}) }{\lambda^2_1(\sfKmat_{n})}\right|\nonumber\\
        &+ \left|\frac{\Phi_{1}(\bU_{n})^{\top}(\bcA_n-\bcB_n-\bcA^{(m)}_n+\bcB^{(m)}_n)(\bcA^{(m)}_n-\bcB^{(m)}_n)^{-1}(\lambda_1(\sfKmat_n)-\bcA^{(m)}_n+\bcB^{(m)}_n)\Phi_{1}(\bU_{n}) }{\lambda^2_1(\sfKmat_{n})}\right|\nonumber\\
        &+\left|\frac{\Phi_{1}(\bU_{n})^{\top}(\lambda_1(\sfKmat_n)-\bcA_n+\bcB_n)(\bcA^{(m)}_n-\bcB^{(m)}_n)^{-1}(\bcA_n-\bcB_n-\bcA^{(m)}_n+\bcB^{(m)}_n)\Phi_{1}(\bU_{n}) }{\lambda^2_1(\sfKmat_{n})}\right|.
    \end{align}
    \normalsize
    Because of the bounds from Lemma \ref{lemma:inftytok} it is now enough to find the error of approximating $\bcA_{n} - \bcB_{n}$ by $\bcA_{n}^{(m)} - \bcB_{n}^{(m)}$ and a bound on the inner product of $\bcA_{n} - \bcB_{n}$ with the vector $\Phi_{1}(\bU_{n})$. With that goal we first find the approximation error. 
    \cred
    Towards that we define the following finite rank kernel $\sfK_m$,
    \begin{align}\label{eq:defKm}
        \sfK_{m}(x,y) = \sum_{\ell=1}^{m}\lambda_{\ell}(\sfK)\phi_{\ell}(x)\phi_{\ell}(y) + \sum_{\ell=1}^{m}\lambda_{\ell}^{\prime}(\sfK)\phi_{\ell}^{\prime}(x)\phi_{\ell}^{\prime}(y).
    \end{align}
    We denote the corresponding kernel matrix of $\sfK_m$ as
    \begin{align}\label{eq:Knmmat}
        \sfKmat_{n,m}:= ((\sfK_{m}(U_{i},U_{j})\delta_{i\neq j}))_{i,j=1}^{n}.
    \end{align}
    Recalling the definition of $\bcA_n^{(m)}, \bcB_n^{(m)}$ from \eqref{e:defABk} shows
    \small
    \begin{align}\label{eq:AnmBnmdiff}
        \bcA_n^{(m)} - \bcB_n^{(m)} = \lambda_1(\sfKmat_n) + \lambda_1(\sfK)\bD_{n,1} - \left(\sfKmat_{n,m} - \lambda_1(\sfK)\left(\Phi_1(\bU_n)\Phi_1(\bU_n)^\top - \bD_{n,1}\right)\right).
    \end{align}
    \normalsize
    \cblack
    Hence we can now apply Proposition \ref{lemma:AnBninvertible} and notice that for given $m\geq 2$ there exists $n(m)\geq 1$ such that for all $n\geq n(m)$, $\bcA_{n}^{(m)} - \bcB_{n}^{(m)}$ is invertible and,
    \begin{align}\label{eq:AnkBnkinvnormbd}
        \left\|\left(\bcA_{n}^{(m)} - \bcB_{n}^{(m)}\right)^{-1}\right\|_{2\ra 2}\leq \frac{|\lambda_{1}(\sfK_{m}) - \lambda_{2}(\sfK_{m})|}{2n} =   \frac{|\lambda_{1}(\sfK) - \lambda_{2}(\sfK)|}{2n}
    \end{align}
    with probability at least $1-16n\exp\left(-\frac{1}{6}(\log n)^2\right)$. As an easy consequence of Lemma \ref{lemma:inftytok}, Proposition \ref{lemma:AnBninvertible} and \eqref{eq:AnkBnkinvnormbd} we get that for any $\vep>0$ and $m\geq m(\vep)$ there exists $n(m,\vep)\geq 1$ such that for all $n\geq n(m,\vep)$,
    \begin{align}\label{eq:AnBnconcAnkBnk}
        \left\|\left(\bcA_{n} - \bcB_{n}\right)^{-1} - \left(\bcA_{n}^{(m)} - \bcB_{n}^{(m)}\right)^{-1}\right\|_{2\ra 2}\lesssim_{\sfK} \frac{\vep}{n}
    \end{align}
    with probability at least $1-32n\exp\left(-\frac{1}{6}(\log n)^2\right) - \vep$, giving us the approximation error. Now combining Lemma \ref{lemma:inftytok} and Lemma \ref{lemma:IIItermOn} we get that for a given $\vep$, there exists $m(\vep)\in\N$ such that for all $m\geq m(\vep)$ and for all $n\geq n(m, \vep)$,
    \begin{align*}
        \left\|\left(\lambda_{1}(\sfKmat_{n}) - \bcA_{n} + \bcB_{n}\right)\Phi_{1}(\bU_{n})\right\|_{2}\lesssim_{\sfK}\frac{n}{\vep^{1/4}} + \sqrt{n} + \vep n
    \end{align*}
    with probability at least $1-\sqrt{\vep} - \vep$, which gives us a bound on the inner product. Now recall the bound from \eqref{eq:replace}. Then by Lemma \ref{lemma:inftytok}, \eqref{eq:lambda1Wnconc}, \eqref{eq:AnkBnkinvnormbd}, \eqref{eq:AnBnconcAnkBnk} and Lemma \ref{lemma:IIItermOn} for small enough $\vep>0$ and $m\geq m(\vep)$ there exists $n(m,\vep)\geq 1$ such that for all $n\geq n(m,\vep)$,
    \footnotesize
    \begin{align*}
        \left| III-\frac{\Phi_{1}(\bU_{n})^{\top}(\lambda_1(\sfKmat_n)-\bcA^{(m)}_n+\bcB^{(m)}_n)(\bcA^{(m)}_n-\bcB^{(m)}_n)^{-1}(\lambda_1(\sfKmat_n)-\bcA^{(m)}_n+\bcB^{(m)}_n)\Phi_{1}(\bU_{n}) }{\lambda^2_1(\sfKmat_{n})}\right|\lesssim_{\sfK} \frac{\sqrt{\vep}}{n}
    \end{align*}
    \normalsize
    with probability at least $1-2\sqrt{\vep} - 3\vep -Cn\exp(-\frac{1}{6}(\log n)^2)$. Choosing $n(m,\vep)$ large enough \blue{this bounds} the approximation errors from $II$ (from \eqref{eq:AnkBnkinvnormbd}) and $III$ happens with probability at least $1-7\sqrt{\vep}$, completing the proof.
    
    \subsubsection{Proof of Lemma \ref{lemma:inftytok}}\label{sec:proofofinftytok}
    
    Consider the modified function,
    \begin{align}\label{eq:defWtildeminusk}
        \widetilde{\sfK}_{-m}(x,y) := \sfK(x,y) - \sum_{\ell=1}^{m}\lambda_{1}(\sfK)\phi_{1}(x)\phi_{1}(y) - \sum_{\ell=1}^{m}\lambda_{1}^{\prime}(\sfK)\phi_{1}^{\prime}(x)\phi_{1}^{\prime}(y)
    \end{align}
    where the equality is in $L_{2}$ sense. Note that by definition,
    \begin{align*}
        \widetilde{\sfK}_{-m}(x,y) = \sum_{\ell>m}\lambda_{1}(\sfK)\phi_{1}(x)\phi_{1}(y) + \sum_{\ell>m}\lambda_{1}^{\prime}(\sfK)\phi_{1}^{\prime}(x)\phi_{1}^{\prime}(y).
    \end{align*}
    \blue{Orthonormality} of \blue{eigenfunctions} implies,
    \begin{align}\label{eq:orthotildeWk}
        \int\widetilde{\sfK}_{-m}(x,y)\phi_{1}(y)\rd y = 0\text{ for almost every }x\in [0,1].
    \end{align}
    Also note that,
    \begin{align}\label{eq:AnBnAkBktildeW}
        (\bcA_{n} - \bcB_{n}) - (\bcA_{n}^{(m)} - \bcB_{n}^{(m)}) = \widetilde{\sfKmat}_{n,-m},\quad \widetilde{\sfKmat}_{n,-m}:=\left(\left(\widetilde{\sfK}_{-m}(U_{i},U_{j})\delta_{i\neq j}\right)\right)_{i,j=1}^{n}.
    \end{align}
    Now fix $\vep>0$ and consider $m(\vep)\geq 1$ such that,
    \begin{align}\label{eq:bddsqeigensum}
        \sum_{\ell> m(\vep)}\lambda_{\ell}^{2}(\sfK) + \sum_{\ell>m(\vep)}\lambda_{\ell}^{\prime^2}(\sfK)\leq \frac{\vep^3}{6}.
    \end{align}
    In particular, this implies that $|\lambda_{m}(\sfK)|\leq \vep/2$ for all $m\geq m(\vep)$ and $\|\widetilde{\sfK}_{-m}\|_{2\ra 2}\leq \vep/2$. Now note that for all $m\geq m(\vep)$ and $n\geq 2$, by \eqref{eq:orthotildeWk} and \eqref{eq:bddsqeigensum} we get,
    \begin{align*}
        \E\left[\left(\Phi_{1}(\bU_{n})^{\top}\widetilde{\sfKmat}_{n,-m}\Phi_{1}(\bU_{n})\right)^2\right] 
        & = \E\left[\left(\sum_{i\neq j}\phi_{1}(U_{i})\phi_{1}(U_{j})\widetilde{\sfK}_{-m}(U_{i},U_{j})\right)^2\right]\\
        & \blue{= 2\sum_{i\neq j}\E\left[\phi_{1}^2(U_{i})\phi_{1}^2(U_{j})\widetilde{\sfK}_{-m}^2(U_{i}, U_{j})\right]}\\
        & \blue{\leq 2n^2\int \phi_{1}^2(x)\phi_{1}^2(y)\widetilde{\sfK}_{-m}^2(x,y)\rd x\rd y}\leq \frac{\vep^3n^2}{3\lambda_{1}(\sfK)^4}
    \end{align*}
    where the last inequality follows by the bounds from Lemma \ref{lemma:efuncbdd} replacing $f$ by $\sfK$. Thus, Markov Inequality along with \eqref{eq:AnBnAkBktildeW} shows,
    \begin{align*}
        \P\left(\left|\Phi_{1}(\bU_{n})^{\top}\left(\bcA_{n} - \bcB_{n} - \bcA_{n}^{(m)} - \bcB_{n}^{(m)}\right)\Phi_{1}(\bU_{n})\right|>\frac{\vep n}{\lambda_{1}(\sfK)^2}\right)\leq \frac{\vep}{3}
    \end{align*}
    for all $m\geq m(\vep)$ and $n\geq 2$ completing the proof of \eqref{eq:AnBnAkBk1}. Observe that it is enough to bound $\|\widetilde{\sfKmat}_{n,-m}\Phi_{1}(\bU_{n})\|_{2}$ in $L_{2}$ to show \eqref{eq:AnBnAkBk2}. Note that,
    \begin{align}\label{eq:WnminuskPhi1}
        \E\left[\|\widetilde{\sfKmat}_{n,-m}\Phi_{1}(\bU_{n})\|_{2}^2\right] 
        & = \sum_{i=1}^{n}\E\left[\left(\sum_{j=1, j\neq i}^{n}\widetilde{\sfK}_{-m}(U_{i}, U_{j})\phi_{1}(U_{j})\right)^2\right]\nonumber\\
        & = \sum_{i=1}^{n}\E\left[\sum_{\substack{j,\ell=1,\\j,\ell\neq i}}^{n}\phi_{1}(U_{j})\widetilde{\sfK}_{-m}(U_{i}, U_{j})\phi_{1}(U_{\ell})\widetilde{\sfK}_{-m}(U_{i}, U_{\ell})\right]\nonumber\\
        & = \sum_{i\neq j}\E\left[\phi_{1}^2(U_{j})\widetilde{\sfK}_{-m}^2(U_{i}, U_{j})\right]\leq \frac{\vep^3n^2}{3\lambda_{1}(\sfK)^2},
    \end{align}
    where the last inequality once again follows by the bounds from Lemma \ref{lemma:efuncbdd} replacing $f$ by $\sfK$ and \eqref{eq:bddsqeigensum}. By Markov inequality,
    \begin{align*}
        \P\left(\left\|\left(\bcA_{n} - \bcB_{n} - \bcA_{n}^{(m)} - \bcB_{n}^{(m)}\right)\Phi_{1}(\bU_{n})\right\|_{2}>\frac{\vep n}{\lambda_{1}(\sfK)}\right)\leq \frac{\vep}{3}
    \end{align*}
    for all $m\geq m(\vep)$ and $n\geq 2$, which shows the bound from \eqref{eq:AnBnAkBk2}. For the proof of \eqref{eq:AnBnAkBk3} notice that by definition there exists constants $L(m,\sfK)$ and $B(m,\sfK)$ such that $|\widetilde{\sfK}_{-m}|\leq B(m,\sfK)$ and \blue{$\widetilde{\sfK}_{-m}$ is Lipschitz} with Lipschitz constant $L(m,\sfK)$. Then by Lemma \ref{lemma:TW0minusTW} we get,
    \begin{align*}
        \left|\frac{1}{n}\left\|\widetilde{\sfKmat}_{n,-m}\right\|_{2\ra 2} - \|\widetilde{\sfK}_{-m}\|_{2\ra 2}\right|\lesssim_{m,\sfK} \frac{2L(m,W)\log n}{\sqrt{n}} + \frac{2B(m,W)}{\sqrt{n}}
    \end{align*}
    with probability at least $1-4n\exp\left(-\frac{1}{6}(\log n)^2\right)$. There exists $n(m,\vep)$ such that,
    \begin{align*}
        \frac{2L(m,W)\log n}{\sqrt{n}} + \frac{2B(m,W)}{\sqrt{n}}\leq \frac{\vep}{2}\text{ and }4n\exp\left(-\frac{1}{6}(\log n)^2\right)\leq \frac{\vep}{3}\text{ for all }n\geq n(m,\vep).
    \end{align*}
    Then for all $n\geq n(m,\vep)$ we have,
    \begin{align*}
        \P\left(\left\|\bcA_{n} - \bcB_{n} - \bcA_{n}^{(m)} + \bcB_{n}^{(m)}\right\|_{2\ra 2}>n\vep\right)\leq \frac{\vep}{3}.
    \end{align*}
    which completes the proof of \eqref{eq:AnBnAkBk3}.
    
    \subsubsection{Proof of Lemma \ref{lemma:IIItermOn}}\label{section:proofofIIItermOn}
    \blue{Recall the finite rank kernel $\sfK_m$ from \eqref{eq:defKm}, the corresponding kernel matrix $\sfKmat_{n,m}$ from \eqref{eq:Knmmat}. The identity \eqref{eq:AnmBnmdiff} can be written as,
    \begin{align*}
        \lambda_{1}(\sfKmat_{n}) - \bcA_{n}^{(m)} - \bcB_{n}^{(m)} = \sfKmat_{n,m} - \lambda_{1}(\sfK)\Phi_{1}(\bU_{n})\Phi_1(\bU_{n})^{\top}.
    \end{align*}}
    Define $\widetilde{\sfK}_{n,m}(x,y) := \sfK_{m}(x,y) - \lambda_{1}(\sfK)\phi_{1}(x)\phi_{1}(y)$. Then it is easy to observe that,
    \begin{align}\label{eq:lambda1AnBneqWnkDn1}
        \lambda_{1}(\sfKmat_{n}) - \bcA_{n}^{(m)} - \bcB_{n}^{(m)} = \widetilde{\sfKmat}_{n,m} - \bD_{n,1}
    \end{align}
    where,
    \begin{align*}
        \widetilde{\sfKmat}_{n,m} = \left(\left(\widetilde{\sfK}_{n,m}(U_{i},U_{j})\right)\right)_{i\neq j}.
    \end{align*}
    Following arguments similar to \eqref{eq:WnminuskPhi1} we get,
    \begin{align*}
        \E\left[\left\|\widetilde{\sfKmat}_{n,m}\Phi_{1}(\bU_{n})\right\|_{2}^{2}\right] = \sum_{i\neq j}\E\left[\phi_{1}(U_{j})^2\widetilde{\sfK}_{n,m}(U_{i},U_{j})^2\right].
    \end{align*}
    Now recall $\widetilde{\sfK}_{-m}$ from \eqref{eq:defWtildeminusk} and choose $k(\vep)\in \N$ such that for all $m\geq m(\vep)$, $\|\widetilde{\sfK}_{-m}\|_{2}\leq \vep$. Noting that, $\widetilde{\sfK}_{n,m}(x,y)  = \sfK(x,y) - W_{-m}(x,y) - \lambda_{1}\phi_{1}(x)\phi_{1}(y)$ we get,
    \begin{align}\label{eq:markov1bd}
        \E\left[\left\|\widetilde{\sfKmat}_{n,m}\Phi_{1}(\bU_{n})\right\|_{2}^{2}\right]\leq 8n^2\frac{\lambda_{1}(\sfK)^2(1+\vep^2) + 1}{\lambda_{1}(\sfK)^4}\leq 8n^2\frac{2\lambda_{1}(\sfK)^2+1}{\lambda_{1}(\sfK)^4}.
    \end{align}
    Recalling the bound on $|\phi_{1}|$ from Lemma \ref{lemma:efuncbdd}, replacing $f$ by $\sfK$ shows,
    \begin{align}\label{eq:markov2}
        \left\|\left(\widetilde{\sfKmat}_{n,m} - \bD_{n,1}\right)\Phi_{1}(\bU_{n})\right\|_{2}\leq \left\|\widetilde{\sfKmat}_{n,m}\Phi_{1}(\bU_{n})\right\|_{2} + \frac{\sqrt{n}}{|\lambda_{1}(\sfK)|^3}.
    \end{align}
    An easy application of Markov inequality along with \eqref{eq:markov1bd} and \eqref{eq:markov2} shows,
    \begin{align*}
        \P\bigg(\left\|\left(\widetilde{\sfKmat}_{n,m} - \bD_{n,1}\right)\Phi_{1}(\bU_{n})\right\|_{2}>\frac{n}{\vep^{1/4}}\sqrt{8\frac{2\lambda_{1}(\sfK)^2 + 1}{\lambda_{1}(\sfK)^4}} + \frac{\sqrt{n}}{|\lambda_{1}(\sfK)|^3}\bigg)\leq \sqrt{\vep}.
    \end{align*}
    The proof is now completed by recalling \eqref{eq:lambda1AnBneqWnkDn1}.

    \subsection{Proof of Proposition \ref{p:la1decomposition}}\label{sec:proofofla1decomposition}
    We recall from \eqref{eq:reducedtok}, the following holds with probability $1-8\sqrt{\varepsilon}$,
    \begin{align}\label{eq:reducedtok2}
        \left|\frac{1}{\lambda_{1}(\sfK)} - \Phi_{1}(\bU_{n})^{\top}\left(\bcA_{n}^{(m)} - \bcB_{n}^{(m)}\right)^{-1}\Phi_{1}(\bU_{n})\right|\lesssim_{\sfK} \frac{\sqrt{\vep}}{n}.
    \end{align}
    
    \Cref{p:la1decomposition} infact follows from a special case of the following general proposition. \blue{For notational brevity we resuse notations, which will be clear from the context. The proof of this proposition is postponed to Section \ref{sec:proofoffinitelambdaexpansion}.}
    \begin{prop}\label{prop:finitelambdaexpansion}
        Let $f:[0,1]^2\ra\mathbb R$ be a symmetric Lipschitz function with Lipschitz constant $L_{f}$ and $|f|\leq B_{f}$ and suppose $f$ has $k\geq 1$ many non-zero eigenvalues $\lambda_{1}(f)>\lambda_{2}(f)\geq\lambda_{3}(f)\geq\cdots\geq\lambda_{k}(f)$ with corresponding eigenfunctions $\phi_{i}, 1\leq i\leq k$. Let $\bF_{n} = ((f(U_{i}, U_{j})))_{i\neq j}$ with largest eigenvalue $\lambda_{1}(\bF_{n})$. Define,
        \begin{align*}
            \Phi_{i}(\bU_{n}):= \left(\phi_{i}(U_{1}),\ldots,\phi_{i}(U_{n})\right)^\top,\text{ and }\bD_{n,i} = \diag(\phi_{i}^{2}(U_{1}),\ldots,\phi_{i}^{2}(U_{n})), 1\leq i\leq k.
        \end{align*}
        Consider $\blambda_{n}\in\bR$ satisfying 
         \begin{align}\label{eq:blambdanconc2}
            \left|\frac{\blambda_{n}}{n} - \lambda_{1}(f)\right|\lesssim_{f}\frac{\log n}{\sqrt{n}}
        \end{align}
        with probability at least $1 - 8n\exp\left(-\frac{1}{6}(\log n)^2\right)$. Define,
        \begin{align}\label{e:cAcB}
            \bcA_{n}:= \blambda_{n}\bI_{n} + \sum_{\ell=1}^{k}\lambda_{\ell}(f)\bD_{n,\ell}\text{ and }\bcB_{n}:= \sum_{\ell=2}^{k}\lambda_{\ell}(f)\Phi_{\ell}(\bU_{n})\Phi_{\ell}(\bU_{n})^{\top}.
        \end{align}
        Then for large enough $n, (\bcA_{n} - \bcB_{n})$ is invertible with probability at least $1-12n\exp\left(-\frac{1}{6}(\log n)^2\right)$ and for $|t_{n}|\lesssim_{f}k^{2}\frac{(\log n)^3}{n^{3/2}}$,
        \footnotesize
        \begin{align}
            \lambda_{1}(f)\Phi_{1}(\bU_{n})^{\top}\left(\bcA_{n} - \bcB_{n}\right)^{-1}\Phi_{1}(\bU_{n})
            & = \frac{\lambda_{1}(f)}{\blambda_{n}}\|\Phi_{1}(\bU_{n})\|_{2}^2 -\frac{\lambda_{1}(f)}{\blambda_{n}^2}\sum_{\ell=1}^{k}\lambda_{\ell}(f)\sum_{i=1}^{n}\phi_{1}(U_{i})^2\phi_{\ell}(U_{i})^2\nonumber\\
            &+ \frac{\lambda_{1}(f)}{\blambda_{n}^2}\sum_{j=2}^{k}\frac{\lambda_{j}(f)\lambda_{1}(f)}{\lambda_{1}(f) - \lambda_{j}(f)}\left(\sum_{i=1}^{n}\phi_{1}(U_{i})\phi_{j}(U_{i})\right)^2 + t_{n}\label{eq:finitekexpansion}
        \end{align}
        \normalsize
        with probability at least $1-Cnk\exp\left(-\frac{1}{6}(\log n)^2\right)$, where $C>0$ is a universal constant.
    \end{prop}
    \begin{proof}[Proof of \Cref{p:la1decomposition}]
    We take $f=\sfK$, $k=2m$, $\bm \lambda_n=\lambda_1(\bm\sfK_n)$, 
    $(\lambda_1(f)>\lambda_2(f)\geq \cdots \geq\lambda_k(f))=(\lambda_1(\sfK)>\lambda_2(\sfK)\geq \cdots \geq \lambda_m(\sfK)\geq\lambda_m'(\sfK)\geq \cdots\geq \lambda_2'(\sfK)\geq \lambda_1'(\sfK))$,
    then $\bm \cA_n, \bm \cB_n$ in \Cref{e:cAcB} are $\bm \cA^{(m)}_n, \bm \cB^{(m)}_n$ in \eqref{eq:reducedtok}. Then \eqref{eq:lambda1Wnconc}
    verifies that \eqref{eq:blambdanconc2} holds with probability  $1-8n\exp\left(-\frac{1}{6}(\log n)^2\right)$. \Cref{prop:finitelambdaexpansion} together with \eqref{eq:reducedtok2} gives that the following holds with probability $1-9\sqrt{\varepsilon}$,
    \begin{align}\begin{split}\label{e:la1est}
      \frac{\sqrt{\varepsilon}}{n}
      &\gtrsim_\sfK  \left|\frac{\lambda_1(\bm\sfK_n)}{n}-\frac{\lambda_1(\sfK)\lambda_1(\bm\sfK_n)}{n} \Phi_{1}(\bU_{n})^{\top}\left(\bcA_{n}^{(m)} - \bcB_{n}^{(m)}\right)^{-1}\Phi_{1}(\bU_{n})\right|\\
      &=\left|\frac{\lambda_1(\bm\sfK_n)}{n}-
       \frac{\lambda_1(\sfK)\|\Phi_{1}(\bU_{n})\|_{2}^2}{n} +\frac{1}{n}T_{n,m,1} -\frac{1}{n}T_{n,m,2}+\frac{\lambda_1(\bm\sfK_n)t_{n}}{n}
      \right|
    \end{split}  \end{align}
    where $|t_{n}|\lesssim_{\sfK}m^{2}\frac{(\log n)^3}{n^{3/2}}$. The claim of \Cref{p:la1decomposition} follows from rearranging \eqref{e:la1est}.
    \end{proof}
    
    \subsubsection{Proof of \Cref{prop:finitelambdaexpansion}}\label{sec:proofoffinitelambdaexpansion}

        Without loss of generality we can consider $B_{f}=1$. Observe that by definition,
        \begin{align*}
            \bcA_{n} - \bcB_{n} = \blambda_{n} + \lambda_{1}(f)\bD_{n,1} - (\bF_{n} - \lambda_{1}(f)(\Phi_{1}(\bU_{n})\Phi_{1}(\bU_{n})^\top - \bD_{n,1})).
        \end{align*}
        Then by Proposition \ref{lemma:AnBninvertible} $\bcA_{n} - \bcB_{n}$ is invertible with probability $1-16n\exp\left(-\frac{1}{6}(\log n)^2\right)$. By definition, 
        \begin{align*}
            \bcB_{n} = \bV\bLambda\bV^{\top}
        \end{align*}
        where,
        \begin{align}\label{eq:defVLambda}
            \bV = [\Phi_{2}(\bU_{n}),\cdots, \Phi_{k}(\bU_{n})]\text{ and }\bLambda = \diag(\lambda_{2}(f),\cdots,\lambda_{k}(f)),
        \end{align}
        and by  Woodbury's formula we have,
        \begin{align}\label{eq:Woodbury}
            (\bcA_{n} - \bcB_{n})^{-1} = \bcA_{n}^{-1} - \bcA_{n}^{-1}\bV(-\bLambda^{-1} + \bV^{\top}\bcA_{n}^{-1}\bV)^{-1}\bV^{\top}\bcA_{n}^{-1}.
        \end{align}
        To proceed with the proof \blue{of} Proposition \ref{prop:finitelambdaexpansion} we first provide a Taylor expansion of $\bcA_{n}^{-1}$ and use the dominating terms to provide an expression of the quadratic form up to a negligible error. With that goal in mind note that,
        \begin{align}\label{eq:defM}
            \bcA_{n} = \blambda_{n}\left[\bI_{n} + \frac{\bm M_{n}}{\blambda_{n}}\right]\text{ where }\bM_{n}:= \sum_{\ell=1}^{k}\lambda_{\ell}(f)\bD_{n,\ell}.
        \end{align}
        In the following lemma we provide a bound on the norm of $\bM_{n}$ which in particular shows that we can have a Taylor series expansion of $\bcA_{n}^{-1}$.
        \begin{lemma}\label{lemma:normMbdd}
            For the $n\times n$ matrix $\bM_{n}$ defined in \eqref{eq:defM}, there exists $n_{1}\in\mathbb{N}$ such that,
            \begin{align*}
                \left\|\frac{\bM_{n}}{\blambda_{n}}\right\|_{2\ra 2}<1
            \end{align*}
            with probability at least $1 - 8n\exp\left(-\frac{1}{6}(\log n)^2\right)$ for all $n\geq n_1$.
        \end{lemma}
        The proof of Lemma \ref{lemma:normMbdd} is given in Section C.1. By Lemma \ref{lemma:normMbdd}, and the Taylor expansion we get,
        \footnotesize
        \begin{align}\label{eq:taylorexpand}
            \Phi_{1}(\bU_{n})^{\top}\bcA_{n}^{-1}\Phi_{1}(\bU_{n}) = \frac{\|\Phi_{1}(\bU_{n})\|_{2}^2}{\blambda_{n}} - \frac{\Phi_{1}(\bU_{n})^{\top}\bM_{n}\Phi_{1}(\bU_{n})}{\blambda_{n}^2} + \underbrace{\sum_{\ell=2}^{\infty}(-1)^{\ell}\frac{\Phi_{1}(\bU_{n})^{\top}\bM_{n}^{\ell}\Phi_{1}(\bU_{n})}{\blambda_{n}^{\ell+1}}}_{L_{n}}
        \end{align}
        \normalsize
        with probability at least $1 - 8n\exp\left(-\frac{1}{6}(\log n)^2\right)$ for all large \blue{enough} $n$. Next we show that first two terms in the expansion of \eqref{eq:finitekexpansion} are contributed by the first two terms of \eqref{eq:taylorexpand}, while the third term is negligible with high probability. Note that,
        \begin{align*}
            \Phi_{1}(\bU_{n})^{\top}\bM_{n}\Phi_{1}(\bU_{n}) = \sum_{\ell=1}^{k}\lambda_{\ell}(f)\sum_{i=1}^{n}\phi_{1}(U_{i})^2\phi_{\ell}(U_{i})^2
        \end{align*}
        which contributed the second term in the expansion \eqref{eq:finitekexpansion}. Next we show that $L_{n}$ in \eqref{eq:taylorexpand} is negligible. By the bounds from Lemma \ref{lemma:efuncbdd} it is easy to conclude that,
        \begin{align}\label{eq:bddMii}
            |\bM_{n}(i,i)|\leq \sum_{\ell=1}^{k}|\lambda_{\ell}(f)|^{-1}\text{ for all }1\leq i\leq n.
        \end{align}
        Hence,
        \begin{align*}
            \left|\Phi_{1}(\bU_{n})^{\top}\bM_{n}^{\ell}\Phi_{1}(\bU_{n})\right|\leq \|\Phi_{1}(\bU_{n})\|_{2}^2\|\bM_{n}^\ell\|_{2\ra 2}\leq \|\Phi_{1}(\bU_{n})\|_{2}^2\left(\sum_{\ell=1}^{k}|\lambda_{\ell}(f)|^{-1}\right)^{\ell}.
        \end{align*}
        Then recalling \eqref{eq:blambdanconc2} and bounds from Lemma \ref{lemma:efuncbdd} shows,
        \begin{align*}
            |L_{n}|=\left|\sum_{\ell=2}^{\infty}(-1)^{\ell}\frac{\Phi_{1}(\bU)^{\top}\bM_{n}^{\ell}\Phi_{1}(\bU)}{\blambda_{n}^{\ell + 1}}\right|\leq \|\Phi_{1}(\bU)\|_{2}^2\sum_{\ell=2}^{\infty}\frac{\left(\sum_{\ell=1}^{k}|\lambda_{\ell}(f)|^{-1}\right)^{\ell}}{|\blambda_{n}|^{\ell + 1}}\lesssim_{f}\frac{1}{n^2}
        \end{align*}
        with probability at least $1 - 8n\exp\left(-\frac{1}{6}(\log n)^2\right)$. Thus by the expansion from \eqref{eq:taylorexpand}, for all large enough $n$, we get,
        \footnotesize
        \begin{align}\label{eq:expr1A1}
            \lambda_{1}(f)\Phi_{1}(\bU_{n})^{\top}\bcA_{n}^{-1}\Phi_{1}(\bU_{n}) =
            & \frac{\lambda_{1}(f)}{\blambda_{n}}\|\Phi_{1}(\bU_{n})\|_{2}^2 -\frac{\lambda_{1}(f)}{\blambda_{n}^2}\sum_{\ell=1}^{k}\lambda_{\ell}(f)\sum_{i=1}^{n}\phi_{1}(U_{i})^2\phi_{\ell}(U_{i})^2 + O\left(n^{-2}\right)
        \end{align}
        \normalsize
        with probability at least $1 - 8n\exp\left(-\frac{1}{6}(\log n)^2\right)$. Note that we already have the first two terms in the expansion of \eqref{eq:finitekexpansion}. Now we analyse the second term in \eqref{eq:Woodbury} which contributes the third term in \eqref{eq:finitekexpansion}. Recalling the expression of the second term, we first analyse $(-\bLambda^{-1} + \bV^{\top}\bcA_{n}^{-1}\bV)^{-1}$. In particular, in the following lemma, we start by showing that $\bV^{\top}\bcA_{n}^{-1}\bV$ is approximately a constant times identity matrix. 
        \begin{lemma}\label{lemma:concVTAV}
        For the matrices $\bV$ defined in \eqref{eq:defVLambda} and $\bcA_{n}$ and for large enough $n$,
        \begin{align*}
            \left\|\bV^{\top}\bcA_{n}^{-1}\bV - \frac{\bI_{k-1}}{\lambda_{1}(f)}\right\|_{2\ra 2}\lesssim_{f}k\left(\frac{\log n}{\sqrt{n}} + \frac{1}{n}\right)
        \end{align*}
        with probability at least $1-17nk\exp\left(-\frac{1}{6}(\log n)^2\right)$.
        \end{lemma}
        The proof of Lemma \ref{lemma:concVTAV} is given in Section C.2. Now we show that $(-\bLambda^{-1} + \bV^{\top}\bcA_{n}^{-1}\bV)^{-1}$ can be replaced by $\left(\bLambda^{-1} - \frac{\bI_{k-1}}{\lambda_{1}(f)}\right)^{-1}$. Note that,
        \begin{align}
            \bigg\|
            & \left(\bLambda^{-1} - \bV^{\top}\bcA_{n}^{-1}\bV\right)^{-1} - \left(\bLambda^{-1} - \frac{\bI_{k-1}}{\lambda_{1}(f)}\right)^{-1}\bigg\|_{2\ra 2}\nonumber\\
            & \leq \left\|\left(\bLambda^{-1} - \frac{\bI_{k-1}}{\lambda_{1}(f)}\right)^{-1}\right\|_{2\ra 2}\left\|\bV^{\top}\bcA_{n}^{-1}\bV - \frac{\bI_{k-1}}{\lambda_{1}(f)}\right\|_{2\ra 2}\left\|\left(\bLambda^{-1} - \bV^{\top}\bcA_{n}^{-1}\bV\right)^{-1}\right\|_{2\ra 2}\label{eq:bdonmiddle}.
        \end{align}
        By Lemma \ref{lemma:concVTAV} and Weyl's inequality observe that for all $1\leq i\leq k-1$,
        \begin{align*}
            \left|\lambda_{i}\left(\bLambda^{-1} - \bV^{\top}\bcA_{n}^{-1}\bV\right) - \lambda_{i}\left(\bLambda^{-1} - \frac{\bI_{k-1}}{\lambda_{1}(f)}\right)\right|\lesssim_{f}k\left(\frac{\log n}{\sqrt{n}} + \frac{1}{n}\right)
        \end{align*}
        and hence for large enough $n$,
        \begin{align*}
            \min_{1\leq i\leq k-1}\left|\lambda_{i}\left(\bLambda^{-1} - \bV^{\top}\bcA_{n}^{-1}\bV\right)\right|\geq \frac{1}{2}\min_{2\leq i\leq k}\left|\frac{\lambda_{1}(f) - \lambda_{i}(f)}{\lambda_{1}(f)\lambda_{i}(f)}\right|
        \end{align*}
        for large enough $n$ with probability at least $1-17nk\exp\left(-\frac{1}{6}(\log n)^2\right)$. Thus once again using Lemma \ref{lemma:concVTAV} along with \eqref{eq:bdonmiddle} we have,
        \begin{align}\label{eq:approxmiddle}
            \bigg\|\left(\bLambda^{-1} - \bV^{\top}\bcA_{n}^{-1}\bV\right)^{-1} 
            & - \left(\bLambda^{-1} - \frac{\bI_{k-1}}{\lambda_{1}(f)}\right)^{-1}\bigg\|_{2\ra 2}\lesssim_{f}k\left(\frac{\log n}{\sqrt{n}} + \frac{1}{n}\right)
        \end{align}
        with probability at least $1-34nk\exp\left(-\frac{1}{6}(\log n)^2\right)$. Next, once again recalling the expression of the second term from \eqref{eq:Woodbury} we now provide an expansion of the term $\Phi_{1}(\bU_{n})^{\top}\bcA_{n}^{-1}\bV$, showing a simplification with an additional error term.
    
        \begin{lemma}\label{lemma:form1tAinvV}
            For the matrix $\bV$ defined in \eqref{eq:defVLambda},
            \begin{align*}
                \Phi_{1}(\bU_{n})^{\top}\bcA_{n}^{-1}\bV = \frac{\Phi_{1}(\bU_{n})^{\top}\bV}{\blambda_{n}} + \frac{\bs_{n}}{\blambda_{n}}
            \end{align*}
            where $\left\|\bs_{n}\right\|_{2}\lesssim_{f}\sqrt{k}$ with probability at least $1-16n\exp\left(-\frac{1}{6}(\log n)^2\right)$.
        \end{lemma}
        The proof of Lemma \ref{lemma:form1tAinvV} is given in Section C.3. Having detailed the expansions of the terms involved up to negligible constants, we are now ready to collect the results. First we show that the quadratic term $$\Phi_{1}(\bU_{n})^{\top}\bcA_{n}^{-1}\bV\left(\bLambda^{-1} - \bV^{\top}\bcA_{n}^{-1}\bV\right)^{-1}\bV^{\top}\bcA_{n}^{-1}\Phi_{1}(\bU_{n})^{\top}$$ contributed by the second term in \eqref{eq:Woodbury} can be replaced by,
        \begin{align*}
            \Phi_{1}(\bU_{n})^{\top}\bcA_{n}^{-1}\bV\left(\bLambda^{-1} - \frac{1}{\lambda_{1}(f)}\bI\right)^{-1}\bV^{\top}\bcA_{n}^{-1}\Phi_{1}(\bU_{n})^{\top}
        \end{align*} 
        up to an additive negligible error. Towards that notice,
        \small
        \begin{align*}
            \bigg|\Phi_{1}(\bU_{n})^{\top}\bcA_{n}^{-1}
            &\bV\left[\left(\bLambda^{-1} - \bV^{\top}\bcA_{n}^{-1}\bV\right)^{-1} - \left(\bLambda^{-1} - \frac{1}{\lambda_{1}(f)}\bI\right)^{-1}\right]\bV^{\top}\bcA_{n}^{-1}\Phi_{1}(\bU_{n})^{\top}\bigg|\\
            &\leq \left\|\Phi_{1}(\bU_{n})^{\top}\bcA_{n}^{-1}\bV\right\|_{2}^2\left\|\left(\bLambda^{-1} - \bV^{\top}\bcA_{n}^{-1}\bV\right)^{-1} - \left(\bLambda^{-1} - \frac{1}{\lambda_{1}(f)}\bI\right)^{-1}\right\|_{2\ra 2}.
        \end{align*}
        \normalsize
        By \eqref{eq:approxmiddle} we already have a bound on the second term in the R.H.S. So now we only need to figure out a bound on the first term. Recalling the expansion from by Lemma \ref{lemma:form1tAinvV} note that,
        \begin{align*}
            \left\|\Phi_{1}(\bU_{n})^{\top}\bV\right\|_{2}^{2} = \sum_{\ell=2}^{k}\left(\sum_{i=1}^{n}\phi_{1}(U_{i})\phi_{\ell}(U_{i})\right)^2.
        \end{align*}
        \cblue
        Recall $\phi_{j},1\leq j\leq k$ are orthonormal, then by the bounds from Lemma \ref{lemma:efuncbdd} and Hoeffding inequality we have,
        \begin{align*}
            \left|\frac{1}{n}\sum_{i=1}^{n}\phi_{1}(U_{i})\phi_{\ell}(U_{i})\right|\lesssim_{f} \frac{\log n}{\sqrt{n}}
        \end{align*}
        with probability at least $1-2\exp( -\frac{1}{6}(\log n)^2)$ for all $2\leq \ell\leq k$. Recalling the bound from \eqref{eq:blambdanconc2} and using union bound we get,
        \begin{align}\label{eq:bdd1Vnorm}
            \left\|\frac{\Phi_{1}(\bU_{n})^{\top}\bV}{\blambda_{n}}\right\|_{2}\leq \frac{C\sqrt{k}\log n}{\sqrt{n}}
        \end{align}
        with probability at least $1-9nk\exp\left(-\frac{1}{6}(\log n)^2\right)$ for large enough $n$. 
        Once again using \eqref{eq:blambdanconc2} and Lemma \ref{lemma:form1tAinvV} we conclude,
        \begin{align}\label{eq:bdd1AinvV}
            \left\|\Phi_{1}(\bU_{n})^{\top}\bcA_{n}^{-1}\bV\right\|_{2}\leq C\sqrt{k}\left(\frac{\log n}{\sqrt{n}} + \frac{1}{n}\right)
        \end{align}
        with probability at least $1-25nk\exp\left(-\frac{1}{6}(\log n)^2\right)$. Combining \eqref{eq:approxmiddle} and \eqref{eq:bdd1AinvV} we conclude,
        \small
        \begin{align}\label{eq:woodsecondtermequiv}
            \bigg|\Phi_{1}(\bU_{n})^{\top}\bcA_{n}^{-1}\bV\left[\left(\bLambda^{-1} - \bV^{\top}\bcA_{n}^{-1}\bV\right)^{-1} - \left(\bLambda^{-1} - \frac{1}{\lambda_{1}(f)}\bI\right)^{-1}\right]
            & \bV^{\top}\bcA_{n}^{-1}\Phi_{1}(\bU_{n})^{\top}\bigg|\nonumber\\
            & \lesssim_{f}k^{2}\left(\frac{\log n}{\sqrt{n}} + \frac{1}{n}\right)^3
        \end{align}
        \normalsize
        with probability at least $1-59nk\exp\left(-\frac{1}{6}(\log n)^2\right)$.\cblack In the final step using the above approximations we further simplify the term
        \begin{align*}
            \Phi_{1}(\bU_{n})^{\top}\bcA_{n}^{-1}\bV\left(\bLambda^{-1} - \frac{1}{\lambda_{1}(f)}\bI\right)^{-1}\bV^{\top}\bcA_{n}^{-1}\Phi_{1}(\bU_{n})^{\top}.
        \end{align*}
        to gather the third term in R.H.S of \eqref{eq:finitekexpansion} with a negligible error. Note that by Lemma \ref{lemma:form1tAinvV} we have,
        \begin{align}\label{eq:T1T2T3}
            \Phi_{1}(\bU_{n})^{\top}\bcA_{n}^{-1}\bV\left(\bLambda^{-1} - \frac{1}{\lambda_{1}(f)}\bI\right)^{-1}\bV^{\top}\bcA_{n}^{-1}\Phi_{1}(\bU_{n})^{\top} = T_{1} + T_{2} + 2T_{3}
        \end{align}
        with probability at least $1-16n\exp\left(-\frac{1}{6}(\log n)^2\right)$ where,
        \footnotesize
        \begin{align*}
            T_{1} = \frac{\Phi_{1}(\bU_{n})^{\top}\bV}{\blambda_{n}}\left(\bLambda^{-1} - \frac{1}{\lambda_{1}(f)}\bI\right)^{-1}\left(\frac{\Phi_{1}(\bU_{n})^{\top}\bV}{\blambda_{n}}\right)^{\top}, T_{2} = \frac{\bs_{n}}{\blambda_{n}}\left(\bLambda^{-1} - \frac{1}{\lambda_{1}(f)}\bI\right)^{-1}\left(\frac{\bs_{n}}{\blambda_{n}}\right)^{\top},
        \end{align*}
        \normalsize
        and,
        \begin{align*}
            T_{3} = \frac{\Phi_{1}(\bU_{n})^{\top}\bV}{\blambda_{n}}\left(\bLambda^{-1} - \frac{1}{\lambda_{1}(f)}\bI\right)^{-1}\left(\frac{\bs_{n}}{\blambda_{n}}\right)^{\top}.
        \end{align*}
        Using \eqref{eq:blambdanconc} and Lemma \ref{lemma:form1tAinvV} note that $$|T_{2}|\lesssim_{f}k/n^2  \text{ with probability at least } 1-24n\exp\left(-\frac{1}{6}(\log n)^2\right).$$ Additionally using \eqref{eq:bdd1Vnorm} we get $|T_{3}|\lesssim_{f}k\log n/n^{3/2}$ with probability at least $1-34nk\exp\left(-\frac{1}{6}(\log n)^2\right)$. Thus recalling \eqref{eq:T1T2T3}, we have,
        \begin{align}\label{eq:1AVinvVA1}
            \Phi_{1}(\bU_{n})^{\top}\bcA_{n}^{-1}\bV\bigg(\bLambda^{-1}- \frac{1}{\lambda_{1}(f)}\bI\bigg)^{-1}\bV^{\top}\bcA_{n}^{-1}\Phi_{1}(\bU_{n})^{\top} = T_{1} + a_{n}
        \end{align}
        with probability at least $1-74nk\exp\left(-\frac{1}{6}(\log n)^2\right)$, where $|a_{n}|\lesssim_{f}\frac{k\log n}{n^{3/2}}$. Note that by definition,
        \begin{align}\label{eq:1VinvV1}
            T_{1} = \frac{1}{\blambda_{n}^2}\sum_{j=2}^{k}\frac{\lambda_{j}(f)\lambda_{1}(f)}{\lambda_{1}(f) - \lambda_{j}(f)}\left(\sum_{i=1}^{n}\phi_{1}(U_{i})\phi_{j}(U_{i})\right)^2
        \end{align}
        which is exactly the third term on R.H.S of \eqref{eq:finitekexpansion}.
        The proof is now completed by collecting \eqref{eq:Woodbury}, \eqref{eq:expr1A1}, \eqref{eq:woodsecondtermequiv} and \eqref{eq:1AVinvVA1}.
    
\section{Proof of results from Section \ref{s:prft2}}\label{s:Matrixeig}

\blue{In this section we complete the proof of \Cref{t:main2} by providing proofs of \Cref{lemma:Aninv}, \Cref{p:Cninvnorm}, \Cref{prop:vphiconc} and \Cref{p:main3}.}

\subsection{Proof of \Cref{lemma:Aninv}}\label{sec:proofofAninv}

In the following for any matrix $\bm S_{n}$ we will consider the eigenvalues as,
\begin{align*}
    \lambda_{1}(\bm S_{n})\geq \lambda_{1}(\bm S_{n})\geq\cdots\geq\lambda_{n}(\bm S_{n}).
\end{align*}
Define,
\begin{align*}
    \widetilde{\bW}_{n} = \sum_{i=2}^{n}\lambda_{i}(\bW_{n})\bv_{i}\blue{\bv_i^\top.}
\end{align*}
Then note that the spectrum of $\widetilde{\bW}_{n}$ is given by,
\begin{align*}
    \sigma(\widetilde{\bW}_{n}) = \{\lambda_{j}(\bW_{n}):j\neq 1\}\bigcup\{0\}.
\end{align*}
Now by Weyl's inequality,
\footnotesize
\begin{align}\label{eq:lbbAAtildediff}
    \min_{j=1}^{n}\left|\frac{\lambda_{1}(\bA_{n})}{n} - \frac{\lambda_{j}(\widetilde{\bA}_{n})}{n}\right|
    &\geq \min_{j=1}^{n}\left|\frac{\lambda_{1}(\bA_{n})}{n} - \frac{\lambda_{j}(\widetilde{\bW}_{n})}{n}\right| - \frac{1}{n}\left\|\bA_{n} - \bW_{n}\right\|_{2\ra 2}\nonumber\\
    &\geq \min\left\{\left|\frac{\lambda_{1}(\bW_{n})}{n} - \frac{\lambda_{j}(\bW_{n})}{n}\right|: j\neq 1, \frac{|\lambda_{1}(\bW_{n})|}{n}\right\} - \frac{2}{n}\left\|\bA_{n} - \bW_{n}\right\|_{2\ra 2}\nonumber\\
    &\geq \min\left\{\left|\frac{\lambda_{1}(\bW_{n})}{n} - \frac{\lambda_{2}(\bW_{n})}{n}\right|, \frac{|\lambda_{1}(\bW_{n})|}{n}\right\} - \frac{2}{n}\left\|\bA_{n} - \bW_{n}\right\|_{2\ra 2}.
\end{align}
\normalsize
Notice,
\begin{align*}
    \left|\frac{\lambda_{1}(\bW_{n})}{n} - \frac{\lambda_{2}(\bW_{n})}{n}\right|\geq |\lambda_{1}(W) - \lambda_{2}(W)|
    & - \left|\frac{\lambda_{1}(\bW_{n})}{n} - \lambda_{1}(W)\right|\\
    & - \left|\frac{\lambda_{2}(\bW_{n})}{n} - \lambda_{2}(W)\right|\bone\left\{\lambda_{2}(\bW_{n})\geq 0\right\}.
\end{align*}
Following the proof of Lemma \ref{lemma:lambda1Wnconc} it can be easily shown that,
\begin{align*}
    \left|\frac{\lambda_{2}(\bW_{n})}{n} - \lambda_{2}(W)\right|\bone\left\{\lambda_{2}(\bW_{n})\geq 0\right\}\lesssim_{W}\frac{\log n}{n}
\end{align*}
with probability at least $1-8n\exp\left(-\frac{1}{6}(\log n)^2\right)$. Additionally using the bound from Lemma \ref{lemma:lambda1Wnconc} we conclude,
\begin{align}\label{eq:lambda12Wnlbb}
    \left|\frac{\lambda_{1}(\bW_{n})}{n} - \frac{\lambda_{2}(\bW_{n})}{n}\right|\geq \frac{|\lambda_{1}(W) - \lambda_{2}(W)|}{2}
\end{align}
with probability at least $1-16n\exp\left(-\frac{1}{6}(\log n)^2\right)$. The proof is now completed by collecting the lower bounds from \eqref{eq:lbbAAtildediff}, \eqref{eq:lambda12Wnlbb}, Lemma \ref{lemma:lambda1Wnconc} and the upper bound from \eqref{eq:Bnbddvershynin}.

\subsection{Proof of \Cref{p:Cninvnorm}}\label{sec:proofofCinvnorm}
The first statement \eqref{eq:Bnbddvershynin} follows from Corollary 4.4.8 from \cite{vershynin2018high}.

Next we prove \eqref{eq:Cnorm}. The spectrum of the matrix $\bC_{n}$ is given by,
\begin{align}\label{eq:defsigmaCn}
    \sigma(\bC_{n}) = \left\{\lambda_{1}(\bA_{n}), \lambda_{1}(\bA_{n}) - \lambda_{j}(\bW_{n}): j\neq 1\right\}.
\end{align}
Now,
\begin{align*}
    \min_{j\neq 1}\left|\frac{\lambda_{1}(\bA_{n})}{n} - \frac{\lambda_{j}(\bW_{n})}{n}\right|
    &\geq \min_{j\neq 1}\left|\frac{\lambda_{1}(\bW_{n})}{n} - \frac{\lambda_{j}(\bW_{n})}{n}\right| - \frac{1}{n}\left\|\bA_{n} - \bW_{n}\right\|_{2\ra 2}\\
    &\geq \left|\frac{\lambda_{1}(\bW_{n})}{n} - \frac{\lambda_{2}(\bW_{n})}{n}\right| - \frac{1}{n}\left\|\bmB_{n} \right\|_{2\ra 2}.
\end{align*}
Now, to provide a further lower bound, we provide a lower bound on the difference between the eigenvalues $\lambda_1(\bA_n)$ and $\lambda_{1}(\bW_n)$. 
\cred
Notice that using traingle inequality we have,
\begin{align*}
    \left|\frac{\lambda_{1}(\bW_{n})}{n} - \frac{\lambda_{2}(\bW_{n})}{n}\right|\geq \left|\lambda_1(W) - \lambda_2(W)\right| - \left|\lambda_1(W) - \frac{\lambda_1(\bW_n)}{n}\right| - \left|\lambda_2(W) - \frac{\lambda_2(\bW_n)}{n}\right|.
\end{align*}
Combining the above lower bounds and following the proof of Lemma \ref{lemma:lambda1Wnconc} (in particular to control the last term in the above lower bound) shows,
\begin{align*}
    \left|\frac{\lambda_{1}(\bW_{n})}{n} - \frac{\lambda_{2}(\bW_{n})}{n}\right|\geq \left|\lambda_1(W) - \lambda_2(W)\right| - O\left(\frac{\log n}{\sqrt{n}}\right)
\end{align*}
with probability at least $1-16n\exp\left(-\frac{1}{6}(\log n)^2\right)$. Then for large enough $n$ we get,
\begin{align}\label{eq:lambda12Wnlbb0}
    \left|\frac{\lambda_{1}(\bW_{n})}{n} - \frac{\lambda_{2}(\bW_{n})}{n}\right|\geq \frac{2}{3}|\lambda_{1}(W) - \lambda_{2}(W)|.
\end{align}
Following the bound \eqref{eq:lambda12Wnlbb0} and the upper bound of $\|\bmB_{n}\|_{2\ra 2}$ from \eqref{eq:Bnbddvershynin} we get,
\begin{align}\label{eq:lbbAnWn}
    \min_{j\neq 1}\left|\frac{\lambda_{1}(\bA_{n})}{n} - \frac{\lambda_{j}(\bW_{n})}{n}\right|\geq \frac{|\lambda_{1}(W) - \lambda_{2}(W)|}{2}
\end{align}
with probability at least $1-17n\exp\left(-\frac{1}{6}(\log n)^2\right)$.\cblack Additionally using the bound from Lemma \ref{lemma:lambda1Wnconc} and \eqref{eq:Bnbddvershynin} we get,
\begin{align}\label{eq:lbb1An}
    \frac{|\lambda_{1}(\bA_{n})|}{n}\geq \frac{|\lambda_{1}(\bW_{n})|}{n} - \frac{1}{n}\left\|\bA_{n} - \bW_{n}\right\|_{2\ra 2}\geq \frac{|\lambda_{1}(W)|}{2}.
\end{align}
with probability at least $1-9n\exp\left(-\frac{1}{6}(\log n)^2\right)$. The proof of \eqref{eq:Cnorm} is now completed by recalling the collection $\sigma(\bC_{n})$ from \eqref{eq:defsigmaCn} and the lower bounds from \eqref{eq:lbbAnWn} and \eqref{eq:lbb1An}.

Thanks to \eqref{eq:Bnbddvershynin} and \eqref{eq:Cnorm}, we have $\|\bC_{n}^{-1}\bmB_{n}\|_{2\ra2}\lesssim_{W}n^{-1/2}$ with probability at least $1-Cn\exp\left(-\frac{1}{6}(\log n)^2\right)$. Then for large $n$ we have the following \blue{Taylor} Expansion,
\small
\begin{align*}
    (\lambda_{1}(\bA_{n})\bI_{n} - \widetilde{\bA}_{n})^{-1} = (\bI_{n} - \bC_{n}^{-1}\bmB_{n})^{-1}\bC_{n}^{-1} = \bC_{n}^{-1} + \bC_{n}^{-1}\bmB_{n}\bC_{n}^{-1} + \sum_{k\geq 2}(\bC_{n}^{-1}\bmB_{n})^k\bC_{n}^{-1}.
\end{align*}
\normalsize
Recalling the definition of $\bC_{n}$ from \eqref{eq:defCn} it is easy to see that $\lambda_{1}(\bA_{n})\bC_{n}^{-1}\bv_{1} = \bv_{1}$. Multiplying $\bv_1$ from the left and right on both sides of the above Taylor expansion gives,
\begin{align}\label{eq:taylorfirstdiff1}
    \bv_1^\top  (\lambda_{1}(\bA_{n})\bI_{n} - \widetilde{\bA}_{n})^{-1} \bv_1= \frac{1}{\lambda_{1}(\bA_{n})} +  \frac{\bv^\top_1 \bmB_{n}\bv_1}{\lambda_{1}(\bA_{n})^2} + \frac{\sum_{k\geq 1}\bv_{1}^\top \bmB_n (\bC_{n}^{-1}\bmB_{n})^k\bv_{1}}{\lambda_{1}(\bA_{n})^2}.
\end{align}
Condition on that $\|\bmB_n\|_{2\ra2}\lesssim_W \sqrt n$ and $\|\bC_{n}^{-1}\bmB_{n}\|_{2\ra2}\lesssim_{W}n^{-1/2}$, for $k\geq 2$, we have
\begin{align}\label{e:highorder}
\bv_{1}^\top \bmB_n (\bC_{n}^{-1}\bmB_{n})^k\bv_{1}\lesssim_W \sqrt n n^{-k/2}=n^{-(k-1)/2}.
\end{align}
Plugging \eqref{e:highorder} into \eqref{eq:taylorfirstdiff1}, and using the equation \eqref{eq:Anresolvent}, we conclude the statement \eqref{eq:taylorfirstdiff}
\begin{align*}
    \frac{\lambda_{1}(\bA_{n})}{\lambda_{1}(\bW_{n})}(\lambda_{1}(\bA_{n}) - \lambda_{1}(\bW_{n})) = \bv_{1}^{\top}\bmB_{n}\bv_{1} + \bv_{1}^{\top}\bmB_{n}\bC_{n}^{-1}\bmB_{n}\bv_{1} + O_{W}\left(\frac{1}{\sqrt{n}}\right)
\end{align*}
with probability at least $1-Cn\exp\left(-\frac{1}{6}(\log n)^2\right)$.

\subsection{Proof of  \Cref{prop:vphiconc}}\label{sec:proofofvphiconc}

Before proceeding with the proofs we first introduce some notation which will be used throughout this section. Recalling $U_{1}, U_{2},\ldots, U_{n}$ consider the permutation matrix $\Pi_{n}$ from Lemma \ref{lemma:permmat}. We define, $\bu_{n}^{\Pi} = \Pi_{n}\bu_n$ for any vector $\bu_n\in \bR^{n}$ and $\bm S_{n}^{\Pi} = \Pi_n(\bm S_n)\Pi^{\top}_n$ for any matrix $\bm S_n$. Further for any vector $\bu_n\in \bR^{n}$ we consider a functional embedding on $[0,1]$ as,
\begin{align*}
    \blue{f_{\bu_{n}}(x) = \sum_{j=1}^{n}\sqrt{n}u_{j}\one\left[x\in I_{j}\right]}\text{ where }I_{j} = \left[\frac{j-1}{n}, \frac{j}{n}\right), 1\leq j\leq n.
\end{align*}
By definition notice that for two vectors $\bu_{n,1}$ and $\bu_{n,2}$,
\begin{align}\label{eq:equivfvector0}
    \left\|f_{\bu_{n,1}} - f_{\bu_{n,2}}\right\|_{2} = \|\bu_{n,1} - \bu_{n,2}\|_{2}.
\end{align}

We recall $\bPhi_1$ from \Cref{e:defPhi}, and $\bv_1$ is the  eigenvector of $\bW_n$ corresponding to the largest eigenvalue. 
We first show that $f_{\bv_{1}}$ and $f_{\bPhi_{1}}$ are close in $\|\cdot\|_{2}$ norm, which will imply that $\bv_{1}$ is close to $\bPhi_{1}$. In particular the following lemma states that $f_{\bPhi_{1}^{\Pi}}$ is close to $\phi_{1}$ with high probability. \blue{The proof is given in Section D.1 in the supplementary material.}
\begin{lemma}\label{lemma:fphiapproxphi}
    For the graphon $W$,
    \begin{align*}
        \left\|f_{\bPhi_{1}^{\Pi}} - \phi_{1}\right\|_{2}\lesssim_{W}\frac{\log n}{\sqrt{n}}
    \end{align*}
    with probability at least $1-2n\exp(-2(\log n)^2/3)$.
\end{lemma}
Next we turn our attention to the vector $\bv_{1}$. In the following proposition, \blue{with proof given in Section D.2}, we study the approximation of $\phi_{1}$ by $f_{\bv_{1}^{\Pi}}$. 
\begin{prop}\label{prop:fv1approxphi1}
    Recalling the eigenvector $\bv_{1}$ define,
    \begin{align*}
        \widetilde{\bv}_{1}^{\Pi} = \bv_{1}^{\Pi}\one\left[\langle \phi_{1}, f_{\bv_{1}^{\Pi}}\rangle>0\right] - \bv_{1}^{\Pi}\one\left[\langle \phi_{1}, f_{\bv_{1}^{\Pi}}\rangle\leq 0\right].
    \end{align*}
    Then for large enough $n$,
    \begin{align*}
        \left\|f_{\widetilde{\bv}_{1}^{\Pi}} - \phi_{1}\right\|_{2}\lesssim_{W}\left(\frac{\log n}{\sqrt{n}}\right)^{1/2}
    \end{align*}
    with probability at least $1-40n\exp\left(-\frac{1}{6}(\log n)^2\right)$.
\end{prop}
Now combining Lemma \ref{lemma:fphiapproxphi} and Proposition \ref{prop:fv1approxphi1} and \eqref{eq:equivfvector0} we get,
\begin{align}\label{eq:Phivdiff}
    \left\|\bPhi_{1}^{\Pi}- \widetilde{\bv}_{1}^{\Pi}\right\|\lesssim_{W}\left(\frac{\log n}{\sqrt{n}}\right)^{1/2}
\end{align}
with probability at least $1-42n\exp\left(-\frac{1}{6}(\log n)^2\right)$. Now note that,
\begin{align*}
    \left(\widetilde{\bv}_{1}^{\Pi}\right)^{\top}\bmB_{n}^{\Pi}\widetilde{\bv}_{1}^{\Pi} = \left(\bv_{1}^{\Pi}\right)^{\top}\bmB_{n}^{\Pi}\bv_{1}^{\Pi} = \bv_{1}^{\top}\bmB_{n}\bv_{1}\text{ and }\left(\bPhi_{1}^{\Pi}\right)^{\top}\bmB_{n}^{\Pi}\bPhi_{1}^{\Pi} = \bPhi_{1}^{\top}\bmB_{n}\bPhi_{1}.
\end{align*}
Then,
\begin{align}\label{eq:vPhiquadbdd}
    \left|\bv_{1}^{\top}\bmB_{n}\bv_{1} - \bPhi_{1}^{\top}\bmB_{n}\bPhi_{1}\right|
    & = \left|\left(\widetilde{\bv}_{1}^{\Pi}\right)^{\top}\bmB_{n}^{\Pi}\widetilde{\bv}_{1}^{\Pi} - \left(\bPhi_{1}^{\Pi}\right)^{\top}\bmB_{n}^{\Pi}\bPhi_{1}^{\Pi} \right|\nonumber\\
    &\leq \left|\left(\widetilde{\bv}_{1}^{\Pi} - \bPhi_{1}^{\Pi}\right)^{\top}\bmB_{n}^{\Pi}\bPhi_{1}^{\Pi} \right| + \left|\left(\widetilde{\bv}_{1}^{\Pi} - \bPhi_{1}^{\Pi}\right)^{\top}\bmB_{n}^{\Pi}\widetilde{\bv}_{1}^{\Pi} \right|.
\end{align}
Now, given $U_{1},U_{2},\ldots, U_{n}$ the entries of the symmetric matrix (above the diagonal) are independent with bounded subgaussian norm. Then using the Hoeffding-Inequality with conditioning on $U_{1},U_{2},\ldots, U_{n}$ shows,
\begin{align*}
    \P\left(\left|\left(\widetilde{\bv}_{1}^{\Pi} - \bPhi_{1}^{\Pi}\right)^{\top}\bmB_{n}^{\Pi}\widetilde{\bv}_{1}^{\Pi} \right|>t|U_{1},\cdots, U_{n}\right)
    &\leq 2\exp\left(-\frac{c_{W}t^2}{\left\|\widetilde{\bv}_{1}^{\Pi} - \bPhi_{1}^{\Pi}\right\|_{2}^2\left\|\widetilde{\bv}_{1}^{\Pi}\right\|_{2}^2}\right).
\end{align*}
Taking expectations on both sides yields,
\begin{align*}
    \P\left(\left|\left(\widetilde{\bv}_{1}^{\Pi} - \bPhi_{1}^{\Pi}\right)^{\top}\bmB_{n}^{\Pi}\widetilde{\bv}_{1}^{\Pi} \right|>t\right)
    &\leq2\E\left[\exp\left(-\frac{c_{W}t^2}{\left\|\widetilde{\bv}_{1}^{\Pi} - \bPhi_{1}^{\Pi}\right\|_{2}^2}\right)\right]\\
    &\leq 2\exp\left(-\frac{c_{W}t^2\sqrt{n}}{\log n}\right) + O\left(n\exp\left(-\frac{1}{6}(\log n)^2\right)\right)
\end{align*}
where the last inequality follows from \eqref{eq:Phivdiff}. Choosing $t = \left((\log n)^3/6c_{W}\sqrt{n}\right)^{1/2}$ shows,
\begin{align*}
    \P\left(\left|\left(\widetilde{\bv}_{1}^{\Pi} - \bPhi_{1}^{\Pi}\right)^{\top}\bmB_{n}^{\Pi}\widetilde{\bv}_{1}^{\Pi} \right|>\frac{(\log n)^{3/2}}{6c_{W}n^{1/4}}\right) \lesssim n\exp\left(-\frac{1}{6}(\log n)^2\right).
\end{align*}
Similarly we can show,
\begin{align*}
    \P\left(\left|\left(\widetilde{\bv}_{1}^{\Pi} - \bPhi_{1}^{\Pi}\right)^{\top}\bmB_{n}^{\Pi}\bPhi_{1}^{\Pi} \right|>\frac{(\log n)^{3/2}}{6c_{W}n^{1/4}}\right) \lesssim n\exp\left(-\frac{1}{6}(\log n)^2\right).
\end{align*}
Finally recalling \eqref{eq:vPhiquadbdd} shows,
\small
\begin{align*}
    \left|\bv_{1}^{\top}\bmB_{n}\bv_{1} - \bPhi_{1}^{\top}\bmB_{n}\bPhi_{1}\right|\lesssim_{W}\left(\frac{\log^{3}n}{\sqrt{n}}\right)^{\frac{1}{2}}\text{ with probability at least }1 -  Cn\exp\left(-\frac{1}{6}(\log n)^2\right).
\end{align*}
\normalsize
This finishes the proof of \eqref{e:replace}. 

Next we prove \eqref{e:vBCBv}. The proof proceeds stepwise by replacing the matrix $\bC_{n}$ up to negligible error. In the following for any matrix $\bm S_{n}$ we will consider the eigenvalues as,
\begin{align*}
    \lambda_{1}(\bm S_{n})\geq \lambda_{1}(\bm S_{n})\geq\cdots\geq\lambda_{n}(\bm S_{n}).
\end{align*}
In the following lemma we replace $\bC_{n}$ and $\bv_1$ in the \blue{expression} $\bv_{1}^{\top}\bmB_{n}\bC_{n}^{-1}\bmB_{n}\bv_{1}$ by terms depending only on the matrix $\bW_{n}$ and $\bPhi_{1}$. 
\begin{lemma}\label{lemma:replaceC}
    Consider,
    \begin{align*}
        \bcC_{n} = \lambda_{1}(\bW_{n})\bI_{n} - \bW_{n} + \lambda_{1}(\bW_{n})\bPhi_{1}\bPhi_{1}^{\top}.
    \end{align*}
    Then $\|\bcC_{n}^{-1}\|_{2\ra 2}\lesssim_{W}n^{-1}$ and ,
    \begin{align*}
        \left|\bv_{1}^{\top}\bmB_{n}\bC_{n}^{-1}\bmB_{n}\bv_{1} - \left(\bPhi_{1}^{\Pi}\right)^{\top}\bmB_{n}^{\Pi}\left(\bcC_{n}^{\Pi}\right)^{-1}\bmB_{n}^{\Pi}\bPhi_{1}^{\Pi}\right|\lesssim_{W}\left(\frac{\log n}{\sqrt{n}}\right)^{1/2}
    \end{align*}
    with probability at least $1-Cn\exp\left(-\frac{1}{6}(\log n)^2\right)$
\end{lemma} 

Now we analyse the term $\left(\bPhi_{1}^{\Pi}\right)^{\top}\bmB_{n}^{\Pi}\left(\bcC_{n}^{\Pi}\right)^{-1}\bmB_{n}^{\Pi}\bPhi_{1}^{\Pi}$. Note that for all $1\leq k\leq n$,
\begin{align*}
    \left(\bmB_{n}^{\Pi}\Phi_{1}^{\Pi}\right)_{k} = \frac{1}{\sqrt{n}}\sum_{j<k}\bmB_{n}^{\Pi}(k,j)\phi_{1}(U_{(j)}) + \frac{1}{\sqrt{n}}\sum_{j>k}\bmB_{n}^{\Pi}(k,j)\phi_{1}(U_{(j)})
\end{align*}
which follows by \eqref{eq:UorderPieq} from Lemma \ref{lemma:permmat}. Define,
\begin{align*}
    \bZ_{1} = \left(\frac{1}{\sqrt{n}}\sum_{j<k}\bmB_{n}^{\Pi}(k,j)\phi_{1}(U_{(j)})\right)_{k=1}^{n}\text{ and }\bZ_{2} = \left(\frac{1}{\sqrt{n}}\sum_{j>k}\bmB_{n}^{\Pi}(k,j)\phi_{1}(U_{(j)})\right)_{k=1}^{n}
\end{align*}
where by convention sum over empty sets is set as $0$.
Then,
\begin{align}\label{eq:BCtermexpand}
    \left(\bPhi_{1}^{\Pi}\right)^{\top}\bmB_{n}^{\Pi}\left(\bcC_{n}^{\Pi}\right)^{-1}\bmB_{n}^{\Pi}\bPhi_{1}^{\Pi} 
    & = (\bZ_{1} + \bZ_{2})^{\top}\left(\bcC_{n}^{\Pi}\right)^{-1}(\bZ_{1} + \bZ_{2})\nonumber\\
    & = \bZ_{1}^{\top}\left(\bcC_{n}^{\Pi}\right)^{-1}\bZ_{1} + 2\bZ_{1}^{\top}\left(\bcC_{n}^{\Pi}\right)^{-1}\bZ_{2} + \bZ_{2}^{\top}\left(\bcC_{n}^{\Pi}\right)^{-1}\bZ_{2}.
\end{align}
By a conditional version of Hanson-Wright inequality we get,
\begin{align*}
    \P\bigg(\bigg|\bZ_{1}^{\top}\left(\bcC_{n}^{\Pi}\right)^{-1}\bZ_{1}
    & - \E\left[\bZ_{1}^{\top}\left(\bcC_{n}^{\Pi}\right)^{-1}\bZ_{1}|\bU_{n}\right]\bigg|>t|\bU_{n}\bigg)\\
    &\leq 2\exp\left( - c_{W}\min\left\{\frac{t^2}{\|\left(\bcC_{n}^{\Pi}\right)^{-1}\|_{F}^2},\frac{t}{\|\left(\bcC_{n}^{\Pi}\right)^{-1}\|_{2\ra 2}}\right\}\right)
\end{align*}
where $\bU_{n} = (U_{1},\cdots, U_{n})$. Then for large enough $n$ choosing $t = \frac{\log n}{\sqrt{6c_{W}n}}$, and using Lemma \ref{lemma:replaceC} with expectations on both sides of the above inequality shows,
\begin{align}\label{eq:Zquadconc}
    \left|\bZ_{1}^{\top}\left(\bcC_{n}^{\Pi}\right)^{-1}\bZ_{1} - \E\left[\bZ_{1}^{\top}\left(\bcC_{n}^{\Pi}\right)^{-1}\bZ_{1}|\bU_{n}\right]\right|\lesssim_{W}\frac{\log n}{\sqrt{n}}
\end{align}
with probability at least $1-Cn\exp\left(-\frac{1}{6}(\log n)^2\right)$. Similarly one can show,
\begin{align}\label{eq:Z2quadconc}
    \left|\bZ_{2}^{\top}\left(\bcC_{n}^{\Pi}\right)^{-1}\bZ_{2} - \E\left[\bZ_{2}^{\top}\left(\bcC_{n}^{\Pi}\right)^{-1}\bZ_{2}|\bU_{n}\right]\right|\lesssim_{W}\frac{\log n}{\sqrt{n}}
\end{align}
with probability at least $1-Cn\exp\left(-\frac{1}{6}(\log n)^2\right)$. Now consider $\bS_{n}$ to be a ${n\choose 2}\times{n\choose 2}$ matrix with entries,
\begin{align*}
    \bm S_{n}\left[(a,b),(c,d)\right] = \phi_{1}(U_{(a)})\left(\bcC_{n}^{\Pi}\right)^{-1}[b,c]\phi_{1}(U_{(d)})\text{ for all }1\leq a<b\leq n\text{ and }1\leq c<d\leq n
\end{align*}
and consider a vector $\bm X_{n}$ as ,
\begin{align*}
    \bm X_{n} = \left(\bmB_{n}^{\Pi}(a,b)\right)_{1\leq a<b\leq n}.
\end{align*}
Then by definition,
\begin{align*}
    \bZ_{1}^{\top}\left(\bcC_{n}^{\Pi}\right)^{-1}\bZ_{2} = \bm{X}_{n}^{\top}\bm{S}_{n}\bm{X}_{n}.
\end{align*}
Note that $\|\bm S_{n}\|_{F}\lesssim_{W}\|\left(\bcC_{n}^{\Pi}\right)^{-1}\|_{F}$ and hence once again using the Hanson Wright inequality along with Lemma \ref{p:Cninvnorm} as in the proof of \eqref{eq:Zquadconc}, we get,
\begin{align}\label{eq:Z1Z2termconc}
    \left|\bm{X}_{n}^{\top}\bm{S}_{n}\bm{X}_{n} - \E\left[\bm{X}_{n}^{\top}\bm{S}_{n}\bm{X}_{n}|\bU_{n}\right]\right|\lesssim_{W}\frac{\log n}{n^{1/4}}
\end{align}
with probability at least $1-Cn\exp\left(-\frac{1}{6}(\log n)^2\right)$. Now by direct computation,
\small
\begin{align*}
    \E\left[\bm{X}_{n}^{\top}\bm{S}_{n}\bm{X}_{n}|\bU_{n}\right] = \frac{1}{n}\sum_{i < j}\phi_{1}(U_{(i)})\phi_{1}(U_{(j)})\left(\bcC_{n}^{\Pi}\right)^{-1}[i,j]W(U_{(i)}, U_{(j)})(1-W(U_{(i)}, U_{(j)})).
\end{align*}
\normalsize
Then by the bounds on $\phi_{1}$ from Lemma \ref{lemma:efuncbdd},
\begin{align}\label{eq:Z1Z2small}
    \left|\E\left[\bm{X}_{n}^{\top}\bm{S}_{n}\bm{X}_{n}|\bU_{n}\right]\right|\lesssim_{W}\frac{1}{n}\sum_{i<j}\left|\left(\bcC_{n}^{\Pi}\right)^{-1}[i,j]\right|\leq \left\|\left(\bcC_{n}^{\Pi}\right)^{-1}\right\|_{F}\leq \frac{1}{\sqrt{n}}
\end{align}
with probability at least $1-Cn\exp\left(-\frac{1}{6}(\log n)^2\right)$, where the final bound follows from the bound on the operator norm from Lemma \ref{p:Cninvnorm}. Now combining the concentrations from \eqref{eq:Zquadconc}, \eqref{eq:Z2quadconc}, \eqref{eq:Z1Z2termconc}, along with the expansion from \eqref{eq:BCtermexpand} and the bound from \eqref{eq:Z1Z2small} we get,
\small
\begin{align*}
    \left|\left(\bPhi_{1}^{\Pi}\right)^{\top}\bmB_{n}^{\Pi}\left(\bcC_{n}^{\Pi}\right)^{-1}\bmB_{n}^{\Pi}\bPhi_{1}^{\Pi} - \E\left[\bZ_{1}^{\top}\left(\bcC_{n}^{\Pi}\right)^{-1}\bZ_{1}|\bU_{n}\right] - \E\left[\bZ_{2}^{\top}\left(\bcC_{n}^{\Pi}\right)^{-1}\bZ_{2}|\bU_{n}\right]\right|\lesssim_{W}\frac{\log n}{n^{1/4}}
\end{align*}
\normalsize 
with probability at least $1-Cn\exp\left(-\frac{1}{6}(\log n)^2\right)$. Invoking the following lemma, with proof given in Appendix D.4, completes the proof of \eqref{e:vBCBv}.
\begin{lemma}\label{lemma:Ctermfinalconvg}
    Consider,
    \begin{align*}
        T(\bU_{n}) = \E\left[\bZ_{1}^{\top}\left(\bcC_{n}^{\Pi}\right)^{-1}\bZ_{1}|\bU_{n}\right] + \E\left[\bZ_{2}^{\top}\left(\bcC_{n}^{\Pi}\right)^{-1}\bZ_{2}|\bU_{n}\right].
    \end{align*}
    Then with probability at least $1-Cn\exp\left(-\frac{1}{6}(\log n)^2\right)$, 
    \begin{align*}
        \left|T(\bU_{n}) - \frac{1}{\lambda_{1}(W)}\int\frac{\phi_{1}^{2}(x) + \phi_{1}^{2}(y)}{2}W(x,y)(1-W(x,y))\rd x\rd y\right|\lesssim_{W}\frac{\log n}{\sqrt{n}}.
    \end{align*}
\end{lemma}

\subsection{Proof of  \Cref{p:main3}}\label{sec:proofofmain3}

By plugging  \eqref{e:replace} and \eqref{e:vBCBv} into \eqref{eq:taylorfirstdiff}, we conclude that,
\footnotesize
\begin{align}\begin{split}\label{eq:replacelAWn}
& \phantom{{}={}} \frac{\lambda_{1}(\bA_{n})}{\lambda_{1}(\bW_{n})}(\lambda_{1}(\bA_{n}) - \lambda_{1}(\bW_{n})) 
  = \bv_{1}^{\top}\bmB_{n}\bv_{1} + \bv_{1}^{\top}\bmB_{n}\bC_{n}^{-1}\bmB_{n}\bv_{1} + O_{W}\left(\frac{1}{\sqrt{n}}\right)\\
 &=\bPhi_{1}^{\top}\bmB_{n}\bPhi_{1}
 +\frac{1}{\lambda_{1}(W)}\int\frac{\phi_{1}^{2}(x) + \phi_{1}^{2}(y)}{2}W(x,y)(1-W(x,y))\rd x\rd y+O_{W}\left(\frac{\log^{3}n}{\sqrt{n}}\right)^{1/2}
\end{split}\end{align}
\normalsize
with probability at least $1-Cn\exp\left(-\frac{1}{6}(\log n)^2\right)$.
\cblue
We now notice that $\bPhi_{1}^{\top}\bmB_{n}\bPhi_{1}$ in \eqref{e:replace} is given by
\begin{align*}
    \bPhi_{1}^{\top}\bmB_{n}\bPhi_{1} = \frac{2}{n}\sum_{i<j}\phi_{1}(U_{i})\phi_{1}(U_{j})(\bA_{n}(i,j) - W(U_{i},U_{j})).
\end{align*}
Notice that conditional on $\bU_n$, by the above decomposition, $\bPhi_{1}^{\top}\bmB_n\bPhi_1$ is a sum of independent elements. To find a CLT, we will now use Lyapunov's version, albeit in a conditional sense. Define,
\begin{align*}
    s_{n}^2 = \frac{4}{n^{2}}\sum_{i<j}\phi_{1}^2(U_{i})\phi_{1}^2(U_{j})W(U_{i},U_{j})(1 - W(U_{i},U_{j})),
\end{align*}
which is a U-statistics. 
By Theorem 5.4.A from \cite{serfling2009approximation} there exists a set $\cA$ of $(U_1, U_2,U_3,\cdots)$ such that $\mathbb P(\cA)=1$ on the set $\cA$,
\begin{align}\label{e:sn2}
    s_{n}^2\ra 2\int \phi_{1}^{2}(x)\phi_{1}^{2}(y)W(x,y)(1-W(x,y))\rd x\rd y,
\end{align}
and as $n\rightarrow \infty$,
\begin{align}\label{e:sn3}
   \sum_{i<j}\E_{\bA_n}\left[\left(\frac{2}{n}\phi_{1}(U_{i})\phi_{1}(U_{j})(\bA_{n}(i,j)-W(U_{i},U_{j}))\right)^3\middle|U_{1},\cdots U_{n}\right]\ra 0,
\end{align}
where the convergence follows by noticing that $\phi_1$ and $\bA_n(i,j) - W(U_i,U_j)$ are bounded by an universal constant depending on $W$.
The two statements \eqref{e:sn2} and \eqref{e:sn3} verify the Lyapunov condition  for $\bPhi_{1}^{\top}\bmB_{n}\bPhi_{1}$ conditioning on $\bU_{n}$. Now recalling the convergence from \eqref{e:sn2} we conclude that on $\cA$, $\bPhi_{1}^{\top}\bmB_{n}\bPhi_{1}$ converges to the normal distribution,
\begin{align}\label{e:pBpconverge}
    \bPhi_{1}^{\top}\bmB_{n}\bPhi_{1}|\bU_{n}\dto \cN(0,\sigma^2),
\end{align}
where
\begin{align*}
\sigma^2= 2\int \phi_{1}^{2}(x)\phi_{1}^{2}(y)W(x,y)(1-W(x,y))\rd x\rd y.
\end{align*}
\cblack
By Lemma \ref{lemma:lambda1Wnconc} notice that $\lambda_{1}(\bW_n)/n\overset{p}{\rightarrow}\lambda_{1}(W)$. Additionally, an application of Weyl's inequality shows that $\left|\lambda_{1}(\bA_n)/n - \lambda_{1}(\bW_n)/n\right| \overset{p}{\rightarrow} 0$. Combining we conclude that the ratio $\lambda_1(\bA_n)/\lambda_1(\bW_n)\rightarrow 1$ in probability. Then it follows from \eqref{eq:replacelAWn}, that conditioned on $\bU_n$,
$\lambda_{1}(\bA_{n}) - \lambda_{1}(\bW_{n})$
converges to the normal distribution $\cN(\alpha,\sigma^2)$ as in \eqref{e:cltla}. This finishes the proof of \Cref{p:main3}.

\color{black}
\small 
\bibliographystyle{abbrvnat}
\bibliography{ref} 
\normalsize

\appendix

\section{Hilbert Schmidt Operators from Kernel and Kernel Matrix}\label{s:kernelproof}
Consider $f:[0,1]^2\ra \bR$ to be a Lipschitz continuous and symmetric function with Lipschitz constant $L_{f}$.   Now for a the symmetric function $f:[0,1]^2\ra\bR$ define the Hilbert Schmidt Operator from $L^2([0,1])$ to $L^2([0,1])$,
    \begin{align}\label{eq:defHilbertf}
        T_{f}g (x) = \int f(x,y)g(y)\rd y,
    \end{align}

\subsection{Eigenfunctions of Hilbert Schmidt Operators from Kernel}
 
In this section we prove Lemma 5.1, which states that the eigenfunctions of Hilbert Schmidt Operator $T_f$ from \eqref{eq:defHilbertf} are bounded and Lipschitz.
\begin{proof}[Proof of Lemma 5.1]
    \blue{To prove part (a) notice that by definition,
    \begin{align*}
        \phi_j(x) = \frac{1}{\lambda_j}\int f(x,y)\phi_j(y)\rd y.
    \end{align*}
    Hence an application of Cauchy-Schwarz inequality shows that,
    \begin{align*}
        |\phi_j(x)|\leq \frac{1}{|\lambda_j(f)|}\sqrt{\int f^2(x,y)\rd y}\sqrt{\int \phi_j(y)^2\rd y}\leq \frac{B_f}{|\lambda_j(f)|}.
    \end{align*}}
    Now for part (b) conside $j\geq 1$ and note that,
    \begin{align*}
        \left|\lambda_{j}(f)\right|\left|\phi_{j}(x) - \phi_{j}(x')\right| = \left|\int_{0}^{1}\left(f(x,y) - f(x',y)\right)\phi_{j}(y)\rd y\right|\leq L_{f}|x-x'|\int_{0}^{1}|\phi_{j}(y)|\rd y
    \end{align*}
    Recall that $\phi_{j}$ are orthonormal, hence by Cauchy Schwarz inequality,
    \begin{align*}
        \left|\phi_{j}(x) - \phi_{j}(x')\right|\leq \frac{L_{f}}{|\lambda_{j}(f)|}|x-x'|
    \end{align*}
    which shows that $\phi_{j}$ is Lipschitz continuous with Lipschitz constant $L_{f}/|\lambda_{j}(f)|$. A similar proof holds for $\phi_{j}'$ for all $j\geq 1$
\end{proof}

\subsection{Concentration of Hilbert Schmidt Operators from Kernel Matrix}
Consider a sequence $U_{1},U_{2},\ldots,U_{n}$ of randomly drawn samples from the Uniform distribution on $[0,1]$. In this section we consider a $n\times n$ matrix with elements $f(U_{(i)}, U_{(j)})$, where $U_{(1)},\ldots,U_{(n)}$ are the order statistics of $U_1,\ldots,U_n$ and study the concentration of an operator derived from such a matrix by embedding it in $[0,1]^2$.\\

First we show a high probability approximation to the position of the order statistics of the random sample $U_{1},U_{2},\ldots,U_{n}$. 
\begin{lemma}\label{lemma:concorderU}
    Let $U_{1},U_{2},\ldots, U_{n}$ be randomly generated from $\text{Unif}[0,1]$. Let $U_{(1)}\leq U_{(2)}\leq \cdots\leq U_{(n)}$ be the arrangement of $\{U_{i}:1\leq i\leq n\}$ in increasing order. Then,
    \begin{align*}
        \P\left(\left|U_{(k)} - \frac{k}{n}\right|>\frac{\log n}{\sqrt{n}}, 1\leq k\leq n\right)\leq 2n\exp\left(-\frac{2}{3}(\log n)^2\right)
    \end{align*}
    for all $n>2$.
    \end{lemma}
    \begin{proof}
    By union bound it is enough to show that for all $1\leq k\leq n$,
    \begin{align*}
        \P\left(\left|U_{(k)} - \frac{k}{n}\right|>\frac{\log n}{\sqrt{n}}\right)\leq 2\exp\left(-\frac{2}{3}(\log n)^2\right)
    \end{align*}
    By Lemma 3.1.1 from \cite{reiss2012approximate} we get,
    \begin{align*}
        \P\left(\left|U_{(k)} - \frac{k}{n+1}\right|>\frac{\vep}{\sqrt{n}}\right)\leq 2\exp\left(-\frac{\vep^2}{3\left(\sigma_{k}^2 + \vep/\sqrt{n}\right)}\right)
    \end{align*}
    with $\sigma_{k}^2 = \frac{k}{n+1}\left(1-\frac{k}{n+1}\right)$. Choosing $\vep = \frac{\log n}{2}$ we have,
    \begin{align*}
        \P\left(\left|U_{(k)} - \frac{k}{n+1}\right|>\frac{\log n}{2\sqrt{n}}\right)\leq 2\exp\left(-\frac{(\log n)^2}{3\left(\sigma_{k}^2 + (\log n)/2\sqrt{n}\right)}\right)
    \end{align*}
    Now observe that $\sigma_{k}^2\leq 1/4$ for all $1\leq k\leq n$ and $\frac{\log n}{2\sqrt{n}}\leq \frac{1}{4}$ for all $n\geq 1$. Then we have,
    \begin{align}\label{eq:Ukconcfirst}
        \P\left(\left|U_{(k)} - \frac{k}{n+1}\right|>\frac{\log n}{2\sqrt{n}}\right)\leq 2\exp\left(-\frac{2}{3}(\log n)^2\right)
    \end{align}
    Finally for all $n>2$, by \eqref{eq:Ukconcfirst} shows,
    \begin{align*}
        \P\left(\left|U_{(k)} - \frac{k}{n}\right|>\frac{\log n}{\sqrt{n}}\right)
        &\leq \P\left(\left|U_{(k)} - \frac{k}{n+1}\right|>\frac{\log n}{2\sqrt{n}}\right)\leq 2\exp\left(-\frac{2}{3}(\log n)^2\right)
    \end{align*}
    \end{proof}
    
    Next, we show that embedding the matrix $\bm \cF_{n}:=\left(\left(f(U_{(i)}, U_{(j)})\right)\right)$ in $[0,1]^2$ gives a good approximation to the function $f$ with high probability.

    \begin{lemma}\label{lemma:WnWdiff}
    For a Lipschitz continuous, symmetric function $f:[0,1]^2\ra\bR$ with Lipschitz constant $L_{f}$ and $U_{1},U_{2}.\ldots, U_{n}$ generated randomly from $\text{Unif}[0,1]$ define,
    \begin{align}\label{eq:defWnfn}
        f_{n}(x,y) = \sum_{i=1}^{n}\sum_{j=1}^{n}f\left(U_{(i)}, U_{(j)}\right)\one\left\{\frac{i-1}{n}<x\leq \frac{i}{n}, \frac{j-1}{n}<y\leq \frac{j}{n}\right\}.
    \end{align}
    Then,
    \begin{align*}
        \sup_{x,y\in [0,1]}|f(x,y) - f_{n}(x,y)|\lesssim_{f} \frac{\log n}{\sqrt{n}}
    \end{align*}
    with probability at least $1-4n\exp\left(-\frac{1}{6}(\log n)^2\right)$
    \end{lemma}
    \begin{proof}
    Fix $(x,y)\in [0,1]^2$ and without loss of generality suppose that $(x,y)\in \left(\frac{i-1}{n}, \frac{i}{n}\right]\times \left(\frac{j-1}{n}, \frac{j}{n}\right]$. Then recalling that $f$ is Lipschitz we have,
    \small
    \begin{align*}
        |f(x,y) - f_{n}(x,y)|\leq L_{f}\left\|(x,y) - (U_{(i)},U_{(j)})\right\|_{2}\leq L_{f}\sqrt{2}\max_{1\leq i\leq n}\left\{\left|U_{(i)} - \frac{i}{n}\right|, \left|U_{(i)} - \frac{i-1}{n}\right|\right\}
    \end{align*}
    \normalsize
    By Lemma \ref{lemma:concorderU} we easily conclude that,
    \begin{align*}
        \P\left(\left|U_{(k)} - \frac{k}{n}\right|>\frac{\log n}{\sqrt{n}}, 1\leq k\leq n\right)\leq 2n\exp\left(-\frac{1}{6}(\log n)^2\right)
    \end{align*}
    and,
    \begin{align*}
        \P\left(\left|U_{(k)} - \frac{k-1}{n}\right|>\frac{\log n}{\sqrt{n}}, 1\leq k\leq n\right)\leq 2n\exp\left(-\frac{1}{6}(\log n)^2\right)
    \end{align*}
    Then,
    \begin{align}\label{eq:whpbddUi}
        \P\left(\max_{1\leq k\leq n}\left\{\left|U_{(k)} - \frac{k-1}{n}\right|, \left|U_{(k)} - \frac{k}{n}\right|\right\}>\frac{\log n}{\sqrt{n}}\right)\leq 4n\exp\left(-\frac{1}{6}(\log n)^2\right)
    \end{align}
    Since our choice of $(x,y)$ was arbitrary, then we can conclude that,
    \begin{align*}
        \P\left(\sup_{x,y\in [0,1]}\left|\sfK(x,y) - \sfK_{n}(x,y)\right|>\frac{\sqrt{2}L\log n}{\sqrt{n}}\right)\leq 4n\exp\left(-\frac{1}{6}(\log n)^2\right)
    \end{align*}
    completing the proof of the lemma.
    \end{proof}

      In the following lemma we show that Hilbert Schmidt operator corresponding to the functions $f$ and $f_{n}$ (defined in \eqref{eq:defWnfn}) are close with high probability.
    \begin{lemma}\label{lemma:TWminusTWn}
    For a Lipschitz symmetric function $f:[0,1]^2\ra\bR$ with lipschitz constant $L_{f}$,
    \begin{align*}
        \|T_f - T_{f_{n}}\|_{2\ra 2}\lesssim_{f} \frac{\log n}{\sqrt{n}}\text{ with probability at least } 1 - 4n\exp\left(-\frac{1}{6}(\log n)^2\right)
    \end{align*}
    where $f_{n}$ is defined in \eqref{eq:defWnfn}.
    \end{lemma}
    \begin{proof}
    By definition,
    \begin{align*}
        \|T_f - T_{f_{n}}\|_{2\ra 2} = \sup_{\|h\|_{2} = \|g\|_{2}=1}\int h(x)(f(x,y) - f_{n}(x,y))g(y)\rd x\rd y
    \end{align*}
    Now by Lemma \ref{lemma:WnWdiff} and Cauchy Schwarz inequality we get,
    \begin{align*}
        \|T_f - T_{f_{n}}\|_{2\ra 2}\leq \frac{2L_{f}\log n}{\sqrt{n}}\sup_{\|h\|_{2} = \|g\|_{2}=1}\int |h(x)g(y)|\rd x\rd y\leq \frac{2L_{f}\log n}{\sqrt{n}}
    \end{align*}
    with probability at least $1-4n\exp\left(-\frac{1}{6}(\log n)^2\right)$
    \end{proof}
    

In the following we prove Lemma 5.4, which is an easy consequence of \Cref{lemma:TWminusTWn}. 
       \begin{proof}[Proof of Lemma 5.4]
    Note that,
    \begin{align*}
        \left\|T_{f_{n}} - T_{f_{n}^{\circ}}\right\|_{2}\leq \frac{2B_{f}}{\sqrt{n}}
    \end{align*}
    The proof is now completed by invoking Lemma \ref{lemma:TWminusTWn} along with the triangle inequality. With probability at least $1 - 4n\exp\left(-\frac{1}{6}(\log n)^2\right)$
 \begin{align*}
        \|T_f - T_{f^\circ_{n}}\|_{2\ra 2}\leq 
        \|T_f - T_{f_{n}}\|_{2\ra 2}+\|T_{f_n} - T_{f^\circ_{n}}\|_{2\ra 2}
        \lesssim_{f} \frac{\log n}{\sqrt{n}}+\frac{2B_f}{\sqrt n}.
    \end{align*}
    \end{proof}
    
    \subsection{Eigenvalues of Uniform Kernel Matrix}

    In this section we study the concentration of sample eigenvalues of kernel matrix $\left(\left(f(U_{(i)}, U_{(j)})\right)\right)_{i\neq j}$. First we prove Lemma 5.2 that the spectrum of the above matrix is same as $\left(\left(f(U_{i}, U_{j})\right)\right)_{i\neq j}$.

    \begin{proof}[Proof of Lemma 5.2]
    \blue{For all $1\leq j\leq n$ let $\text{Rank}(U_{j})$ be the rank of $U_{j}$ among $U_{1},\ldots,U_{n}$.} Consider $\pi:[n]\ra[n]$ to be a permutation such that,
    \begin{align*}
        \left\{\pi^{-1}(j) = \text{Rank}(U_{j}): 1\leq j\leq n\right\}.
    \end{align*}
    Then by definition,
    \begin{align}\label{eq:Upiequalorder}
        U_{\pi(j)} = U_{(j)}, 1\leq j\leq n.
    \end{align}
    Now consider $\Pi_{n}$ to be the permutation matrix corresponding to $\pi$. Then for any matrix $\bA_{n}$ we must have,
    \begin{align*}
        \Pi_{n}\bA_{n}\Pi_{n}^{\top} = \left(\left(\bA_{n}(\pi(i),\pi(j))\right)\right)_{1\leq i,j\leq n}
    \end{align*}
    which now completes the proof. 
    \end{proof}
    
 Recall $\bm F_{n}:= ((f(U_{i},U_{j})))_{i\neq j=1}^{n}$, next we prove 
 Lemma 5.3, which states that the largest eigenvalue of $\bm F_n$ concentrates around $\lambda_1(f)$.
    
    \begin{proof}[Proof of Lemma 5.3]
    Recalling $(5.2)$ and Lemma 5.2 it is easy to note that $\lambda$ is an eigenvalue of $\bm F_{n}$ if and only if $\lambda/n$ is an eigenvalue of the operator $T_{f_{n}^{\circ}}$. By Lemma 5.4 we have,
    \begin{align*}
        \|T_{f} - T_{f_{n}^{\circ}}\|_{2\ra 2}\leq C_{f}\frac{\log n}{\sqrt{n}}\text{ with probability at least }1-4n\exp\left(-\frac{1}{6}(\log n)^2\right).
    \end{align*}
    Observe that,
    \begin{align*}
        \P\left(\lambda_{1}(\bm F_{n})\leq 0\right)
        & \leq \P\left(\lambda_{1}(\bm F_{n})\leq 0, \|T_{f} - T_{f_{n}}^{\circ}\|_{2\ra 2}\leq C_{f}\frac{\log n}{\sqrt{n}}\right) + 4n\exp\left(-\frac{1}{6}(\log n)^2\right)\\
        & \leq \P\left(|\lambda_{1}(f)|\leq C_{f}\frac{\log n}{\sqrt{n}}\right)+ 4n\exp\left(-\frac{1}{6}(\log n)^2\right)
    \end{align*}
    where the last inequality follows by noting that on the event $\lambda_{1}(\bm F_{n})\leq 0$ the operator $T_{f_{n}^{\circ}}$ has no positive eigenvalues and invoking Lemma \ref{l:eigdist}. Then for large enough $n$,
    \begin{align*}
        \lambda_{1}(\bm F_{n})>0\text{ with probability at least }1-4n\exp\left(-\frac{1}{6}(\log n)^2\right).
    \end{align*}
    Now once again invoking Lemma \ref{l:eigdist} and Lemma 5.4 we get,
    \begin{align*}
        \P\bigg(\bigg|\frac{\lambda_{1}(\bm F_{n})}{n}
        & - \lambda_{1}(f)\bigg|> C_{f}\frac{\log n}{\sqrt{n}}\bigg)\\
        &\leq \P\left(\left|\frac{\lambda_{1}(\bm F_{n})}{n} - \lambda_{1}(f)\right|> C_{f}\frac{\log n}{\sqrt{n}}, \lambda_{n}(\bm F_{n})>0\right) + 4n\exp\left(-\frac{1}{6}(\log n)^2\right)\\
        &\leq \P\left(\|T_{f} - T_{f_{n}^{\circ}}\|>C_{f}\frac{\log n}{\sqrt{n}}\right) + 4n\exp\left(-\frac{1}{6}(\log n)^2\right)\\
        &\leq 8n\exp\left(-\frac{1}{6}(\log n)^2\right)
    \end{align*}
    for all large enough $n$, thus completing the proof of the lemma.
    \end{proof}

    \section{Proof of Proposition 5.1}\label{s:kernelproof2}
    Consider,
     \begin{align}\label{eq:defftilde}
         \widetilde{f}(x,y): = f(x,y) - \lambda_{1}(f)\phi_{1}(x)\phi_{1}(y)
     \end{align}
     Then it is easy to observe that,
     \begin{align*}
         \widetilde{\bF}_{n}: = \bF_{n} - \lambda_{1}(f)\left(\Phi_{1}(\bU)\Phi_{1}(\bU)^{\top}- \bD_{n}\right) = \left(\left(\widetilde{f}(U_{i},U_{j})\right)\right)_{i\neq j}
     \end{align*}
     By definition $\bm X_{n} = \blambda_{n}\mathbb{I}_{n} - \left(\widetilde{\bF}_{n} - \lambda_{1}(f)\bD_{n}\right)$. Note that proving the lemma amounts to showing,
     \begin{align}\label{eq:toshowlambdalower}
         \inf_{1\leq i\leq n}\left|\frac{\blambda_{n}}{n} - \frac{1}{n}\lambda_{i}\left(\widetilde{\bF}_{n} - \lambda_{1}(f)\bD_{n}\right)\right| \geq \frac{1}{2}|\lambda_{1}(f) - \lambda_{2}(f)|
     \end{align}
     with probability at least $1 - 16n\exp\left(-\frac{1}{6}(\log n)^2\right)$ for large enough $n$. With that goal in mind, first we show a lower bound on L.H.S of \eqref{eq:toshowlambdalower}.
     \begin{lemma}\label{lemma:firstlambdalbd}
         Let $U_{(1)}\leq \cdots\leq U_{(n)}$ be the non-decreasing ordering of $U_{1},\ldots, U_{n}$. Recalling $\widetilde{f}$ defined in \eqref{eq:defftilde}, consider $\widetilde{\bF}_{n}^{\perm}$ to be a $n\times n$ matrix with $0's$ on the diagonal and the $(i,j)^{th}$ entry given by $\widetilde{f}(U_{(i)},U_{(j)})$ for all $1\leq i,j\leq n$. Then,
         \small
         \begin{align*}
             \inf_{1\leq i\leq n}\left|\frac{\blambda_{n}}{n} - \frac{1}{n}\lambda_{i}\left(\widetilde{\bF}_{n} - \lambda_{1}(f)\bD_{n}\right)\right|\geq |\lambda_{1}(f) - \lambda_{2}(f)| - \|T_{h_{\widetilde{\bF}_{n}^{\perm}}} - T_{\widetilde{f}}\|_{2\ra 2} - \left|\frac{\blambda_{n}}{n} - \lambda_{1}(f)\right|
         \end{align*}
         \normalsize
         where,
         \begin{align}\label{eq:hFntilde}
             h_{\widetilde{\bF}_{n}^{\perm}}(x,y) = \sum_{i\neq j}\widetilde{f}\left(U_{(i)}, U_{(j)}\right)\one\left\{\frac{i-1}{n}<x\leq \frac{i}{n}, \frac{j-1}{n}<y\leq \frac{j}{n}\right\}
         \end{align}
         and $T_{\gamma}$ is the Hilbert Schmidt integral operator corresponding to the function $\gamma$.
     \end{lemma}
     The proof of Lemma \ref{lemma:firstlambdalbd} is given in Section \ref{section:proofoffirstlambdalbd}. Recalling the bound from $(5.3)$ and Lemma \ref{lemma:firstlambdalbd} the natural next step is to show an upper bound on $\|T_{h_{\widetilde{\bF}_{n}^{\perm}}} - T_{\widetilde{f}}\|_{2\ra 2}$. Towards that define the matrix $\bF_{n}^{\perm}$ to be a $n\times n$ matrix with the $(i,j)^{th}$ entry given by $f(U_{(i)},U_{(j)})$ for all $1\leq i,j\leq n$. Now similar to \eqref{eq:hFntilde} consider,
     \begin{align}\label{eq:hFnperm}
         h_{\bF_{n}^{\perm}}(x,y) = \sum_{i\neq j}f\left(U_{(i)}, U_{(j)}\right)\one\left\{\frac{i-1}{n}<x\leq \frac{i}{n}, \frac{j-1}{n}<y\leq \frac{j}{n}\right\}
     \end{align}
     By triangle inequality note that,
     \begin{align}\label{eq:bddonTdiff}
         \left\|T_{h_{\widetilde{\bF}_{n}^{\perm}}} - T_{\widetilde{f}}\right\|_{2\ra 2}\leq \left\|h_{\widetilde{\bF}_{n}^{\perm}} - h_{\bF_{n}^{\perm}} - \widetilde{f} + f\right\|_{2} + \left\|T_{h_{\bF_{n}^{\perm}}} - T_{f}\right\|_{2\ra 2}
     \end{align}
     By Lemma $5.4$ it is now enough to have a bound on $\left\|h_{\widetilde{\bF}_{n}^{\perm}} - h_{\bF_{n}^{\perm}} - \widetilde{f} + f\right\|_{2}$, which is provided in the following result.
     \begin{lemma}\label{lemma:upbdong}
         Recalling \eqref{eq:hFntilde} and \eqref{eq:hFnperm} we have,
         \begin{align*}
             \left\|h_{\widetilde{\bF}_{n}^{\perm}} - h_{\bF_{n}^{\perm}} - \widetilde{f} + f\right\|_{2}\lesssim_{f}\frac{\log n}{\sqrt{n}}
         \end{align*}
         with probability at least $1 - 4n\exp\left(-\frac{1}{6}(\log n)^2\right)$.
     \end{lemma}
     The proof of Lemma \ref{lemma:upbdong} is given in Section \ref{sec:proofofupbdong}. Now we are ready to complete the proof of Proposition $5.1$. Lemma \ref{lemma:upbdong}, Lemma $5.4$ and \eqref{eq:bddonTdiff} combines to show,
     \begin{align*}
         \left\|T_{h_{\widetilde{\bF}_{n}^{\perm}}} - T_{\widetilde{f}}\right\|_{2\ra 2}\lesssim_{f}\frac{\log n}{\sqrt{n}}
     \end{align*}
     with probability at least $1 - 8n\exp\left(-\frac{1}{6}(\log n)^2\right)$. Recall the lower bound from Lemma \ref{lemma:firstlambdalbd}, then using $(5.3)$ we get,
     \begin{align*}
         \inf_{1\leq i\leq n}\left|\frac{\blambda_{n}}{n} - \frac{1}{n}\lambda_{i}\left(\widetilde{\bF}_{n} - \lambda_{1}(f)\bD_{n}\right)\right|\geq |\lambda_{1}(f) - \lambda_{2}(f)| - C_{f}\frac{\log n}{\sqrt{n}}
     \end{align*} 
     with probability at least $1 - 16n\exp\left(-\frac{1}{6}(\log n)^2\right)$. The proof is now completed by noting that R.H.S in the above inequality is lower bounded by $|\lambda_{1}(f) - \lambda_{2}(f)|/2$ for large enough $n$.

 \subsection{Proof of Lemma \ref{lemma:firstlambdalbd}}\label{section:proofoffirstlambdalbd}
     Observe that $\lambda$ is an eigenvalue of $\widetilde{\bF}_{n}^{\perm}$ if and only if $\lambda/n$ is an eigenvalue of the operator $T_{h_{\widetilde{\bF}_{n}^{\perm}}}$ and similarly $\lambda$ is an eigenvalue of $\bF_{n}^{\perm}$ if and only if $\lambda/n$ is an eigenvalue of $T_{h_{\bF_{n}^{\perm}}}$. Now consider,
     \begin{align*}
         \lambda_{1}(h_{\widetilde{\bF}_{n}^{\perm}})\geq \lambda_{2}(h_{\widetilde{\bF}_{n}^{\perm}})\geq\cdots\geq 0\text{ and }\lambda_{1}^{\prime}(h_{\widetilde{\bF}_{n}^{\perm}})\leq \lambda_{2}^{\prime}(h_{\widetilde{\bF}_{n}^{\perm}})\leq\cdots\leq 0
     \end{align*}
     be the collection of positive and negative eigenvalues (padded with $0's$) of $T_{f_{\widetilde{\bF}_{n}^{\perm}}}$. Similarly let,
     \begin{align*}
         \lambda_{1}(\widetilde{f})\geq \lambda_{2}(\widetilde{f})\geq \cdots\geq 0\text{ and }\lambda_{1}^{\prime}(\widetilde{f})\leq \lambda_{2}^{\prime}(\widetilde{f})\leq\cdots\leq 0
     \end{align*}
     be the collection of positive and negative eigenvalue of $T_{\widetilde{f}}$. For an arbitray eigenvalue $\lambda(h_{\widetilde{\bF}_{n}^{\perm}})$ define,
     \begin{align}\label{eq:lambdatildef}
         \lambda(\widetilde{f}) = 
         \begin{cases}
             \lambda_{j}(\widetilde{f}) & \text{ if } \lambda(h_{\widetilde{\bF}_{n}^{\perm}}) = \lambda_{j}(h_{\widetilde{\bF}_{n}^{\perm}})\text{ for some }j\in\mathbb{N}\\
             \lambda_{j}^{\prime}(\widetilde{f}) & \text{ if } \lambda(h_{\widetilde{\bF}_{n}^{\perm}}) = \lambda_{j}^{\prime}(h_{\widetilde{\bF}_{n}^{\perm}})\text{ for some }j\in\mathbb{N}
         \end{cases}
     \end{align}
     For an operator $T$ let $\sigma(T)$ denote the collection of eigenvalue of $T$. By definition, $T_{h_{\widetilde{\bF}_{n}^{\perm}}}$ and $T_{\widetilde{f}}$ are self-adjoint compact operators. Then by Lemma \ref{l:eigdist} we get,
     \begin{align}\label{eq:diffwyloplambda}
         \left|\lambda(h_{\widetilde{\bF}_{n}^{\perm}}) - \lambda(\widetilde{f})\right|\leq \left\|T_{h_{\widetilde{\bF}_{n}^{\perm}}} - T_{\widetilde{f}}\right\|_{2\ra 2}
     \end{align}
     Now recall that $\widetilde{\bF}_{n}$ and $\widetilde{\bF}_{n}^{\perm}$ has the same spectrum. Then by Weyl's inequality and Lemma $5.1$ we get,
     \begin{align}
         \inf_{1\leq i\leq n}\left|\frac{\blambda_{n}}{n} - \frac{1}{n}\lambda_{i}\left(\widetilde{\bF}_{n} - \lambda_{1}(f)\bD_{n}\right)\right|
         & \geq \inf_{1\leq i\leq n}\frac{1}{n}\left|\blambda_{n} - \lambda_{i}(\widetilde{\bF}_{n})\right| - \frac{C_{f}}{n}\nonumber\\
         & \geq \inf_{1\leq i\leq n}\frac{1}{n}\left|\blambda_{n} - \lambda_{i}(\widetilde{\bF}_{n}^{\perm})\right| - \frac{C_{f}}{n}\label{eq:lowerbdeq1}
     \end{align}
     for some constant $C_{f}>0$ depending on the function $f$. Once again recalling the equivalence between eigenvalues of $\widetilde{\bF}_{n}^{\perm}$ and the operator $T_{h_{\widetilde{\bF}_{n}^{\perm}}}$ note that,
     \begin{align}
         \inf_{1\leq i\leq n}\frac{1}{n}\left|\blambda_{n} - \lambda_{i}(\widetilde{\bF}_{n}^{\perm})\right|
         & \geq \inf_{\sigma\left(T_{h_{\widetilde{\bF}_{n}^{\perm}}}\right)}\left|\frac{\blambda_{n}}{n} - \lambda(h_{\widetilde{\bF}_{n}^{\perm}})\right|\label{eq:lowerbdeq2}
     \end{align}
     Considering an arbitrary eigenvalue $\lambda(h_{\widetilde{\bF}_{n}^{\perm}})$ and using \eqref{eq:diffwyloplambda} observe that,
     \begin{align}
         \left|\frac{\blambda_{n}}{n} - \lambda(h_{\widetilde{\bF}_{n}^{\perm}})\right|
         & \geq \left|\lambda_{1}(f) - \lambda(h_{\widetilde{\bF}_{n}^{\perm}})\right| - \left|\frac{\blambda_{n}}{n} - \lambda_{1}(f)\right|\nonumber\\
         &\geq \left|\lambda_{1}(f) - \lambda(\widetilde{f})\right| - \|T_{h_{\widetilde{\bF}_{n}^{\perm}}} - T_{f_{\widetilde{f}}}\|_{2\ra 2} - \left|\frac{\blambda_{n}}{n} - \lambda_{1}(f)\right|\nonumber\\
         &\geq d(\lambda_{1}(f), T_{\widetilde{f}}) - \|T_{h_{\widetilde{\bF}_{n}^{\perm}}} - T_{f_{\widetilde{f}}}\|_{2\ra 2} - \left|\frac{\blambda_{n}}{n} - \lambda_{1}(f)\right|\label{eq:lowerbdeq3}
     \end{align}
     where,
     \begin{align*}
         d(\lambda_{1}(f), T_{\widetilde{f}}):=\inf\{|\lambda_{1}(f) - \lambda_{j}(\widetilde{f})|,|\lambda_{1}(f) - \lambda_{j}^{\prime}(\widetilde{f})|, j\geq 1\}
     \end{align*}
     Recalling definition of $\widetilde{f}$ it follows that $d(\lambda_{1}(f), T_{\widetilde{f}}) = |\lambda_{1}(f) - \lambda_{2}(f)|>0$, which now completes the proof.
 
     \subsection{Proof of Lemma \ref{lemma:upbdong}}\label{sec:proofofupbdong}
     Define,
     \begin{align*}
         g = f_{\widetilde{\bF}_{n}^{\perm}} - f_{\bF_{n}^{\perm}} - \widetilde{f} + f
     \end{align*}
     Then,
     \begin{align}\label{eq:normgsq}
         \|g\|_{2}^2 = \sum_{i\neq j}\int_{I_{i}\times I_{j}}g(x,y)^2\mathrm{d}x\mathrm{d}y + \sum_{\ell=1}^{n}\int_{I_{\ell}\times I_{\ell}}g(x,y)^2\mathrm{d}x\mathrm{d}y.
     \end{align}
     Suppose $i\neq j$ and consider $(x,y)\in I_{i}\times I_{j}$. Then by definition,
     \begin{align*}
         g(x,y) = \lambda_{1}(f)\left(\phi_{1}(x)\phi_{1}(y) - \phi_{1}(U_{(i)})\phi_{1}(U_{(j)})\right)
     \end{align*}
     and hence,
     \begin{align*}
         |g(x,y)|
         &\leq |\lambda_{1}(f)|\left[|\phi_{1}(x)|\left|\phi_{1}(y) - \phi_{1}(U_{(j)})\right| + |\phi_{1}(U_{(j)})|\left|\phi_{1}(x) - \phi_{1}(U_{(i)})\right|\right]\\
         &\leq B_{f}\left|\phi_{1}(y) - \phi_{1}(U_{(j)})\right| + B_{f}\left|\phi_{1}(x) - \phi_{1}(U_{(i)})\right|
     \end{align*}
     where the last inequality follows by noting the bound from Lemma $5.1$. Recalling that $\phi_{1}$ is Lipschitz from Lemma $5.1$ we conclude that,
     \begin{align*}
         |g(x,y)|\leq L_{1,f}\max_{i=1}^{n}\left\{\left|U_{(i)} - \frac{i}{n}\right|, \left|U_{(i)} - \frac{i-1}{n}\right|\right\}
     \end{align*}
     where $L_{1,f} = 2BL_{\phi_{1}}$ with $L_{\phi_{1}}$ the Lipschitz constant of $\phi_{1}$. Now if $(x,y)\in I_{i}\times I_{i}$ then,
     \begin{align*}
         g(x,y) = \lambda_{1}(f)\phi_{1}(x)\phi_{1}(y)
     \end{align*}
     Recalling \eqref{eq:normgsq} we get,
     \begin{align*}
         \|g\|_{2}^2\leq \frac{n(n-1)}{n^2}\left(L_{1,f}\max_{i=1}^{n}\left\{\left|U_{(i)} - \frac{i}{n}\right|, \left|U_{(i)} - \frac{i-1}{n}\right|\right\}\right)^2 + B_{f}^2|\lambda_{1}(f)|^{-2}\frac{1}{n}
     \end{align*}
     Then,
     \begin{align*}
         \|g\|_{2}^2\lesssim_{f}\left(\max_{i=1}^{n}\left\{\left|U_{(i)} - \frac{i}{n}\right|, \left|U_{(i)} - \frac{i-1}{n}\right|\right\}^2 + \frac{1}{n}\right)
     \end{align*}
     The proof is now concluded by recalling \eqref{eq:whpbddUi}.
 
 \section{Proof of Results from Section $6$}\label{s:Kerneleigproof}
     \subsection{Proof of Lemma $6.3$}\label{section:proofofnormMnbdd}
     By $(6.27)$ observe that,
         \begin{align}\label{eq:bddtraceM}
             \left\|\frac{\bM_{n}}{\blambda_{n}}\right\|_{2\ra 2}\leq \frac{\sum_{\ell=1}^{k} |\lambda_{\ell}(f)|^{-1}}{|\blambda_{n}|}.
         \end{align}
         Now by $(6.33)$ recall that for large enough $n$, 
         \begin{align*}
             \left|\frac{\blambda_{n}}{n} - \lambda_{1}(f)\right|\leq \frac{C\log n}{\sqrt{n}}
         \end{align*}
         with probability at least $1-8n\exp\left(-\frac{1}{6}(\log n)^2\right).$ Thus for large enough $n$,
         \begin{align*}
             \left\|\frac{\bM_{n}}{\blambda_{n}}\right\|_{2\ra 2}\leq \frac{\sum_{\ell=1}^{k} |\lambda_{\ell}(f)|^{-1}}{2n|\lambda_{1}(f)| - C\sqrt{n}\log n}<1
         \end{align*}
         with probability at least $1-8n\exp\left(-\frac{1}{6}(\log n)^2\right).$
 
         \subsection{Proof of Lemma $6.4$}\label{section:proofofconcVTAV}
         By a Taylor series expansion of $\bcA_{n}$ and Lemma $6.3$ note that,
             \begin{align*}
                 \bV^{\top}\bcA_{n}^{-1}\bV = \frac{\bV^{\top}\bV}{\blambda_{n}} + \sum_{\ell=1}^{\infty}(-1)^{\ell}\frac{\bV^{\top}\bM_{n}^{\ell}\bV}{\blambda_{n}^{\ell+1}}
             \end{align*}
             
             Observe that,
             \begin{align*}
                 \left\|\frac{\bV^{\top}\bV}{\blambda_{n}} - \frac{\bI_{k-1}}{\lambda_{1}(f)}\right\|_{2\ra 2}\leq \left(\sum_{i=2}^{k}\sum_{j=2}^{k}\left(\frac{\Phi_{i}(\bU_{n})^{\top}\Phi_{j}(\bU_{n})}{\blambda_{n}} - \frac{\delta_{ij}}{\lambda_{1}(f)}\right)^2\right)^{\frac{1}{2}}
             \end{align*}
             For fixed $2\leq i,j\leq k$ by definition we get,
             \begin{align*}
                 \frac{\Phi_{i}(\bU_{n})^{\top}\Phi_{j}(\bU_{n})}{\blambda_{n}} = \frac{n}{\blambda_{n}}\frac{1}{n}\sum_{\ell=1}^{n}\phi_{i}(U_{\ell})\phi_{j}(U_{\ell})
             \end{align*}
             By Lemma $5.1$ and Hoeffding's inequality we have,
             \begin{align*}
                 \left|\frac{1}{n}\sum_{\ell=1}^{n}\phi_{i}(U_{\ell})\phi_{j}(U_{\ell}) - \delta_{ij}\right|\lesssim_{f}\frac{\log n}{\sqrt{n}}
             \end{align*}
             with probability $1-2\exp(-\frac{1}{6}(\log n)^2)$. Now recalling the bound from $(6.33)$,
             \begin{align*}
                 \left|\frac{\Phi_{i}(\bU_{n})^{\top}\Phi_{j}(\bU_{n})}{\blambda_{n}} - \frac{\delta_{ij}}{\lambda_{1}(f)}\right|\lesssim_{f} \frac{\log n}{\sqrt{n}}
             \end{align*}
             with probability at least $1 - 9n\exp\left(-\frac{1}{6}(\log n)^2\right)$ for large enough $n$. Using an union bound argument we get, 
             \begin{align}\label{eq:concVTVI}
                 \left\|\frac{\bV^{\top}\bV}{\blambda_{n}} - \frac{\bI_{k-1}}{\lambda_{1}(f)}\right\|_{2\ra 2}\lesssim_{f} \frac{k\log n}{\sqrt{n}}
             \end{align}
             with probability at least $1 - 9nk\exp( - \frac{1}{6}(\log n)^2)$. Recalling that $\bM_{n}$ is a diagonal matrix and using $(6.27)$ note that for all $\ell\geq 1$,
             \begin{align*}
                 \left\|\bV^{\top}\bM^{\ell}\bV\right\|_{2\ra 2}\leq \left(\sum_{\ell=1}^{k}|\lambda_{\ell}(f)|^{-1}\right)^{\ell}\|\bV\|_{F}^2\leq C_{f}^{\ell}nk
             \end{align*}
             where the last inequality follows by the bounds from Lemma $5.1$. Thus by the bounds from $(6.33)$,
             \begin{align}\label{eq:concVTVIrem}
                 \sum_{\ell=1}^{\infty}\frac{\left\|\bV^{\top}\bM_{n}^{\ell}\bV\right\|_{2\ra 2}}{|\blambda_{n}|^{\ell+1}}\leq k\sum_{\ell=1}^{\infty}\frac{C_{f}^{\ell}}{n^{\ell}}\left|\frac{n}{\blambda_{n}}\right|^{\ell+1} \lesssim_{f}\frac{k}{n}
             \end{align}
             with probability at least $1-8n\exp\left(-\frac{1}{6}(\log n)^2\right)$ for large enough $n$. Now combining \eqref{eq:concVTVI} and \eqref{eq:concVTVIrem} we conclude that,
             \begin{align*}
                 \left\|\bV^{\top}\bcA_{n}^{-1}\bV - \frac{\bI_{k-1}}{\lambda_{1}(f)}\right\|_{2\ra 2}\lesssim_{f}k\left(\frac{\log n}{\sqrt{n}} + \frac{1}{n}\right)
             \end{align*}
             with probability at least $1-17nk\exp\left(-\frac{1}{6}(\log n)^2\right)$
         
         \subsection{Proof of Lemma $6.5$}\label{sec:proofofform1tAinvV}
         
         Once again by a Taylor series expansion of $\bcA_{n}$ and Lemma $6.3$ note that,
             \begin{align*}
                 \Phi_{1}(\bU_{n})^{\top}\bcA_{n}^{-1}\bV = \frac{\Phi_{1}(\bU_{n})^{\top}\bV}{\blambda_{n}} + \sum_{\ell=1}^{\infty}(-1)^{\ell}\frac{\Phi_{1}(\bU_{n})^{\top}\bM_{n}^{\ell}\bV}{\blambda_{n}^{\ell+1}}
             \end{align*}
             with probability at least $1-8n\exp\left(-\frac{1}{6}(\log n)^2\right)$. Now recalling the bound on eigenfunctions from Lemma $5.1$ we have,
             \begin{align*}
                 \left\|\Phi_{1}(\bU_{n})^{\top}\bM_{n}^{\ell}\bV\right\|_{2}^{2}\leq \|\Phi_{1}(\bU_{n})\|_{2}^2\left\|\bV\right\|_{F}^2\left\|\bM_{n}^{\ell}\right\|_{2\ra 2}^{2}\leq C_{f}^{2\ell}k n^2 .
             \end{align*}
             for some constant $C_f$ depending on $f$. Thus recalling the bounds from $(6.33)$,
             \begin{align*}
                 \left\|\sum_{\ell=1}^{\infty}(-1)^{\ell}\frac{\Phi_{1}(\bU_{n})^{\top}\bM_{n}^{\ell}\bV}{\blambda_{n}^{\ell}}\right\|_{2} 
                 \leq \frac{\sqrt{k}n}{\left|\blambda_{n}\right|}\sum_{\ell=0}^{\infty}\frac{C_f^{\ell}}{n^{\ell}}\frac{n^{\ell}}{\left|\blambda_{n}\right|^{\ell}}\lesssim_{f}\sqrt{k}
             \end{align*}
             with probability at least $1-8n\exp\left(-\frac{1}{6}(\log n)^2\right)$.

 \section{Proof of Results from Section $7$} \label{s:Matrixeigproof}
 Before proceeding with the proofs we first introduce some notation which will be used throughout this section. Recalling $U_{1}, U_{2},\ldots, U_{n}$ consider the permutation matrix $\Pi_{n}$ from Lemma $5.2$. We define, $\bu_{n}^{\Pi} = \Pi_{n}\bu$ for any vector $\bu\in \bR^{n}$ and $\bm S_{n}^{\Pi} = \Pi(\bm S_n)\Pi^{\top}$ for any matrix $\bm S_n$. Further for any vector $\bu\in \bR^{n}$ we consider a functional embedding on $[0,1]$ as,
 \begin{align*}
     f_{\bu_{n}}(x) = \sum_{i=1}^{n}\sqrt{n}u_{j}\one\left[x\in I_{j}\right]\text{ where }I_{j} = \left[\frac{j-1}{n}, \frac{j}{n}\right), 1\leq j\leq n.
 \end{align*}
 By definition notice that for two vectors $\bu_{n,1}$ and $\bu_{n,2}$,
 \begin{align}\label{eq:equivfvector}
     \left\|f_{\bu_{n,1}} - f_{\bu_{n,2}}\right\|_{2} = \|\bu_{n,1} - \bu_{n,2}\|_{2}.
 \end{align}

 \subsection{Proof of Lemma $7.1$} Consider $\pi$ to be the permutation corresponding to the permutation matrix $\Pi$. Then,
 \begin{align*}
     \left\|f_{\bPhi_{1}^{\Pi}} - \phi_{1}\right\|_{2}^2
     &\lesssim \sum_{j=1}^{n}\int_{I_{j}}|\phi_{1}(t) - \phi_{1}(j/n)|^2\rd t + \int_{I_{j}}|\phi_{1}(j/n) - \phi_{i}(U_{\pi(j)})|^2\rd t\\
     &\lesssim_{W}\frac{1}{n^2} + \frac{1}{n}\sum_{j=1}^{n}|U_{\pi(j)} - j/n|^2
 \end{align*}
 where the last step uses the Lipschitz property of $W$. By Lemma \ref{lemma:concorderU} and \eqref{eq:Upiequalorder} we know,
 \begin{align*}
     |U_{\pi(j)} - j/n| = |U_{(j)} - j/n|\leq \frac{\log n}{\sqrt{n}}, 1\leq j\leq n,
 \end{align*}
 with probability at least $1-2n\exp(-2(\log n)^2/3)$. Combining we conclude,
 \begin{align*}
     \left\|f_{\bPhi_{1}^{\Pi}} - \phi_{1}\right\|_{2}\lesssim_{W} \frac{\log n}{\sqrt{n}}
 \end{align*}
 with probability at least $1-2n\exp(-2(\log n)^2/3)$
 
 \subsection{Proof of Proposition $7.1$}
 Consider the matrix $\bW_{n}^{\perm} = \Pi_{n}\bW_{n}\Pi_{n}$. Then define the function,
 \begin{align*}
     h_{\bW_{n}^{\perm}} = \sum_{i\neq j}W\left(U_{(i)}, U_{(j)}\right)\one\left\{\frac{i-1}{n}<x\leq \frac{i}{n}, \frac{j-1}{n}<y\leq \frac{j}{n}\right\}
 \end{align*}
 Now for the functions $W$ and $h_{\bW_{n}^{\perm}}$ consider the Hilbert-Schmidt operators $T_{W}$ and $T_{h_{\bW_{n}^{\perm}}}$ as defined in \eqref{eq:defHilbertf}. By definition it is now easy to note that, $\lambda_{1}(\bW_{n})/n$ is an eigenvalue of $T_{h_{\bW_{n}^{\perm}}}$ with eigenfunction $f_{\bv_{1}^{\Pi}}$. Now consider the following operators,
 \begin{align*}
     \Delta = T_{h_{\bW_{n}^{\perm}}} - T_{W}
 \end{align*}
 and let $P$ be the Hilbert Schmidt operator with kernel $k_{P} = \lambda_{1}\phi_{1}(x)\phi_{1}(y)$. Define,
 \begin{align*}
     T_{0} = T_{W} - P + \Delta.
 \end{align*}
 Then $T_{h_{\bW_{n}^{\perm}}} = P + T_{0}$. Now note that by definition,
 \begin{align}\label{eq:v1first}
     (T_{h_{\bW_{n}^{\perm}}} - T_{0})f_{\bv_{1}^{\Pi}}(\cdot) = \lambda_{1}(W)\phi_{1}(\cdot)\langle\phi_{1},f_{\bv_{1}^{\Pi}}\rangle
 \end{align}
 Further by recalling that $\lambda_{1}(\bW_{n})/n$ is an eigenvalue of $T_{h_{\bW_{n}^{\perm}}}$ we get,
 \begin{align}\label{eq:v1second}
     (T_{h_{\bW_{n}^{\perm}}} - T_{0})f_{\bv_{1}^{\Pi}} = \left(\frac{\lambda_{1}(\bW_{n})}{n} - T_{0}\right)f_{\bv_{1}^{\Pi}}
 \end{align}
 In the following we first show that $\lambda_{1}(\bW_{n})/n\not\in\sigma(T_W-P)$ with high probability. Note that,
 \begin{align*}
     \sigma(T_{W} - P) = \left\{\lambda_{j}(W):j\neq 1\right\}\bigcup\{\lambda_{j}^{\prime}(W):j\geq 1\}\bigcup\{0\}.
 \end{align*}
 For any $\lambda\in \sigma(T_{W} - P)$,
 \begin{align*}
     \left|\frac{\lambda_{1}(\bW_{n})}{n} - \lambda\right|
     &\geq|\lambda_{1}(W) - \lambda| -\left|\frac{\lambda_{1}(\bW_{n})}{n} - \lambda_{1}(W)\right|\\
     &\geq \min\{|\lambda_{1}(W) - \lambda_{2}(W)|,|\lambda_{1}(W)|\} - \left|\frac{\lambda_{1}(\bW_{n})}{n} - \lambda_{1}(W)\right|.
 \end{align*}
 Then by Lemma $5.3$, for large enough $n$,
 \begin{align}\label{eq:lbbTWP}
     \mathrm{dist}\left(\frac{\lambda_{1}(\bW_{n})}{n}, \sigma(T_{W} - P)\right)\geq \min\{|\lambda_{1}(W) - \lambda_{2}(W)|,|\lambda_{1}(W)|\}/2>0
 \end{align}
 with probability at least $1-8n\exp\left(-\frac{1}{6}(\log n)^2\right)$. Then by Lemma \ref{lemma:normbddresolvent} we get,
 \begin{align}\label{eq:Minversebdd}
     \left\|\left(\frac{\lambda_{1}(\bW_{n})}{n} - T_{W} + P\right)^{-1}\right\|_{2\ra 2} \lesssim_{W}1
 \end{align}
 with probability at least $1-8n\exp\left(-\frac{1}{6}(\log n)^2\right)$. Now recalling the expnasion of the resolvent from \eqref{eq:resolventexpand} it is now easy to see that,
 \begin{align}\label{eq:Minversephi1}
     \left(\frac{\lambda_{1}(\bW_{n})}{n} - T_{W} + P\right)^{-1}\phi_{1} = \frac{n}{\lambda_{1}(\bW_{n})}\phi_{1}
 \end{align} 
 with probability at least $1-8n\exp\left(-\frac{1}{6}(\log n)^2\right)$. Additionally following the arguments from \eqref{eq:diffwyloplambda}, in particular considering corresponding eigenvalues of $T_{W}-P$ and $T_{h_{\bW_{n}^{\perm}}}-P$ as in \eqref{eq:diffwyloplambda} it can be showed that with probability at least $1-12n\exp\left(-\frac{1}{6}(\log n)^2\right)$,
 \begin{align*}
     \mathrm{dist}\left(\frac{\lambda_{1}(\bW_{n})}{n}, \sigma(T_{h_{\bW_{n}^{\perm}}} - P)\right)
     &\geq \mathrm{dist}\left(\frac{\lambda_{1}(\bW_{n})}{n}, \sigma(T_{W} - P)\right) - \left\|T_{W} - T_{h_{\bW_{n}^{\perm}}}\right\|\\
     &\geq \min\{|\lambda_{1}(W) - \lambda_{2}(W)|,|\lambda_{1}(W)|\}/4
 \end{align*}
 where the last inequality follows from \eqref{eq:lbbTWP} and Lemma $5.4$. Then once again using Lemma \ref{lemma:normbddresolvent} shows,
 \begin{align}\label{eq:Linversebdd}
     \left\|\left(\frac{\lambda_{1}(\bW_{n})}{n} - T_{h_{\bW_{n}^{\perm}}} + P\right)^{-1}\right\|_{2\ra 2} \lesssim_{W}1
 \end{align}
 with probability at least $1-12n\exp\left(-\frac{1}{6}(\log n)^2\right)$. Now combining \eqref{eq:v1first},\eqref{eq:v1second} with the bounds from \eqref{eq:Minversebdd}, \eqref{eq:Linversebdd} and the equality from \eqref{eq:Minversephi1} along with the identity,
 \footnotesize
 \begin{align*}
     \left(\frac{\lambda_{1}(\bW_{n})}{n} - T_{0}\right)^{-1} = \left(\frac{\lambda_{1}(\bW_{n})}{n} - T_{W} + P\right)^{-1} + \left(\frac{\lambda_{1}(\bW_{n})}{n} - T_{0}\right)^{-1}\Delta\left(\frac{\lambda_{1}(\bW_{n})}{n} - T_{W} + P\right)^{-1}
 \end{align*}
 \normalsize
 shows,
 \begin{align*}
     \left|\left\|f_{\bv_{1}^{\Pi}}\right\|_{2} - \left|\frac{\lambda_{1}(W)n}{\lambda_{1}(\bW_{n})}\langle\phi_{1},f_{\bv_{1}^{\Pi}}\rangle\right|\right|\lesssim_{W}\|\Delta\|_{2\ra 2}
 \end{align*}
 with probability at least $1-28n\exp\left(-\frac{1}{6}(\log n)^2\right)$. Finally recalling the approximation from Lemma $5.4$ and \eqref{eq:equivfvector} shows,
 \begin{align*}
     \left|\left\|f_{\bv_{1}^{\Pi}}\right\|_{2} - \left|\frac{\lambda_{1}(W)n}{\lambda_{1}(\bW_{n})}\langle\phi_{1},f_{\bv_{1}^{\Pi}}\rangle\right|\right| = \left|1 - \left|\frac{\lambda_{1}(W)n}{\lambda_{1}(\bW_{n})}\langle\phi_{1},f_{\bv_{1}^{\Pi}}\rangle\right|\right|\lesssim_{W} \frac{\log n}{\sqrt{n}}
 \end{align*}
 with probability at least $1-32n\exp\left(-\frac{1}{6}(\log n)^2\right)$. By Lemma $5.3$ and the Cauchy-Schwarz inequality note that,
 \begin{align*}
     \left|1 - \left|\langle\phi_{1},f_{\bv_{1}^{\Pi}}\rangle\right|\right|
     &\leq \left|1 - \left|\frac{\lambda_{1}(W)n}{\lambda_{1}(\bW_{n})}\langle\phi_{1},f_{\bv_{1}^{\Pi}}\rangle\right|\right| + \left|\langle\phi_{1},f_{\bv_{1}^{\Pi}}\rangle\right|\left|1 - \left|\frac{\lambda_{1}(W)n}{\lambda_{1}(\bW_{n})}\right|\right|\\
     &\lesssim_{W}\frac{\log n}{\sqrt{n}}
 \end{align*}
 with probability at least $1-40n\exp\left(-\frac{1}{6}(\log n)^2\right)$. Now note that,
 \begin{align*}
     \|f_{\widetilde{\bv}_{1}^{\Pi}} - \phi_{1}\|_{2}^2 = 2 - 2\left|\langle\phi_{1},f_{\bv_{1}^{\Pi}}\rangle\right|\lesssim_{W}\frac{\log n}{\sqrt{n}}
 \end{align*}
 with probability at least $1-40n\exp\left(-\frac{1}{6}(\log n)^2\right)$.
 \subsection{Proof of Lemma $7.2$}
 Define,
 \begin{align*}
     \bC_{n,1} = \lambda_{1}(\bW_{n})\bI_{n} - \bW_{n} + \lambda_{1}(\bW_{n})\bv_{1}\bv_{1}^{\top}.
 \end{align*}
 Then by Weyl's inequality note that for any $1\leq i\leq n$,
 \begin{align*}
     \min_{i=1}^{n}\lambda_{i}(\bC_{n,1})\geq \min_{i=1}^{n}\lambda_{i}(\bC_{n}) - \|\bA_{n} - \bW_{n}\|_{2\ra 2}
 \end{align*}
 Then combining Lemma $4.4$ and $(4.29)$ we conclude that $\|\bC_{n,1}^{-1}\|_{2\ra 2}\lesssim_{W}1/n$ with probability at least $1-Cn\exp\left(-\frac{1}{6}(\log n)^2\right)$. Using the identity $\bC_{n,1}^{-1} - \bC_{n}^{-1} = \bC_{n}^{-1}(\bC_{n} - \bC_{n,1})\bC_{n,1}^{-1}$ we conclude that $\|\bC_{n}^{-1} - \bC_{n,1}^{-1}\|_{2\ra 2}\lesssim_{W}n^{-3/2}$ with probability at least $1-Cn\exp\left(-\frac{1}{6}(\log n)^2\right)$. Now define,
 \begin{align*}
     \bC_{n,2} = \lambda_{1}(\bW_{n})\bI_{n} - \bW_{n} + \lambda_{1}(\bW_{n})\bPhi_{1}\bPhi_{1}^{\top}
 \end{align*}
 By $(7.8)$ and Lemma $5.1$ we get,
 \begin{align}\label{eq:phivouterprodbdd}
     \left\|\widetilde{\bv}_{1}^{\Pi}\left(\widetilde{\bv}_{1}^{\Pi}\right)^{\top} - \bPhi_{1}^{\Pi}\left(\bPhi_{1}^{\Pi}\right)^{\top}\right\|\lesssim_{W}\left(\frac{\log n}{\sqrt{n}}\right)^{1/2}
 \end{align}
 with probability at least $1-Cn\exp\left(-\frac{1}{6}(\log n)^2\right)$.
 \small
 \begin{align*}
     \min_{i=1}^{n}\lambda_{i}(\bC_{n,2})
     &\geq \min_{i=1}^{n}\lambda_{i}(\bC_{n,1}) - |\lambda_{1}(\bW_{n})|\left\|\bv_{1}\bv_{1}^{\top} - \bPhi_{1}\bPhi_{1}^{\top}\right\|_{2\ra 2}\\
     & = \min_{i=1}^{n}\lambda_{i}(\bC_{n,1}) - |\lambda_{1}(\bW_{n})|\left\|\widetilde{\bv}_{1}^{\Pi}\left(\widetilde{\bv}_{1}^{\Pi}\right)^{\top} - \bPhi_{1}^{\Pi}\left(\bPhi_{1}^{\Pi}\right)^{\top}\right\|_{2\ra 2}\\
     &\geq \min_{i=1}^{n}\lambda_{i}(\bC_{n}) - \|\bA_{n} - \bW_{n}\|_{2\ra 2} - |\lambda_{1}(\bW_{n})|\left\|\widetilde{\bv}_{1}^{\Pi}\left(\widetilde{\bv}_{1}^{\Pi}\right)^{\top} - \bPhi_{1}^{\Pi}\left(\bPhi_{1}^{\Pi}\right)^{\top}\right\|_{2\ra 2}
 \end{align*}
 \normalsize
 Now combining Lemma $4.4$, Lemma $5.3$ along with \eqref{eq:phivouterprodbdd} and $(4.29)$ we conclude $\|\bC_{n,2}^{-1}\|_{2\ra 2}\lesssim_{W}1/n$ with probability at least $1-Cn\exp\left(-\frac{1}{6}(\log n)^2\right)$. Notice that by \eqref{eq:phivouterprodbdd} and recalling the bound from $(6.6)$ we get,
 \begin{align*}
     \left\|\bC_{n,1} - \bC_{n,2}\right\|_{2\ra 2} = |\lambda_{1}(\bW_{n})|\left\|\widetilde{\bv}_{1}^{\Pi}\left(\widetilde{\bv}_{1}^{\Pi}\right)^{\top} - \bPhi_{1}^{\Pi}\left(\bPhi_{1}^{\Pi}\right)^{\top}\right\|\lesssim_{W}n^{3/4}\sqrt{\log n}
 \end{align*}
 with probability at least $1-Cn\exp\left(-\frac{1}{6}(\log n)^2\right)$. Then once again considering the identity $\bC_{n,2}^{-1} - \bC_{n,1}^{-1} = \bC_{n,1}^{-1}(\bC_{n,1} - \bC_{n,2})\bC_{n,2}^{-1}$ we conclude that $\|\bC_{n,1}^{-1} - \bC_{n,2}^{-1}\|_{2\ra 2}\lesssim_{W}n^{-5/4}\sqrt{\log n}$ with probability at least $1-Cn\exp\left(-\frac{1}{6}(\log n)^2\right)$. Then,
 \begin{align*}
     \left|\bv_{1}^{\top}\bmB_{n}\bC_{n}^{-1}\bmB_{n}\bv_{1} - \bv_{1}^{\top}\bmB_{n}\bC_{n,2}^{-1}\bmB_{n}\bv_{1}\right|\leq \|\bmB_{n}\|_{2\ra 2}^2\left\|\bC_{n} - \bC_{n,2}\right\|_{2\ra 2}\lesssim\left(\frac{\log n}{\sqrt{n}}\right)^{1/2}
 \end{align*}
 with probability at least $1-Cn\exp\left(-\frac{1}{6}(\log n)^2\right)$, where the last inequality follows by combining the bounds on $\left\|\bC_{n} - \bC_{n,1}\right\|_{2\ra 2}$ and $\left\|\bC_{n,1} - \bC_{n,2}\right\|_{2\ra 2}$ and bounds on $\|\bmB_{n}\|_{2\ra 2}$ from $(4.29)$.
 Observe that,
 \begin{align*}
     \bv_{1}^{\top}\bmB_{n}\bC_{n,2}^{-1}\bmB_{n}\bv_{1} = \left(\widetilde{\bv}_{1}^{\Pi}\right)^{\top}\bmB_{n}^{\Pi}\left(\bC_{n,2}^{\Pi}\right)^{-1}\bmB_{n}^{\Pi}\widetilde{\bv}_{1}^{\Pi}.
 \end{align*}
 Recalling the bound from $(7.8)$ we can equivalently write,
 \begin{align*}
     \left|\bv_{1}^{\top}\bmB_{n}\bC_{n}^{-1}\bmB_{n}\bv_{1} - \left(\bPhi_{1}^{\Pi}\right)^{\top}\bmB_{n}^{\Pi}\left(\bC_{n,2}^{\Pi}\right)^{-1}\bmB_{n}^{\Pi}\bPhi_{1}^{\Pi}\right|\lesssim\left(\frac{\log n}{\sqrt{n}}\right)^{1/2}
 \end{align*}
 with probability at least $1-Cn\exp\left(-\frac{1}{6}(\log n)^2\right)$, which completes the proof. 
 
 \subsection{Proof of Lemma $7.3$} 
 By definition note that,
 \footnotesize
 \begin{align*}
     \E\left[\bZ_{1}^{\top}\left(\bcC_{n}^{\Pi}\right)^{-1}\bZ_{1}|\bU_{n}\right] = \frac{1}{n}\sum_{i}\left(\bcC_{n}^{\Pi}\right)^{-1}[i,i]\sum_{j<i}\phi_{1}^{2}(U_{(j)})W(U_{(i)},U_{(j)})(1-W(U_{(i)},U_{(j)})).
 \end{align*}
 \normalsize
 and,
 \small
 \begin{align*}
     \E\left[\bZ_{2}^{\top}\left(\bcC_{n}^{\Pi}\right)^{-1}\bZ_{2}|\bU_{n}\right] = \frac{1}{n}\sum_{i}\left(\bcC_{n}^{\Pi}\right)^{-1}[i,i]\sum_{j>i}\phi_{1}^{2}(U_{(j)})W(U_{(i)},U_{(j)})(1-W(U_{(i)},U_{(j)})).
 \end{align*}
 \normalsize
 Notice,
 \begin{align}\label{eq:Cdiagapproxlambda1}
     \frac{1}{n}\sum_{i}\sum_{j<i}\left|\left(\bcC_{n}^{\Pi}\right)^{-1}[i,i] - \frac{1}{\lambda_{1}(\bW_{n})}\right|
     &\left|\phi_{1}^{2}(U_{(j)})W(U_{(i)},U_{(j)})(1-W(U_{(i)},U_{(j)}))\right|.\\
     &\lesssim_{W}\sum_{i}\left|\left(\bcC_{n}^{\Pi}\right)^{-1}[i,i] - \frac{1}{\lambda_{1}(\bW_{n})}\right|\nonumber
 \end{align}
 Recall that for two matrices $\bm S_{1}$ and $\bm S_{2}$, $\bm{S}_{1}^{-1} - \bm{S}_{2}^{-1} = \bm{S}_{2}^{-1}(\bm S_{2} - \bm S_{1})\bm{S}_{1}^{-1}.$ Invoking this identity we get,
 \begin{align*}
     \left(\bcC_{n}^{\Pi}\right)^{-1} - \frac{1}{\lambda_{1}(\bW_{n})}\bI = \frac{1}{\lambda_{1}(\bW_{n})}\left(\bW_{n}^{\Pi} - \lambda_{1}(\bW_{n})\bPhi_{1}^{\Pi}\left(\bPhi_{1}^{\Pi}\right)^{\top}\right)\left(\bcC_{n}^{\Pi}\right)^{-1}
 \end{align*}
 and hence,
 \footnotesize
 \begin{align*}
     \sum_{i}\left|\left(\bcC_{n}^{\Pi}\right)^{-1}[i,i] - \frac{1}{\lambda_{1}(\bW_{n})}\right| 
     & = \frac{1}{|\lambda_{1}(\bW_{n})|}\sum_{i}\left|\left[\left(\bW_{n}^{\Pi} - \lambda_{1}(\bW_{n})\bPhi_{1}^{\Pi}\left(\bPhi_{1}^{\Pi}\right)^{\top}\right)\left(\bcC_{n}^{\Pi}\right)^{-1}\right](i,i)\right|\\
     & \leq \frac{1}{|\lambda_{1}(\bW_{n})|}\sum_{i,j}\left|\left[\bW_{n}^{\Pi} - \lambda_{1}(\bW_{n})\bPhi_{1}^{\Pi}\left(\bPhi_{1}^{\Pi}\right)^{\top}\right](i,j)\left(\bcC_{n}^{\Pi}\right)^{-1}(j,i)\right|\\
     &\lesssim_{W}\frac{1}{n}\sum_{i,j}\left|\left(\bcC_{n}^{\Pi}\right)^{-1}(j,i)\right|
 \end{align*}
 \normalsize
 with probability at least $1-Cn\exp\left(-\frac{1}{6}(\log n)^2\right)$, where the last inequality follows from the definition of $\bW_{n}^{\Pi}$, $\bPhi_{1}^{\Pi}$ and the bounds from Lemma $5.3$. Finally recalling the bound from $(7.14)$ shows,
 \begin{align*}
     \sum_{i}\left|\left(\bcC_{n}^{\Pi}\right)^{-1}[i,i] - \frac{1}{\lambda_{1}(\bW_{n})}\right| \lesssim_{W}\frac{1}{\sqrt{n}}
 \end{align*}
 with probability at least $1-Cn\exp\left(-\frac{1}{6}(\log n)^2\right)$. Now recalling \eqref{eq:Cdiagapproxlambda1} shows,
 \begin{align*}
     \frac{1}{n}\sum_{i}\sum_{j<i}\left|\left(\bcC_{n}^{\Pi}\right)^{-1}[i,i] - \frac{1}{\lambda_{1}(\bW_{n})}\right|
     &\left|\phi_{1}^{2}(U_{(j)})W(U_{(i)},U_{(j)})(1-W(U_{(i)},U_{(j)}))\right|\lesssim_{W}\frac{1}{\sqrt{n}}
 \end{align*}
 with probability at least $1-Cn\exp\left(-\frac{1}{6}(\log n)^2\right)$. Similarly,
 \begin{align*}
     \frac{1}{n}\sum_{i}\sum_{j>i}\left|\left(\bcC_{n}^{\Pi}\right)^{-1}[i,i] - \frac{1}{\lambda_{1}(\bW_{n})}\right|
     &\left|\phi_{1}^{2}(U_{(j)})W(U_{(i)},U_{(j)})(1-W(U_{(i)},U_{(j)}))\right|\lesssim_{W}\frac{1}{\sqrt{n}}
 \end{align*}
 with probability at least $1-Cn\exp\left(-\frac{1}{6}(\log n)^2\right)$. Combining,
 \begin{align*}
     \left|T(\bU_{n}) - \frac{1}{n\lambda_{1}(\bW_{n})}\sum_{i}\sum_{j\neq i}\phi_{1}^{2}(U_{(j)})W(U_{(i)},U_{(j)})(1-W(U_{(i)},U_{(j)}))\right|\lesssim_{W}\frac{1}{\sqrt{n}}
 \end{align*}
 Recall the Lipschitz property of $W$ and $\phi_{1}$ as well as the bounds from Lemma $5.1$ and Lemma $5.3$. Then using the concentration from Lemma \ref{lemma:concorderU} shows,
 \begin{align*}
     \left|T(\bU_{n}) - \frac{1}{n^2\lambda_{1}(W)}\sum_{i}\sum_{j\neq i}\phi_{1}^{2}\left(\frac{j}{n}\right)W\left(\frac{i}{n},\frac{j}{n}\right)\left(1-W\left(\frac{i}{n},\frac{j}{n}\right)\right)\right|\lesssim_{W}\frac{\log n}{\sqrt{n}}
 \end{align*}
 with probability at least $1-Cn\exp\left(-\frac{1}{6}(\log n)^2\right)$. Finally recalling that $W$ is symmetric we get,
 \begin{align*}
     \left|T(\bU_{n}) - \frac{1}{n^2\lambda_{1}(W)}\sum_{i,j}\frac{\phi_{1}^{2}\left(\frac{i}{n}\right) + \phi_{1}^{2}\left(\frac{j}{n}\right)}{2}W\left(\frac{i}{n},\frac{j}{n}\right)\left(1-W\left(\frac{i}{n},\frac{j}{n}\right)\right)\right|\lesssim_{W}\frac{\log n}{\sqrt{n}}
 \end{align*}
 with probability at least $1-Cn\exp\left(-\frac{1}{6}(\log n)^2\right)$. The proof is now completed by a \blue{Riemann} sum approximation argument.

 \section{Spectrum of Self-Adjoint Compact Operators}\label{s:operator}
 In this section we collect various useful results about the spectrum of compact self-adjoint operators on a Hilbert space $\cH$. We start this section with a self-contained proof of the min-max theorem for operators showing equivalence between the non-negative eigenvalues and the \blue{Rayleigh} Quotient of an operator $T$. 
     \begin{theorem}\label{t:min-max}
         Given a self-adjoint compact operator $T$ on a Hilbert space $\cH$. We enumerate positive eigenvalues of $T$ as (if $T$ only have $\ell$ positive eigenvalues, we make the convention that $\lambda_k(T)=0$ for $k\geq \ell+1$) 
         \begin{align*}
         \lambda_1(T)\geq \lambda_2(T)\geq \lambda_3(T)\geq \cdots,\end{align*}
         Then the following Min-Max statement holds
         \begin{align}\label{e:minmax}
         \lambda_k(T)=\sup_{S_k}\min_{x\in S_k, \|x\|=1}\langle x, Tx\rangle,
         \end{align}
         where $S_k\subset \cH$ is a $k$-dimensional subspace.
         \end{theorem}
         \begin{proof}
         If $\lambda_k(T)>0$, the above statement follows from the standard Min-Max theorem. If $T$ has only $\ell$ positive eigenvalues with $\ell<k$, by our convention we have $\lambda_k(T)=0$. We denote the eigenvectors corresponding to $\lambda_1(T), \cdots, \lambda_\ell(T)$ as $u_1, u_2,\cdots, u_\ell$. Then $T-\sum_{i=1}^\ell \lambda_i(T)u_iu_i^*$ is a non-positive semi-definite operator, i.e. for any $x\in \cH$,
         \begin{align}\label{e:negative}
         \langle x, (T-\sum_{i=1}^\ell \lambda_i(T)u_iu_i^*)x\rangle \leq 0.
         \end{align}
         For any $k$-dimensional subspace $S_k\subset \cH$, there exists a $v\in S_k$ such that $\langle v, u_i\rangle=0$ for $1\leq i\leq \ell$ (here we used that $\ell<k$). Then using \eqref{e:negative}
         \begin{align*}
         \min_{x\in S_k, \|x\|=1}\langle x, Tx\rangle\leq \langle v, Tv\rangle =\langle v, (T-\sum_{i=1}^\ell \lambda_i(T)u_iu_i^*)v\rangle \leq 0=\lambda_k(T).
         \end{align*}
         We conclude that
         \begin{align} \label{e:oneside}
         0=\lambda_k(T)\geq \sup_{S_k}\min_{x\in S_k, \|x\|=1}\langle x, Tx\rangle.
         \end{align}
         
         Since $0$ is the only possible cluster point of the eigenvalues of $T$ and $T$ only has $\ell$ positive eigenvalues, for any $\delta>0$, we can find $k-\ell$ non-positive eigenvalues of $T$ such that
         \begin{align*}
         -\delta\leq \widetilde \lambda_{1}(T), \widetilde\lambda_{2}(T),\cdots, \widetilde\lambda_{k-\ell}(T)\leq 0.
         \end{align*}
         We denote their corresponding eigenvectors as $\widetilde u_1, \widetilde u_2,\cdots, \widetilde u_{k-\ell}$. If we take the $k$-dimensional space $S_k={\rm Span}(u_1, u_2,\cdots, u_{\ell}, \widetilde u_1, \widetilde u_2,\cdots, \widetilde u_{k-\ell})$, then
         \begin{align*} 
         \min_{x\in S_k, \|x\|=1}\langle x, Tx\rangle\geq -\delta.
         \end{align*}
         Since we can take $\delta>0$ arbitrarily small, we conclude that 
         \begin{align}\label{e:twoside}
         \sup_{S_k}\min_{x\in S_k, \|x\|=1}\langle x, Tx\rangle\geq 0=\lambda_k(T).
         \end{align}
         The estimate \eqref{e:oneside} and \eqref{e:twoside} together give \eqref{e:minmax}. 
         \end{proof}
         
 Next, we study the difference between corresponding eigenvalues of two compact self-adjoint operators $T_{1}$ and $T_{2}$. In particular, we echo and extend results from matrix theory showing that corresponding eigenvalues of operators must be close if the operators are close in appropriate norm.
         \begin{lemma}\label{l:eigdist}
         Fix small $\varepsilon>0$. Given two self-adjoint compact operators $T_1, T_2$ on a Hilbert space $\cH$, such that $\|T_1-T_2\|_{\cH\rightarrow \cH}\leq \varepsilon$, then the following holds. If we enumerate positive eigenvalues of $T_1, T_2$ as (if $T_i$ only have $\ell$ positive eigenvalues, we make the convention that $\lambda_k(T_i)=0$ for $k\geq \ell+1$) 
         \begin{align*}
         \lambda_1(T_1)\geq \lambda_2(T_1)\geq \lambda_3(T_1)\geq \cdots,\\
         \lambda_1(T_2)\geq \lambda_2(T_2)\geq \lambda_3(T_2)\geq \cdots,
         \end{align*}
         then for any $k\geq 1$,
         \begin{align}\label{e:eigdif1}
         |\lambda_k(T_1)-\lambda_k(T_2)|\leq \varepsilon.
         \end{align}
         The same statement holds for negative eigenvalues. We enumerate negative eigenvalues of $T_1, T_2$ as (if $T_i$ only have $\ell$ negative eigenvalues, we make the convention that $\lambda_k'(T_i)=0$ for $k\geq \ell+1$) 
         \begin{align*}
         \lambda'_1(T_1)\leq \lambda'_2(T_1)\leq \lambda'_3(T_1)\leq \cdots,\\
         \lambda'_1(T_2)\leq \lambda'_2(T_2)\leq \lambda'_3(T_2)\leq \cdots,
         \end{align*}
         then for any $k\geq 1$,
         \begin{align}\label{e:eigdif2}
         |\lambda'_k(T_1)-\lambda_k'(T_2)|\leq \varepsilon.
         \end{align}
         \end{lemma}
         
         \begin{proof}
         We will only prove \eqref{e:eigdif1}, the proof of \eqref{e:eigdif2} follows from considering $-T_1, -T_2$.  
         By the Min-Max theorem \ref{t:min-max},
         \begin{align}\label{e:minmax2}
         \lambda_k(T_1)=\sup_{S_k}\min_{x\in S_k, \|x\|=1}\langle x, T_1x\rangle,
         \quad
         \lambda_k(T_2)=\sup_{S_k}\min_{x\in S_k, \|x\|=1}\langle x, T_2x\rangle,
         \end{align}
         where $S_k\subset \cH$ is a $k$-dimensional subspace.

         Using the first relation in \eqref{e:minmax2}, for any $\delta>0$, there exists a $k$-dimensional subspace $V_k\subset H$, such that
         \begin{align}\label{e:bb1}
         \lambda_k(T_1)-\delta \leq \min_{x\in V_k}\langle x, T_1x\rangle.
         \end{align}
         By the second relation in \eqref{e:minmax2}, we have 
         \begin{align}\label{e:bb2}
         \lambda_k(T_2)\geq \min_{x\in V_k, \|x\|=1}\langle x, T_2x\rangle=\langle y, T_2 y\rangle,
         \end{align}
         for some $y\in V_k$ with $\|y\|=1$. Then combining \eqref{e:bb1} and \eqref{e:bb2}, and using $\|T_1-T_2\|_{\cH\rightarrow \cH}\leq \varepsilon$, we get\
         cblue
         \begin{align*}
         \lambda_k(T_2)\geq\langle y, T_2 y\rangle
         &=\langle y, T_1 y\rangle+\langle y, (T_2-T_1) y\rangle
         \geq \langle y, T_1 y\rangle-\|T_2-T_1\|_{\cH\rightarrow \cH}\\
         &\geq \min_{x\in V_k, \|x\|=1}\langle x, T_1x\rangle-\varepsilon \geq \lambda_k(T_1)-\delta-\varepsilon.
         \end{align*}
         \cblack
         Since $\delta>0$ can be arbitrarily small, by sending $\delta\rightarrow 0$, we conclude that 
         \begin{align*}
         \lambda_k(T_2)\geq \lambda_k(T_1)-\varepsilon.
         \end{align*}
         Repeating the above argument with $(T_1, T_2)$ replaced by $(T_2, T_1)$, we get that $\lambda_k(T_1)\geq \lambda_k(T_2)-\varepsilon$. Thus the claim \eqref{e:eigdif1} follows.
         
         \end{proof}

     Finally, we provide an immediate corollary of the above lemma which shows that for two close (in the appropriate norm) operators the distance of eigenvalue of one operator to the spectrum of the other is also small. 
         
         \begin{corollary}\label{c:eigtoeig}
         Fix small $\varepsilon>0$. Given two self-adjoint compact operators $T_1, T_2$ on a Hilbert space $\mathcal H$, such that $\|T_1-T_2\|_{\cH\rightarrow \cH}\leq \varepsilon$, then for any eigenvalue $\lambda$ of $T_1$, the following holds
         \begin{align}\label{e:distla}
         {\rm dist}(\lambda, {\rm \sigma}(T_2))\leq \varepsilon.
         \end{align}
         \end{corollary}

         \begin{proof}
         If $\lambda=0$, then \eqref{e:distla} follows from  $0\in \overline{{\rm \sigma}(T_2)}$. Otherwise, by symmetry we assume $\lambda>0$. We enumerate the positive eigenvalues of $T_1, T_2$ as ( if $T_i$ only has 
         $\ell$ positive eigenvalues, we make the convention that $\lambda_k(T_i)=0$ for $k\geq \ell+1$) 
         \begin{align*}
         \lambda_1(T_1)\geq \lambda_2(T_1)\geq \lambda_3(T_1)\geq \cdots,\\
         \lambda_1(T_2)\geq \lambda_2(T_2)\geq \lambda_3(T_2)\geq \cdots,
         \end{align*}
         Since $0$ is the only cluster point of the eigenvalues of $T_1$ and $\lambda>0$, there exists an index $k$ such that $\lambda=\lambda_k(T_1)$, and Lemma \ref{l:eigdist} implies that
         \begin{align*}
         |\lambda_k(T_1)-\lambda_k(T_2)|\leq \varepsilon.
         \end{align*}
         Either $\lambda_k(T_2)>0$, or $\lambda_k(T_2)=0$. In both cases we have $\lambda_k(T_2)\in\overline{{\rm \sigma}(T_2)}$, and it follows that ${\rm dist}(\lambda, {\rm \sigma}(T_2))\leq \varepsilon$. This finishes the proof of Corollary \ref{c:eigtoeig}.
         \end{proof}
 
         \begin{lemma}\label{lemma:normbddresolvent}
             Consider a compact self-adjoint operator $T:L_{2}[0,1]\ra L_{2}[0,1]$. Let $\sigma(T)$ be the spectrum of $T$. Then for $z\not\in \sigma(T)\bigcup \{0\}$,
             \begin{align*}
                 \left\|\blue{(z - T)^{-1}}\right\|_{2\ra 2}\leq \frac{1}{{\rm dist}(z,\sigma(T))}
             \end{align*}
         \end{lemma}
         \begin{proof}
             By the spectral theorem note that,
             \begin{align*}
                 T = \sum_{i\geq 1}\lambda_{i}\phi_{i}\phi_{i}^{\star}
             \end{align*}
             where $|\lambda_{1}|\geq |\lambda_{2}|\geq \cdots$ are eigenvalues of the operator $T$ and $\phi_{i}$ is the eigenfunction corresponding to the eigenvalue $\lambda_{i}$ for all $i\geq 1$. Note that $\{\phi_{i}:i\geq 1\}$ forms an orthonormal collection in $L_{2}[0,1]$. Then for $z\not\in\sigma(T)\bigcup \{0\}$  the resolvent $(z-T)^{-1}$ is well defined and,
             \begin{align}\label{eq:resolventexpand}
                 (z-T)^{-1} = \sum_{i\geq 1}(z-\lambda_{i})^{-1}\phi_{i}\phi_{i}^{\star}.
             \end{align}
             Note that for any $v\in L_{2}[0,1]$,
             \begin{align*}
                 \left\|(z-T)^{-1}v\right\|_{2} \leq \sqrt{\sum_{i\geq 1}\frac{1}{|z-\lambda_{i}|^2}|\langle\phi_{i},v\rangle|^2}\leq \sqrt{\frac{\sum_{i\geq 1}|\langle\phi_{i},v\rangle|^2}{\rm dist(z,\sigma(T))^2}} = \frac{\|v\|}{\rm dist(z,\sigma(T))}.
             \end{align*}
             The proof is now completed by recalling the definition of operator norm.
         \end{proof}
 
\end{document}